\documentclass[10pt,a4paper,pointlessnumbers]{scrartcl}
\usepackage[T1]{fontenc}

\usepackage{amssymb}
\usepackage{amsmath}
\usepackage{amsfonts}
\usepackage{bbm}
\usepackage{amsthm}
\usepackage{lmodern}
\usepackage{mathrsfs}
\usepackage[hidelinks]{hyperref}\usepackage{color}
\usepackage[margin=2.3cm]{geometry}
\usepackage[all,cmtip]{xy}
\usepackage[utf8]{inputenc}
\usepackage{graphicx}

\usepackage{varwidth}
\usepackage{overpic}

\makeatletter
\DeclareRobustCommand{\em}{%
	\@nomath\em \if b\expandafter\@car\f@series\@nil
	\normalfont \else \slshape \fi}
\makeatother
\usepackage{upgreek}	
\usepackage{rotating}
\usepackage{tikz}
\usetikzlibrary{matrix,arrows,decorations.pathmorphing,shapes.geometric}
\usepackage{tikz-cd}
\usepackage{needspace}
\newcommand{\spaceplease}{\needspace{10\baselineskip}}
\usetikzlibrary{decorations.markings}

\usetikzlibrary{backgrounds}

\newcommand{\SkAlg}{\catf{SkAlg}}
\usepackage{todonotes}

\newcommand{\E}{\mathbb{E}}

\newcommand{\HOM}{\underline{\catf{Hom}}} \newcommand{\END}{\underline{\catf{\End}}}
\newcommand{\ot}{\boxtimes}

\usepackage{tikz-cd}

\usetikzlibrary{decorations.markings}
\usetikzlibrary{backgrounds}

\usepackage{etex}

\tikzstyle{tikzfig}=[baseline=-0.25em,scale=0.5]

\pgfkeys{/tikz/tikzit fill/.initial=0}
\pgfkeys{/tikz/tikzit draw/.initial=0}
\pgfkeys{/tikz/tikzit shape/.initial=0}
\pgfkeys{/tikz/tikzit category/.initial=0}

\pgfdeclarelayer{edgelayer}
\pgfdeclarelayer{nodelayer}
\pgfsetlayers{background,edgelayer,nodelayer,main}

\tikzstyle{none}=[inner sep=0mm]

\newcommand{\tikzfig}[1]{%
	{\tikzstyle{every picture}=[tikzfig]
		\IfFileExists{#1.tikz}
		{\input{#1.tikz}}
		{%
			\IfFileExists{./figures/#1.tikz}
			{\input{./figures/#1.tikz}}
			{\tikz[baseline=-0.5em]{\node[draw=red,font=\color{red},fill=red!10!white] {\textit{#1}};}}%
	}}%
}

\tikzstyle{every loop}=[]

\usepackage{tikzit}

\tikzstyle{black dot}=[fill=black, draw=black, shape=circle, minimum size=3pt, inner sep=0pt]
\tikzstyle{black dot small}=[fill=black, draw=black, shape=circle, minimum size=2pt, inner sep=0pt]
\tikzstyle{fblack dot}=[fill=black, draw=red, shape=circle, minimum size=2pt, inner sep=0pt]
\tikzstyle{wbox}=[fill=white, draw=black, shape=rectangle, minimum height=0.5cm, minimum width=0.01cm]
\tikzstyle{bbox}=[fill=white, draw=blue, shape=rectangle, minimum height=0.5cm, minimum width=0.01cm]
\tikzstyle{rbox}=[fill=white, draw=red, shape=rectangle, minimum height=0.5cm, minimum width=0.01cm]
\tikzstyle{bwbox}=[draw=blue, shape=rectangle, minimum width=2cm, minimum height=0.5cm]
\tikzstyle{bbwbox}=[draw=blue, shape=rectangle, minimum width=1cm, minimum height=1cm]
\tikzstyle{big white circle}=[fill=white, draw=black, shape=circle, minimum width=0.75cm]
\tikzstyle{white dot big}=[fill=white, draw=black, shape=circle, inner sep=1pt]
\tikzstyle{white dot}=[fill=white, draw=black, shape=circle, minimum size=3pt, inner sep=0pt]
\tikzstyle{flat box}=[fill=white, draw=black, shape=rectangle, minimum width=1.3cm, minimum height=0.5cm,fill=morphismcolor]
\tikzstyle{square}=[fill=white, draw=black, shape=rectangle]
\tikzstyle{flat box 2}=[fill=white, draw=black, shape=rectangle, minimum height=0.5cm, minimum width=0.01cm,fill=morphismcolor]
\tikzstyle{bigbox}=[fill=white, draw=black, shape=rectangle, minimum height=0.5cm, minimum width=0.8cm,fill=white]
\tikzstyle{over }=[front]
\tikzstyle{theta}=[fill=blue, draw=blue, shape=ellipse, minimum height=6pt, minimum width=6pt, inner sep=0pt]
\tikzstyle{thetabig}=[fill=blue, draw=blue, shape=ellipse, minimum width=1cm, minimum height=0.01cm]
\tikzstyle{thetainv}=[fill=blue, draw=red, shape=ellipse, minimum height=6pt, minimum width=6pt, inner sep=0pt]
\tikzstyle{thetabinv}=[fill=blue, draw=red, shape=ellipse, minimum width=1cm, minimum height=0.01cm]
\tikzstyle{bigdisk}=[draw=black, shape=circle, minimum width=3cm]
\tikzstyle{wdisk}=[shape=circle, minimum width=0.48cm,fill=white]
\tikzstyle{bigdisk2}=[draw=black, fill=lightgray, shape=circle, minimum width=3cm]
\tikzstyle{little disk}=[fill=white, draw=black, shape=circle, minimum width=0.5cm]

\tikzstyle{mid arrow}=[-, postaction={on each segment={mid arrow}}]
\tikzstyle{end arrow}=[->]
\tikzstyle{mover}=[-, link]
\tikzstyle{string}=[-, draw=blue,postaction={on each segment={mid arrow}}]
\tikzstyle{stringd}=[-, dotted,draw=blue,postaction={on each segment={mid arrow}}]
\tikzstyle{red}=[-, dotted,draw=red]
\tikzstyle{mydots}=[-,dotted,dashed,draw=gray]
\tikzstyle{mydotsblack}=[-,dotted,dashed,draw=black]
\tikzstyle{open}=[-, line width=1pt,draw=blue]
\tikzstyle{thick}=[-,line width=1pt]
\tikzstyle{rarrow}=[->,draw=red]
\tikzstyle{red mid arrow}=[-, draw={rgb,255: red,214; green,42; blue,51}, postaction={on each segment={mid arrow}}, line width=1pt]
\tikzstyle{RED}=[-, draw={rgb,255: red,214; green,42; blue,51}]
\tikzstyle{REDdashed}=[-,dashed, draw={rgb,255: red,214; green,42; blue,51}]
\tikzstyle{REDarrow}=[->, draw={rgb,255: red,214; green,42; blue,51}]
\tikzstyle{darrow}=[->,dotted]
\tikzstyle{blue}=[-, draw=blue]
\tikzstyle{blue mid arrow}=[-, draw={rgb,255: red,23; green,37; blue,167}, postaction={on each segment={mid arrow}}, line width=1pt]
\tikzstyle{over}=[-, link]
\tikzstyle{bover}=[-, blink]
\tikzstyle{mover}=[-, link]
\tikzstyle{mapsto}=[{|->}]

\usetikzlibrary{decorations.pathreplacing,decorations.markings}
\tikzset{
	on each segment/.style={
		decorate,
		decoration={
			show path construction,
			moveto code={},
			lineto code={
				\path [#1]
				(\tikzinputsegmentfirst) -- (\tikzinputsegmentlast);
			},
			curveto code={
				\path [#1] (\tikzinputsegmentfirst)
				.. controls
				(\tikzinputsegmentsupporta) and (\tikzinputsegmentsupportb)
				..
				(\tikzinputsegmentlast);
			},
			closepath code={
				\path [#1]
				(\tikzinputsegmentfirst) -- (\tikzinputsegmentlast);
			},
		},
	},
	mid arrow/.style={postaction={decorate,decoration={
				markings,
				mark=at position .7 with {\arrow[#1]{stealth}}
	}}},
}
\tikzset{%
	link/.style    = { white, double = black, line width = 1.8pt,
		double distance = 0.4pt },
	channel/.style = { white, double = black, line width = 0.8pt,
		double distance = 0.8pt },
}

\tikzset{%
	blink/.style    = { white, double = blue, line width = 2pt,
		double distance = 1pt },
	channel/.style = { white, double = blue, line width = 2pt,
		double distance = 1pt },
}

\tikzstyle{tikzfig}=[baseline=-0.25em,scale=0.5]

\pgfkeys{/tikz/tikzit fill/.initial=0}
\pgfkeys{/tikz/tikzit draw/.initial=0}
\pgfkeys{/tikz/tikzit shape/.initial=0}
\pgfkeys{/tikz/tikzit category/.initial=0}

\pgfdeclarelayer{edgelayer}
\pgfdeclarelayer{nodelayer}
\pgfsetlayers{background,edgelayer,nodelayer,main}

\tikzstyle{none}=[inner sep=0mm]

\tikzstyle{every loop}=[]

\usepackage{mathtools}
\mathtoolsset{showonlyrefs}
\usepackage{epstopdf}

\newtheoremstyle{mytheorem}
{\topsep}
{\topsep}
{\slshape}
{0pt}
{\bfseries}
{.}
{ }
{\thmname{#1}\thmnumber{ #2}\thmnote{ {\normalfont\slshape(#3)}}}

\newtheoremstyle{mydefinition}
{\topsep}
{\topsep}
{\normalfont}
{0pt}
{\bfseries}
{.}
{ }
{\thmname{#1}\thmnumber{ #2}\thmnote{ {\normalfont\slshape(#3)}}}

\theoremstyle{mytheorem}
\newtheorem{theorem}{Theorem}[section]

\makeatletter
\newtheorem*{rep@theorem}{\rep@title}
\newcommand{\newreptheorem}[2]{%
	\newenvironment{rep#1}[1]{%
		\def\rep@title{#2 \ref{##1}}%
		\begin{rep@theorem}}%
		{\end{rep@theorem}}}
\makeatother

\newreptheorem{theorem}{Theorem}
\newreptheorem{corollary}{Corollary}
\newtheorem{lemma}[theorem]{Lemma}
\newtheorem{proposition}[theorem]{Proposition}
\newtheorem{corollary}[theorem]{Corollary}
\theoremstyle{mydefinition}
\newtheorem{definition}[theorem]{Definition}
\newtheorem{exx}[theorem]{Example}
\newenvironment{example}
{\pushQED{\qed}\exx}
{\popQED\endexx}
\newtheorem{remm}[theorem]{Remark}
\newenvironment{remark}
{\pushQED{\qed}\remm}
{\popQED\endremm}
\numberwithin{equation}{section}
\usepackage{enumitem}

\newenvironment{pnum}{\begin{enumerate}[topsep=2pt,parsep=2pt,partopsep=2pt,itemsep=0pt,label={(\roman{*})}]}{\end{enumerate}}

\DeclareMathSymbol{\Phiit}{\mathalpha}{letters}{"08}\let\Phi\undefined\newcommand{\Phi}{\Phiit}
\DeclareMathSymbol{\Psiit}{\mathalpha}{letters}{"09}\let\Psi\undefined\newcommand{\Psi}{\Psiit}
\DeclareMathSymbol{\Sigmait}{\mathalpha}{letters}{"06}\let\Sigma\undefined\newcommand{\Sigma}{\Sigmait}
\DeclareMathSymbol{\Xiit}{\mathalpha}{letters}{"04}
\DeclareSymbolFont{extraup}{U}{zavm}{m}{n}
\DeclareMathSymbol{\vardiamond}{\mathalpha}{extraup}{87}
\DeclareMathSymbol{\Lambdait}{\mathalpha}{letters}{"03}\let\Lambda\undefined\newcommand{\Lambda}{\Lambdait}
\DeclareMathSymbol{\Piit}{\mathalpha}{letters}{"05}\let\Pi\undefined\newcommand{\Pi}{\Piit}
\DeclareMathSymbol{\Gammait}{\mathalpha}{letters}{"00}\let\Gamma\undefined\newcommand{\Gamma}{\Gammait}
\DeclareMathSymbol{\Omegait}{\mathalpha}{letters}{"0A}\let\Omega\undefined\newcommand{\Omega}{\Omegait}
\DeclareMathSymbol{\Upsilonit}{\mathalpha}{letters}{"07}\let\Upsilon\undefined\newcommand{\Upsilon}{\Upilonit}
\DeclareMathSymbol{\Thetait}{\mathalpha}{letters}{"02}\let\Theta\undefined\newcommand{\Theta}{\Thetait}

\def\Hom{\catf{Hom}}

\def\End{\catf{End}}
\def\id{\mathrm{id}}

\def\dim{\mathrm{dim}}
\let\to\undefined\newcommand{\to}{\longrightarrow}
\let\mapsto\undefined\newcommand{\mapsto}{\longmapsto}
\newcommand{\catf}[1]{\mathsf{#1}}
\newcommand{\Proj}{\operatorname{\catf{Proj}}}\newcommand{\Inj}{\operatorname{\catf{Inj}}}

\newcommand{\Map}{\catf{Map}}

\newcommand{\dg}{^{\normalfont\textbf{\dag}}}
\newcommand{\dgr}{\normalfont\textbf{\dag}}

\def\op{\mathrm{op}}

\newcommand{\qss}{\cat{O}_\Sigma^\cat{A}}
\newcommand{\qssp}{\cat{O}_{\Sigma'}^\cat{A}}

\newcommand{\framed}{\catf{f}E_2}

\newcommand{\Nat}{\catf{Nat}}

\newcommand{\Rexfsi}{\catf{Rex}^\catf{f}_\catf{si}}

\newcommand{\ra}[1]{\xrightarrow{\   #1    \ }}

\def\Grpd{\catf{Grpd}}

\newcommand{\Lexf}{\catf{Lex}^\catf{f}}
\newcommand{\Rexf}{\catf{Rex}^\catf{f}}
\newcommand{\Rex}{\catf{Rex}}

\def\Cat{\catf{Cat}}

\newcommand{\trace}{\catf{t}}

\newcommand{\moduli}{\mathfrak{a}}

\newcommand{\act}{\triangleright}

\newcommand{\FA}{\mathfrak{F}_{\! \cat{A}}}

\newcommand{\vect}{\catf{vect}}

\newcommand{\refl}{\catf{R}}

\newcommand{\PhibarA}{\Phi_{\! \bar{\cat{A}}}}

\newcommand{\cat}[1]{\mathcal{#1}}

\newcommand{\Hbdy}{\catf{Hbdy}}

\newcommand{\Surf}{\catf{Surf}}

\newcommand{\SkA}{\catf{Sk}_{\! \cat{A}}}

\newcommand{\Graphs}{\catf{Graphs}}

\newcommand{\Legs}{\catf{Legs}}
\newcommand{\Add}{\catf{Add}}

\newcommand{\RForests}{\catf{RForests}}
\newcommand{\Forests}{\catf{Forests}}

\newcommand{\skcatA}{\catf{skcat}_\cat{A}}

\newcommand{\PhiAdg}{\Phi_{\! \cat{A}\dg}}

\newcommand{\SurfA}{\Surf_{\! \cat{A}}}
\newcommand{\ma}{\moduli_\Sigma}

\newcommand{\PhiA}{\Phi_{\! \cat{A}}}

\newcommand{\skA}{\catf{sk}_\cat{A}}

\newcommand{\nakal}{\catf{N}^\ell}\newcommand{\nakar}{\catf{N}^r}
\newcommand{\rmod}{\catf{mod}_{\cat{A}}\!-\!}\newcommand{\rmodc}{\catf{mod}_{\cat{C}}\!-\!}
\definecolor{Blue}  {rgb} {0.282352,0.239215,0.803921}
\definecolor{Green} {rgb} {0.133333,0.545098,0.133333}
\definecolor{Red}   {rgb} {0.803921,0.000000,0.000000}
\definecolor{Violet}{rgb} {0.580392,0.000000,0.827450}

\newcounter{jfctodoab}
\newcounter{jfctodolw}

\usepackage{dingbat}

\renewcommand{\today}{\ifcase \month \or January\or February\or March\or %
	April\or May \or June\or July\or August\or September\or October\or November\or %
	December\fi {} \number  \year} 

\makeatletter
\newcommand{\monthyeardate}{%
	\DTMenglishmonthname{\@dtm@month}, \@dtm@year
}
\makeatother

\setkomafont{caption}{\small\slshape}

\setkomafont{section}{\rmfamily\large}
\setkomafont{subsection}{\normalfont\slshape}
\setkomafont{subsubsection}{\rmfamily}
\setkomafont{subparagraph}{\normalfont\scshape}

\newtheorem*{theorem*}{Theorem}
\newtheorem*{corollary*}{Corollary}

\makeatletter
\renewcommand\section{\@startsection {section}{1}{\z@}%
	{-3.5ex \@plus -1ex \@minus -.2ex}%
	{2.3ex \@plus.2ex}%
	{\normalfont\scshape\centering}}
\makeatother

\usepackage{titlesec}
\titleformat{\subsection}[runin]
{\normalfont\slshape}
{\thesubsection}
{0.5em}
{}
[.]

\usepackage{titletoc}
\dottedcontents{section}[1.5em]{\rmfamily}{1.5em}{0.2cm}
\dottedcontents{subsection}[5em]{\rmfamily}{2.3em}{0.2cm}

\begin{document}

	\vspace*{-0.5cm}	\begin{center}	\textbf{\large{Reflection Equivariance 
				and the Heisenberg Picture \\[0.5ex] for Spaces of Conformal Blocks}} \\ 
		\vspace{1cm}{\large Lukas Woike }\\ 	\vspace{5mm}{\slshape  Université Bourgogne Europe\\ CNRS\\ IMB UMR 5584\\ F-21000 Dijon\\ France }\end{center}	\vspace{0.3cm}	
	\begin{abstract}\noindent 
		Monoidal product, braiding, balancing and weak duality are pieces of algebraic information that are well-known to have their origin in oriented genus zero surfaces and their mapping classes. More precisely, each of them correspond to operations of the cyclic framed $E_2$-operad. We extend this correspondence to include another algebraic piece of data, namely the modified trace, by showing that it amounts to a homotopy fixed point structure with respect to the homotopy involution that reverses the orientation of surfaces and dualizes the state spaces. We call such a homotopy fixed point structure \emph{reflection equivariance}. As an application, we describe the effect of orientation reversal on spaces of conformal blocks and skein modules in the non-semisimple setting, throughout relying on their factorization homology description. This has important consequences: For a modular functor that is reflection equivariant relative to a rigid duality, i) the circle category is modular, and the resulting mapping class group representations are automatically the ones built by Lyubashenko, and ii) the modules over the internal skein algebras are generated by one simple representation, carrying a unique projective mapping class group representation making the action equivariant. While i) is a new topological characterization of not necessarily semisimple modular categories, ii) generalizes the implicit description of spaces of conformal blocks purely through the representation theory of moduli algebras given by Alekseev-Grosse-Schomerus from rational conformal field theories admitting a Hopf algebra description to finite rigid logarithmic conformal field theories. This also generalizes several results of Faitg from ribbon factorizable Hopf algebras to arbitrary modular categories.
\end{abstract}

	\tableofcontents
	\normalsize

	\section{Introduction and summary}
The backbone of many constructions in quantum topology are state spaces associated to surfaces or sometimes three-dimensional handlebodies.
Some of these state spaces are referred to as 
\emph{spaces of conformal blocks} 
if the connection to conformal field theory is supposed to be emphasized; 
they are organized into what is called a \emph{modular functor}~\cite{Segal,ms89,turaev,tillmann,baki}. 
These state spaces have in common that they play a double role: They carry \emph{representations of extensions of mapping class groups} (sometimes called \emph{quantum representations of mapping class groups}), and they are, in some shape of form, \emph{skein modules}~\cite{turaevck,turaevskein,hp92,Przytycki,Walker}. 

The purpose of this article is to study the behavior of these quantities under orientation reversal and to derive several consequences.
More precisely, we will consider a certain homotopy fixed point structure with respect to orientation reversal and relate it to theory of modified traces~\cite{geerpmturaev,mtrace1,mtrace2,mtrace3,mtrace}. 
Finally, for modular functors compatible in a specific way with orientation reversal, we will prove that the spaces of conformal blocks automatically 
can be described as the unique simple modules over the skein algebras. This amounts to a vast generalization of the approach of Alekseev-Grosse-Schomerus to spaces of conformal blocks~\cite{alekseevmoduli,agsmoduli,asmoduli} based on algebras of observables called \emph{moduli algebras}.

Before discussing the specific problems, we have to cover a minimum of background and terminology: Throughout, we will need the  \emph{operad $\framed$ of framed little disks} \cite{bv68,mayoperad,bv73,WahlThesis,salvatorewahl}, which is the operad of oriented genus zero surfaces. It is a cyclic operad in the sense of Getzler-Kapranov~\cite{gk}. Moreover, we need two modular operads~\cite{gkmod}:
The
\emph{surface operad $\Surf$}~\cite{gkmod}, which is the 
operad of compact oriented surfaces with parametrized boundary, and the \emph{handlebody operad} $\Hbdy$~\cite{giansiracusa}, which is the operad of compact oriented three-dimensional handlebodies with parametrized
disks embedded in their boundary surface.
A concise and mostly self-contained introduction to cyclic and modular operads is given in~Section~\ref{seccycmodop}.

A modular algebra over a suitable extension of the modular operad of surfaces with values in a suitable symmetric monoidal bicategory of linear categories is called a \emph{modular functor}. It corresponds to the monodromy data of a two-dimensional conformal field theory~\cite{jfcs}.
If the underlying category for the circle is semisimple, the conformal field theory is \emph{rational}; otherwise, one calls it \emph{logarithmic}
(we will put aside for the moment the question to what extent  people typically include finiteness and rigidity assumptions; we will make our framework precise soon).
Early definitions of a modular functor that have helped shape the notion over the last few decades can be found in the above-mentioned texts
\cite{Segal,ms89,turaev,tillmann,baki}.
We should warn the reader that these definitions
are partially inequivalent and do not cover the generality needed in this article.
Instead, we will use the definition in~\cite{brochierwoike} that of course builds on these previous definitions. 
The perspective on modular functors in this article is mostly mathematical. For an introduction to the field-theoretic background, we refer e.g.\
to the encyclopedia entry~\cite{algcften}.

Cyclic framed $E_2$-algebras are exactly \emph{genus zero} modular functors. They may not extend to modular functors, but if they do, the extension is unique and can be characterized in terms of factorization homology~\cite{brochierwoike}. A semisimple version of the uniqueness of extensions is given in \cite{andersenueno}. However, all cyclic framed $E_2$-algebras extend uniquely to modular algebras over the handlebody operad $\Hbdy$, the so-called \emph{ansular functors} as was shown in \cite{mwansular} building on \cite{giansiracusa,cyclic}.

For all the cyclic and modular algebras in this article, we will fix an ambient symmetric monoidal bicategory to take their values in. One of the standard choices in quantum algebra is  $\Rexf$,
the symmetric monoidal bicategory of \begin{itemize}
	\item \emph{finite categories} over an algebraically closed field $k$ fixed throughout (\emph{finite} means here: linear abelian, finite-dimensional morphism spaces, enough projective objects, finitely many simple objects, finite length for every object) \item  
	whose 1-morphisms are right exact functors, \item and whose 2-morphisms 
	are linear natural transformations. \end{itemize}
The monoidal product is the Deligne product $\boxtimes$. 
Finite categories form one of the standard frameworks of 
 quantum algebra~\cite{etingofostrik,egno}. Just like its twin $\Lexf$ (featuring left exact functors instead of right exact ones), $\Rexf$ is a model for the Morita category of finite-dimensional algebras, finite-dimensional bimodules and intertwiners, see~\cite{fss} for more background. 	Working in $\Rexf$ will have some slight technical advantages in the article, but all results transfer to $\Lexf$ after applying the necessary modifications. 

Roughly, a $\Rexf$-valued modular functor $\mathfrak{F}$ with underlying category $\cat{A}\in\Rexf$ of boundary labels
associates to a surface $\Sigma$ with $n\ge 0$ boundary components (for us, all surfaces are compact, oriented and with parametrized boundary)
a right exact functor $\mathfrak{F}(\Sigma;-): \cat{A}^{\boxtimes n} \to \vect$ taking values in finite-dimensional vector spaces over $k$, with an extension of the mapping class group $\Map(\Sigma)$ acting 
 through natural automorphisms.
Once we evaluate on objects $X_1,\dots,X_n \in \cat{A}$, we obtain a vector space $\mathfrak{F}(\Sigma;X_1,\dots,X_n)$ that is often referred to as the 
\emph{space of conformal blocks} for the surface $\Sigma$ and boundary labels $X_1,\dots,X_n$. There is a compatibility with the gluing of surfaces.
For the description of cyclic framed $E_2$-algebras, one restricts this description to genus zero; for ansular functors, one replaces the surfaces with handlebodies.

Explicitly, a framed $E_2$-algebra in $\Rexf$ is a category $\cat{A}\in\Rexf$
with a \emph{balanced braided structure}~\cite{WahlThesis,salvatorewahl},
which is in more detail
\begin{itemize}\item 
	a monoidal product $\otimes :\cat{A}\boxtimes\cat{A}\to\cat{A}$, \item a braiding $c_{X,Y}:X\otimes Y \cong Y \otimes X$ for $X,Y \in \cat{A}$, \item and a balancing $\theta$ which is a natural automorphism $\theta_X :X\to X$ with $\theta_I=\id_I$ (for the monoidal unit $I\in\cat{A}$) and $\theta_{X\otimes Y}=c_{Y,X}\circ c_{X,Y}\circ (\theta_X \otimes \theta_Y)$. \end{itemize}
In order to make $\cat{A}$ additionally a \emph{cyclic} framed $E_2$-algebra, one needs by~\cite{cyclic} an equivalence $D:\cat{A}\to\cat{A}^\op$
that makes $\cat{A}$ a  \emph{ribbon Grothendieck-Verdier category} in the sense of Boyarchenko-Drinfeld~\cite{bd}, see also Barr's work on $*$-autonomous categories~\cite{barr}. This means roughly that $D$ equips $\cat{A}$ with a weak form of duality that satisfies $D\theta_X = \theta_{DX}$ for all $X\in \cat{A}$. We review the details in Section~\ref{seccycmodalgss}. Every modular functor inherits on its category of boundary labels a ribbon Grothendieck-Verdier structure such that $D$ is the orientation reversal for the labels~\cite{cyclic}.

The reader might be more familiar with the notion of a \emph{rigid duality} the existence of which seems to be one of the standing hypotheses in most classical texts in quantum algebra~\cite{turaev,egno}. For a rigid duality, the dual $X^\vee$ of $X$ comes with an evaluation $d_X:X^\vee \otimes X \to I$ and a coevaluation $b_X:I\to X\otimes X^\vee$ satisfying the zigzag identities.
A priori, there is the notion of a left and a right duality, but for all categories that appear in this article, these can be canonically identified through  a pivotal structure (a monoidal trivialization of the double dual), so we allow ourselves to not make the distinction.
A Grothendieck-Verdier duality does not have to be rigid, see~\cite{alsw} for counterexamples coming from vertex operator algebras such as~\cite{grw}, and even if $\cat{A}$ happens to be rigid, $D$ and the rigid duality need not coincide, see~\cite[Section~8.3 \& 8.4]{brochierwoike}, \cite{sn} and~\cite[Example~11.5]{microcosm}.

\subsection*{Reflection equivariance}
The collection of all oriented manifolds 
comes with an involution, namely the \emph{orientation reversal} or \emph{reflection operation}
that sends any oriented manifold $M$ to the manifold $\bar M$ obtained by reversing the orientation of $M$.
It equips the cyclic framed $E_2$-operad, the modular handlebody operad and the modular surface operad with an involution.
As a result, the moduli spaces of genus zero modular functors, ansular functors and modular functors come with a homotopy involution.
One concern of this article is to single out those (genus zero) modular functors  $\mathfrak{F}$ whose spaces of conformal blocks dualize under orientation reversal:
\begin{align}
	\mathfrak{F}(\bar \Sigma; X_1,\dots,X_n) \cong \mathfrak{F}(\Sigma;DX_n,\dots,DX_1)^* \ ,   \label{eqnisoeq} \tag{$*$} \end{align}
where $D:\cat{A}\to \cat{A}^\op$ is the Grothendieck-Verdier duality implementing the orientation reversal on labels; it is implicit here that these isomorphisms are $\Map(\Sigma)$-equivariant (on the right hand side,
we have the contragredient representation) and compatible with the gluing of surfaces. 

If we build modular functors 
without any topological field theory in the background, 
there is no reason for~\eqref{eqnisoeq} to hold. It has to be required as an additional structure~\cite[Definition~5.1.14]{baki}. 
Often \eqref{eqnisoeq} seems to be treated more like a property instead of additional structure without any justification. 
Related notions that are close to the ones that we are interested in here concerning language and motivation, but in an entirely different context are discussed for quantum field theories in \cite[Section~4]{freedhopkins}.

If~\eqref{eqnisoeq} holds for all objects, then the underlying balanced braided category is necessarily semisimple.
One might want to ask~\eqref{eqnisoeq} just for projective objects, but this still leaves the problem  whether~\eqref{eqnisoeq} is property or structure (if so, what structure?). The situation is even a bit  more complicated: Clearly, \eqref{eqnisoeq} should be an equivalence of modular functors,
but generally the right hand side, if we just define it on projective objects, might not even come from a modular functor. In the general case, we will see in fact that this is not necessarily well-defined. Another question that does not seem to be addressed anywhere beyond the semisimple case is the reflection behavior of the state spaces as skein modules (we will address this in Section~\ref{sechandlebodyrefl}). 

We will now describe a framework in which these problems can be resolved: Consider a genus zero modular functor in $\Rexf$, i.e.\ a ribbon Grothendieck-Verdier category $\cat{A}$ in $\Rexf$.
We can then try to define a new cyclic framed $E_2$-algebra by the right hand side of~\eqref{eqnisoeq}, i.e.\ via simultaneous dualization of the spaces of conformal blocks and the boundary labels, but just on projective objects. Afterwards, we extend as a right exact functor.
If $\cat{A}$ is self-injective and the monoidal product exact, this produces indeed a cyclic framed $E_2$-algebra that we denote by $\cat{A}\dg$
(Section~\ref{secdual}), and $\cat{A}\mapsto {\bar{\cat{A}}}\dg$ defines a homotopy involution ($\bar{\cat{A}}$ is obtained from $\cat{A}$ by applying the orientation reversal).
This is in particular the case
if $\cat{A}$, as monoidal category, is rigid and $D$ coincides with the rigid duality. Then we call $\cat{A}$ 
\emph{strongly rigid}. 
If $\cat{A}$ is additionally \emph{simple}, i.e.\ the endomorphisms of the unit are the ground field, $\cat{A}$ is a \emph{finite ribbon category} in the sense of~\cite{egno}. 
Under this assumption, $\cat{A}$ is, at the level of \emph{non-cyclic} framed $E_2$-algebras, already a homotopy fixed point for this involution. We will show this in Corollary~\ref{cormulfinrib}, and it will be based on the fact that generally the monoidal product on $\cat{A}\dg$ can be proven to be the `second' monoidal product in the sense of~\cite[Section~4.1]{bd}, which in the rigid case of course agrees with the original one.

This leads us to the main technical notion of this article: Let $\cat{A}$ still be strongly rigid. We  say that $\cat{A}$ is \emph{reflection equivariant}, as cyclic framed $E_2$-algebra, if it is equipped with the structure of a homotopy fixed point with respect to the homotopy involution $\cat{A}\mapsto {\bar{\cat{A}}}\dg$ as \emph{cyclic} algebra, with the homotopy fixed point structure being \emph{relative} to the non-cyclic one that it already has (Definition~\ref{deffixedpointrel}). 
We will see that this notion of reflection equivariance is the most reasonable way to make precise the genus zero version of~\eqref{eqnisoeq}.

Defining the reflection equivariance relative to a rigid duality seems awfully technical, 
but, for the moment, the rigid case will be our main field of application, and the rigid duality needs to be taken into account to arrive at a reasonable notion: Just as an illustration, on a sphere with three insertions (two incoming, one outgoing), the reflection equivariance is an isomorphism $(Y^\vee \otimes X^\vee) ^\vee \cong X \otimes Y$. Clearly, we already have such an isomorphism
if $-^\vee$ is a rigid pivotal duality, so this should be appropriately built in, and this is what the definition just given allows us to do. 
Through the following result, we achieve a characterization of cyclic reflection equivariance:

\spaceplease
\begin{reptheorem}{thmrefl}
	Let $\cat{A}$ be a cyclic framed $E_2$-algebra in $\Rexf$. Assume that $\cat{A}$ is strongly rigid and simple.
	Then the following structures are equivalent:
	\begin{pnum}
		\item The choice of a two-sided modified trace on $\Proj \cat{A}$, which is unique up to an element in $k^\times$. 
		\label{thmrefli}

		\item Cyclic reflection equivariance for $\cat{A}$, i.e.\ the structure of a $\mathbb{Z}_2$-homotopy fixed point of $\cat{A}$ for the homotopy involution $\cat{A}\mapsto \bar{\cat{A}}\dg$ on cyclic framed $E_2$-algebras relative to the canonical non-cyclic homotopy fixed point structure induced by the rigid pivotal duality of $\cat{A}$.  
		\label{thmreflii}

		\item Reflection equivariance for the unique ansular functor $\widehat{\cat{A}}$ extending $\cat{A}$, i.e.\ the choice of coherent isomorphisms
		\begin{align}
			\widehat{\cat{A}}(\bar H; P_1,\dots,P_n) \cong 	\widehat{\cat{A}}(H; DP_n,\dots,DP_1)^* \qquad \text{with}\qquad P_1,\dots,P_n \in \Proj \cat{A} \label{eqnisoHi}
		\end{align}
		for every handlebody $H$ with at least one embedded disk per connected component,
		such that these isomorphisms 
		\begin{align}
			\left. \begin{array}{l} - \text{are compatible with gluing,} \\ - \text{are involutive,} \\ - \text{and extend, on the level of non-cyclic framed $E_2$-algebras,} \\ \phantom{-} \text{the canonical isomorphisms coming from the rigid pivotal duality.}\end{array}\right\} \label{eqnreflconditions} \tag{$**$}
		\end{align}
		\label{thmrefliii}
	\end{pnum}
	Suppose further \begin{itemize}
		\item that any of the structures \ref{thmrefli}-\ref{thmrefliii} can be chosen and is fixed,
		\item and that $\cat{A}$ extends to a modular functor $\mathfrak{F}$.\end{itemize}
		Then 
	$\cat{A}$ has a non-degenerate braiding, i.e.\ it is modular, and $\mathfrak{F}$ agrees with the Lyubashenko modular functor for $\cat{A}$,
	and the isomorphisms~\eqref{eqnisoHi} extend to isomorphisms
	\begin{align}
		\mathfrak{F}(\bar \Sigma; P_1,\dots,P_n) \cong \mathfrak{F}(\Sigma; DP_n,\dots,DP_1)^* \qquad \text{with}\qquad P_1,\dots,P_n \in \Proj \cat{A}
	\end{align}
	for every surface $\Sigma$ with at least one boundary component per connected component for which the analoga of the conditions~\eqref{eqnreflconditions} hold.
\end{reptheorem}

Theorem~\ref{thmrefl} tells us that reflection equivariance can be easily controlled algebraically through the \emph{modified trace}.
From a different perspective, the modified trace is encoded topologically through a symmetry of the framed $E_2$-operad.
Recall that modified traces, a non-semisimple replacement for quantum traces, are one of the key tools in modern quantum topology; a list of  references that can serve as a starting point is~\cite{geerpmturaev,mtrace1,mtrace2,mtrace3,mtrace}. 
In the just listed texts for the modified trace, it is shown that a one-sided modified trace on $\cat{A}$ exists if and only if $\cat{A}$ is \emph{unimodular} in the sense of~\cite{eno-d}, i.e.\ if the distinguished invertible object $\alpha \in \cat{A}$ characterizing the quadruple dual via the \emph{Radford formula} $-^{\vee\vee\vee\vee}\cong \alpha \otimes-\otimes\alpha^{-1}$ is isomorphic to the unit $I$.
If $\cat{A}$ is additionally a finite ribbon category,
this modified trace is two-sided~\cite[Corollary 6.12]{shibatashimizu}.

In addition to the characterization of reflection equivariance for cyclic framed $E_2$-algebras,
Theorem~\ref{thmrefl} tells us that, if we are in the strongly rigid case, reflection equivariance and the fact that the cyclic framed $E_2$-algebra extends to a modular functor already implies that the category is modular, and that the modular functor is given by the Lyubashenko construction~\cite{lyubacmp,lyu,lyulex,kl}.	Concluding modularity of the circle category from the fact that it produces a modular functor is a subtle point. 
In the semisimple case, results in this direction are available in~\cite{baki,BDSPV15,etingofpenneys}. In the non-semisimple rigid case, 
an argument is given	 in \cite[Section~6.7]{kl} in the context of non-semisimple topological field theories, but this makes strong assumptions on an extension to higher-dimensional manifolds because it relies on the value for the one-holed torus being a Hopf algebra object. 
The more fundamental statement, just for modular functors,
 seems to be new. 
Also note that, without all these assumptions, the restriction of a modular functor to genus zero is generally \emph{not} a modular category, see
\cite[Section~8.3 \& 8.4]{brochierwoike}, \cite{sn} and~\cite[Example~10.5]{microcosm}.
Implications for finite ribbon categories originating from vertex operator algebras are discussed
in Example~\ref{exvoa}.

Theorem~\ref{thmrefl} may be seen as a \emph{globalization
	for modified traces}:
In~\ref{thmrefli}, we have the modified trace, i.e.\ traces for
hom spaces of projective objects, which are exactly spaces of conformal
blocks for spheres with two insertions.
In~\ref{thmreflii}, we have, as will become clear in the main
text, trace functions for spheres with arbitrarily many insertions.
Finally, \ref{thmrefliii} concludes the `globalization' to
higher genus handlebodies and eventually surfaces in a coherent mapping class group equivariant way.
The main point of Theorem~\ref{thmrefl} is that all of these
structures in each step are equivalent.

The main difficulty is the following:
When going from \ref{thmrefli} to~\ref{thmreflii}, it will be
necessary to produce from the traces on genus zero two-point blocks and
the partial trace property the traces for more insertions and the
homotopy fixed point structure.
This part will be done very indirectly without any explicit
reference to the usual list of properties characterizing the modified
trace. Instead, we will use
the equivalent description of modified traces through certain
trivialization of the Nakayama functor as a module functor developed
in~\cite{shibatashimizu,tracesw} building on the Morita invariant description of
Nakayama functors and their relation to the distinguished invertible
object in~\cite{fss}.
Moreover, we will heavily rely on the description of the space of cyclic
structures on a fixed non-cyclic framed $E_2$-algebra $\cat{A}$
as a torsor over the Picard group of the balanced Müger center
of $\cat{A}$ in~\cite{mwcenter}.

\subsection*{The topological characterization of modular categories that Theorem~\ref{thmrefl} relies on}
The part of Theorem~\ref{thmrefl} involving modular functors is technically rather demanding. It
will follow from a more fundamental result that is of independent interest and that we state here separately.
In \cite{brochierwoike} modular functors are defined in a purely topological way (no normalization, no prescribed set of simple objects, no rigid duality, \dots). As a consequence, one may build modular functors that do not come from modular categories. Instead, modular functors are classified by \emph{connected} cyclic framed $E_2$-algebras; we give a summary in Section~\ref{secmf}.

If modular categories are \emph{not} in one-to-one correspondence to modular functors, one should ask: What are modular categories equivalent to? How can we capture them (and only them) topologically? According to the preprint \cite{BDSPV15}, modular fusion categories are equivalent to anomalous once-extended three-dimensional topological field theories; this is contingent on a claimed presentation of the three-dimensional bordism category that is yet to appear.
The following result tells us what \emph{not necessarily semisimple modular categories} are equivalent to:

\begin{reptheorem}{thmmodularcats}
	Genus zero restriction provides a canonical bijection between
	\begin{itemize}
		\item equivalence classes of $\Rexf$-valued modular functors which are strongly rigid, simple and have the property of admitting a reflection equivariant structure,
		
		\item and ribbon equivalence classes of modular categories. 
	\end{itemize}
	The inverse to genus zero restriction is the Lyubashenko construction.
\end{reptheorem}

We recover in particular the recent classification of Etingof-Penneys~\cite{etingofpenneys} in the semisimple case (Corollary~\ref{corssimf}).

\subsection*{The Heisenberg picture for spaces of conformal blocks}

The projective mapping class group representations on the spaces of conformal blocks \label{introagsalgebras}
are notoriously difficult to construct and to investigate.
This is  already true in the semisimple setting, and even more so in 
the non-semisimple setting. 
In the most common approaches, the state spaces on which the mapping class group acts are constructed directly, sometimes in terms of generators and relations for the mapping class group~\cite{lyubacmp,lyu,lyulex,kl}, possibly with a topological field theory as surrounding structure~\cite{rt1,rt2,turaev,gai2} to the extent that this is possible.
This `state space approach' to quantum field theory is the \emph{Schrödinger picture}.

There is a second approach that a priori does not talk about state spaces at all, namely the \emph{Heisenberg picture} that focuses on \emph{algebras of observables}. 
 For this paper, this is not necessarily meant in the sense of \emph{algebraic quantum field theory}; we will not attempt to reflect this body of work by giving references and therefore just point at the overview text~\cite{Halvorson2006-HALAQF}. Instead, we focus on the approach more geared towards the topological description of spaces of conformal blocks.
In the rational case, a precise framework for this  is given by the \emph{moduli algebras} of Alekseev-Grosse-Schomerus~\cite{alekseevmoduli,agsmoduli,asmoduli} and Buffenoir-Roche~\cite{br1,br2}. 
Instead of explaining the original construction, let us first describe our general result in the language established for moduli algebras in \cite{bzbj,bzbj2} using factorization homology. 
 A further thoroughly motivated explanation how the factorization homology is really a refined, Morita invariant version of the Heisenberg picture can be found in~\cite{jfheisenberg}. 
Afterwards, we comment on the case that was known before.

Suppose that $\mathfrak{F}$ is any $\Rexf$-valued modular functor describing the monodromy data of \emph{any} two-dimensional conformal field theory. 
By evaluation in genus zero, we obtain a cyclic framed $E_2$-algebra $\cat{A}$~\cite{cyclic}, and the value of $\mathfrak{F}$ on a surface $\Sigma$ (let us focus on the closed case for the moment) comes by \cite{brochierwoike} automatically with an action of the skein algebra
\begin{align}
	\SkAlg_\cat{A}(\Sigma):= \End_{\int_\Sigma \cat{A}}(\qss)  
	\end{align}
which is defined as the endomorphism algebra of the quantum structure sheaf $\qss\in \int_\Sigma \cat{A}$ in factorization homology introduced in~\cite{bzbj}, see Section~\ref{secfh} for a brief introduction to factorization homology and how it recovers classical skein-theoretic constructions.
The algebra  $\SkAlg_\cat{A}(\Sigma)$ comes with a mapping class group representation making the action
$\SkAlg_\cat{A}(\Sigma)\to\End\, \mathfrak{F}(\Sigma)$ mapping class group equivariant, where the action on $\End\, \mathfrak{F}(\Sigma)$ is the adjoint one (Theorem~\ref{thmequiv}). 
We will now see how $\mathfrak{F}(\Sigma)$ is characterized by  the representation theory of the algebra $	\SkAlg_\cat{A}(\Sigma)$:

\begin{reptheorem}{thmskeinalg}
	Let $\mathfrak{F}$ be a $\Rexf$-valued modular functor whose value on the circle is the cyclic framed $E_2$-algebra $\cat{A}$.
	Suppose that $\mathfrak{F}$ is reflection equivariant relative to a rigid duality and simple, thereby making $\mathfrak{F}$ the Lyubashenko modular functor for the modular category $\cat{A}$.
	Then the following holds:
	\begin{pnum}
		\item
		For any closed surface $\Sigma$, the space of conformal blocks $\mathfrak{F}(\Sigma)$ is entirely characterized as follows:
		\begin{itemize}
		\item	$\mathfrak{F}(\Sigma)$ is the unique simple module over $\SkAlg_\cat{A}(\Sigma)$.
			\item The projective mapping class group action on $\mathfrak{F}(\Sigma)$ is the only one making the action mapping class group equivariant.\end{itemize}
		\item The mapping class group action on $\SkAlg_\cat{A}(\Sigma)$ is by inner automorphisms.
		\end{pnum}
	\end{reptheorem}

In \cite[Theorem~10 \& 11, moreover Section~9.3]{asmoduli} this result is proven, in the language of moduli algebras, in the case when $\cat{A}$ is given by finite-dimensional modules over a 
 semisimple ribbon factorizable Hopf algebra. 
 In a skein-theoretic language, for a certain class of  examples,  related  results appeared in \cite{roberts,masbaumroberts} and amount essentially to the important insight that skein algebras are matrix algebras. 
 
 Theorem~\ref{thmskeinalg} provides a generalization to the non-semisimple case, i.e.\ to finite rigid logarithmic conformal field theories, and it is independent of a Hopf algebra description.
 Characterizing the representations of the Kauffman bracket skein algebras to quantum groups is a problem that has been addressed in \cite{fkbl}. This is a statement that at present we do not know how to interpret using the framework of finite tensor categories, but by means of the tools for the comparison of skein theory and factorization homology \cite{cooke,skeinfin,brochierwoike,brownhaioun,mwskein}, it is conceivable that the techniques leading to Theorem~\ref{thmskeinalg} will eventually allow us to treat generalizations of \cite{fkbl}.
 
 The situation of surfaces with boundary is treated in Theorem~\ref{thmmodulishort}: Suppose that our surface $\Sigma$ is connected with $n\ge 1$ boundary components. We prove that, under the assumptions of Theorem~\ref{thmskeinalg},
 the representations of the internal skein algebra \cite{bzbj,bzbj2,skeinfin} are freely generated, as $\cat{A}^{\boxtimes n}$-module category, by one simple representation, namely the space of conformal blocks, here seen as object in $\cat{A}^{\boxtimes n}$. If $\cat{A}$ is given by finite-dimensional modules over a ribbon factorizable Hopf algebra and if $\Sigma$ has one boundary component, then Theorem~\ref{thmskeinalg} recovers several results of Faitg, and thereby generalizes these results beyond Hopf algebras, see Remark~\ref{remhopfalg}.
 A direct analogue of Theorem~\ref{thmskeinalg} for the situation with boundary that directly characterizes the value of the modular functor through its skein action is given in Corollary~\ref{corboundary}.

 The non-internal skein algebras for a surface with one boundary component, however, are not so well behaved: In Example~\ref{exHmod}, we discuss a situation where this algebra is not semisimple. In other situations, they are, regardless of non-semisimplicity of $\cat{A}$, still semisimple, for example if the mapping class group representations have the Torelli groups in their kernel (Example~\ref{extorelli}).

 Theorem~\ref{thmskeinalg} covers only modular categories, but in
 Example~\ref{exffbs} we show by explicit calculations that an analogue of the result still holds for the 
  Feigin-Fuchs boson from~\cite{alsw}, a module category for a lattice vertex operator algebra that is not ribbon.

The Theorems~\ref{thmskeinalg} and~\ref{thmmodulishort} solve the problem of generalizing the main results of \cite{asmoduli} to finite rigid logarithmic conformal field theories, but we want to remark that there is actually also an application in the rational case. It pertains to one of the comparison problems mentioned in \cite[Section~4.3]{asmoduli}: In this article, we treat spaces of conformal blocks axiomatically through the theory of modular functors. If however we describe the conformal field theory through a vertex operator algebra, there is also an approach that builds the spaces of conformal blocks as vector bundles over a suitable moduli space of curves equipped with the Knizhnik-Zamolodchikov connection. The transfer of the above `Heisenberg picture' to \emph{those} spaces of conformal blocks is far from clear. We address this in Section~\ref{subsecvoa} through a comparison to the spaces of conformal blocks of Ben-Zvi and Frenkel~\cite{frenkelbenzvi} for a strongly rational, self-dual vertex operator algebra $V$. In that case, we know that the 
spaces of conformal blocks $\mathbb{V}$ built from $V$ form a modular functor~\cite{damioliniwoike} (beyond the strongly rational self-dual case, this is open). At genus $g$ (no non-trivial insertions), there is a flat algebra bundle $\mathbb{S}_g$ over the moduli space. The projectively flat connection on $\mathbb{V}_g$ is determined by the requirement that the action $\mathbb{S}_g \to \mathbb{V}_g^* \otimes \mathbb{V}_g$ be compatible with the connection.

		\vspace*{0.2cm}\textsc{Acknowledgments.} For insightful comments on moduli algebras constructed from Hopf algebras, I would like to thank Matthieu Faitg. 
		I would also like to thank 
		Max Demirdilek and Christoph Schweigert for helpful comments on the $\dgr$-operation for cyclic algebras. 
		A version of Theorem~\ref{thminv} had already been mentioned in the introduction of the first version of \cite{brochierwoike}, with the core argument for the proof contained in the proof of other statements in that article. In the present article, a generalized version is fleshed out, and I thank Adrien Brochier  for helpful feedback.
	Moreover, I thank 
	Jorge Becerra, Chiara Damiolini, Simon Lentner,
	Lukas Müller, Kenichi Shimizu, Yang Yang and Deniz Yeral for helpful discussions related to the manuscript.
		LW  gratefully acknowledges support
		by the ANR project CPJ n°ANR-22-CPJ1-0001-01 at the Institut de Mathématiques de Bourgogne (IMB).
		The IMB receives support from the EIPHI Graduate School (ANR-17-EURE-0002).

		\section{Preliminaries on cyclic and modular operads and their algebras\label{seccycmodop}}
	Operads were introduced by Boardman-Vogt and May~\cite{bv68,mayoperad,bv73}
	to describe algebraic structures with multiple inputs and one output. This is done by giving separately for every $n\ge 0$ an object $\cat{O}(n)$ of operations with $n$ inputs and one output, the \emph{$n$-ary operations}. The objects $\cat{O}(n)$ can live in any symmetric monoidal category. This allows us to discuss not only sets of operations, but also chain complexes, categories or spaces of operations. The permutation group $\Sigma_n$ on $n$ letters acts on $\cat{O}(n)$ by 
	permutation of the inputs. Moreover, there is a composition operation that allows us to insert the output of one operation into any of the input slots of another operation. 
	The data is subject to several compatibility conditions. A textbook reference for the theory of operads is \cite[Chapters~1-3]{FresseI}. 
	
	\subsection{The operad of framed little disks}
	One class of examples for operads, that in fact motivated the introduction of the concept in~\cite{bv68,mayoperad,bv73}, are the topological
	\emph{operads $\catf{f}E_r$
		of framed little $r$-disks} with $r\ge 1$
	whose space $\catf{f}E_r(n)$ of arity $n$ operations is given
	by the space of embeddings of $n$ many $r$-dimensional disks into one $r$-dimensional
	disk that are composed of translations, rescalings and rotations. The operad $\framed$ will be of particular interest to us; it takes values in aspherical topological spaces, namely the classifying spaces of ribbon braid groups. Therefore, it can be seen as a groupoid-valued operad. The connection to ribbon braid groups and the resulting groupoid-model for $\framed$ is given in \cite{WahlThesis,salvatorewahl}. 
	
	\subsection{The surface operad as prototypical modular operad}
	Cyclic operads were defined by Getzler and Kapranov~\cite{gk} 
	to describe operads equipped with a prescription how to cyclically permute inputs  with the output. In particular, the $\Sigma_n$-action on $\cat{O}(n)$ 
	will come with 
	an extension to a $\Sigma_{n+1}$-action in this case.
	Since the arity is $n$, we have $n$ inputs and one output, so the \emph{total arity} is $n+1$.
	The same authors define modular operads~\cite{gkmod} as cyclic operads that additionally come with a self-composition of operations. 
	The motivating example is the modular operad $\Surf$ of surfaces. The groupoid version of $\Surf$ (this is the only version that we need because we will consider bicategorical algebras, see Section~\ref{seccycmodalgss})
	has as groupoid $\Surf(n)$ of arity $n$ operations the groupoid whose objects are connected surfaces (for us, surfaces are always compact oriented two-dimensional smooth manifolds) with $n+1$ parametrized boundary components. The boundary parametrization is an orientation-preserving diffeomorphism $(\mathbb{S}^1)^{\sqcup (n+1)} \ra{\cong} \partial \Sigma$. 
	A morphism between two such surfaces is an isotopy class of orientation-preserving and parametrization-preserving diffeomorphisms; in other words, the morphisms are mapping classes, see \cite{farbmargalit} for an textbook introduction to mapping class groups.  The operadic composition is given by the gluing of surfaces. 
	If we restrict the modular operad $\Surf$ to genus zero, we obtain a cyclic operad that is equivalent to $\framed$.
	This allows us to think of $\framed$ as the cyclic operad of genus zero surfaces. 
	
	\subsection{Costello's description of cyclic and modular operads}
	In this paper, we will use a definition of cyclic and 
	modular operads along the lines of \cite{costello}. For this,
	we need  the graph categories $\Graphs$ and $\Forests$. 
	Recall that a graph consists of a set $H$ of \emph{half edges},
	a set $V$ of \emph{vertices}, a map $s : H \to V$ specifying the vertex that each half edge is attached to and an involution on the half edges that maps half edges that are glued together to one another.
	The orbits of the involution are the \emph{edges}; an orbit of size two is an \emph{internal edge} while an orbit of size one will be called a \emph{leg}. 
	
	A graph with one vertex and finitely many external legs (possibly no legs) attached to it will be called  a \emph{corolla}. Finite disjoint unions of corollas will be the objects of the category $\Graphs$ that we need in the sequel. To define the morphisms, we denote for each finite graph $\Gamma$ (finite means that the set of half edges and the set of vertices are finite) \begin{itemize}
		\item by $\nu(\Gamma)$ the disjoint union of corollas obtained by cutting $\Gamma$ open at all internal edges,\item  and by $\pi_0(\Gamma)$ the disjoint union of corollas obtained by contracting all internal edges.\end{itemize} 
	A morphism $ T \to T'$ is an equivalence class of triples $(\Gamma,\varphi_0,\varphi_1)$ of a finite graph $\Gamma$ together with
	identifications $\varphi_1 :\nu(\Gamma)\cong T$ and $\varphi_2 : \pi_0(\Gamma)\cong T'$. 
	Two such triples $(\Gamma,\varphi_0,\varphi_1)$ and $(\Gamma',\varphi_0',\varphi_1')$
	are equivalent if there is an identification $\Gamma \cong \Gamma'$ 
	compatible in the obvious way with $\varphi_0,\varphi_1,\varphi_0'$ and $\varphi_1'$. 
	Figure~\ref{figexgraph} shows an example of a morphism in $\Graphs$.
	The composition $\Gamma_2 \circ \Gamma_1$ of two morphisms is obtained by `inserting' $\Gamma_1$ into
	the vertices of $\Gamma_2$. Disjoint union of graphs endows $\Graphs$ with a symmetric monoidal structure.  For more details, we refer besides~\cite{costello} also to the summary in~\cite[Section~2.1]{cyclic}.

	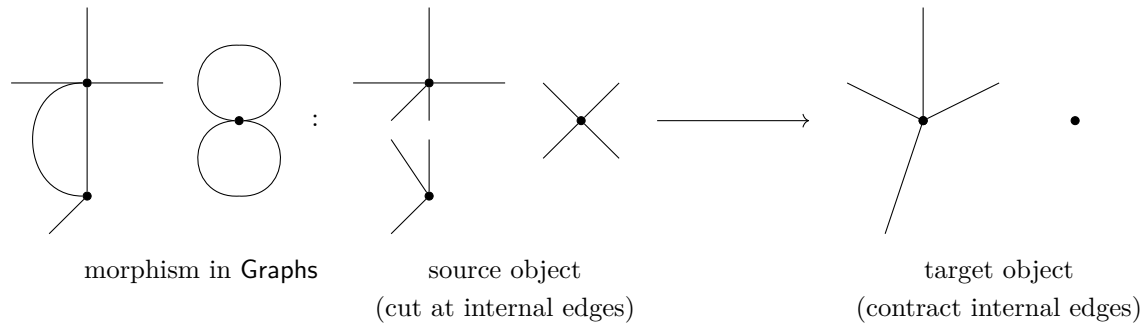
\begin{figure}[h]
	\begin{tikzpicture}[scale=0.5]
				\begin{pgfonlayer}{nodelayer}
					\node [style=none] (0) at (-8, 4) {};
					\node [style=none] (1) at (-8, 2) {};
					\node [style=none] (2) at (-4, 3) {};
					\node [style=none] (3) at (-8, 6) {};
					\node [style=none] (4) at (-10, 4) {};
					\node [style=none] (5) at (-6, 4) {};
					\node [style=black dot] (6) at (-8, 4) {};
					\node [style=black dot] (7) at (-8, 1) {};
					\node [style=black dot] (8) at (-4, 3) {};
					\node [style=none] (9) at (-9, 0) {};
					\node [style=none] (10) at (-4, 5) {};
					\node [style=none] (11) at (-4, 1) {};
					\node [style=none] (12) at (-2, 3) {$:$};
					\node [style=none] (13) at (1, 4) {};
					\node [style=none] (14) at (1, 2) {};
					\node [style=none] (15) at (5, 3) {};
					\node [style=none] (16) at (1, 6) {};
					\node [style=none] (17) at (-1, 4) {};
					\node [style=none] (18) at (3, 4) {};
					\node [style=black dot] (19) at (1, 4) {};
					\node [style=black dot] (20) at (1, 1) {};
					\node [style=black dot] (21) at (5, 3) {};
					\node [style=none] (22) at (0, 0) {};
					\node [style=none] (25) at (0, 3) {};
					\node [style=none] (26) at (1, 3) {};
					\node [style=none] (27) at (0, 2.5) {};
					\node [style=none] (28) at (1, 2.5) {};
					\node [style=none] (29) at (4, 2) {};
					\node [style=none] (30) at (6, 2) {};
					\node [style=none] (31) at (4, 4) {};
					\node [style=none] (32) at (6, 4) {};
					\node [style=none] (33) at (7, 3) {};
					\node [style=none] (34) at (11, 3) {};
					\node [style=none] (35) at (14, 4) {};
					\node [style=none] (36) at (14, 2) {};
					\node [style=none] (37) at (18, 3) {};
					\node [style=none] (38) at (14, 6) {};
					\node [style=none] (39) at (12, 4) {};
					\node [style=none] (40) at (16, 4) {};
					\node [style=black dot] (41) at (14, 3) {};
					\node [style=black dot] (42) at (14, 3) {};
					\node [style=black dot] (43) at (18, 3) {};
					\node [style=none] (44) at (13, 0) {};
					\node [style=none] (45) at (18, 5) {};
					\node [style=none] (46) at (3, -2) {(cut at internal edges)};
					\node [style=none] (47) at (16, -2) {(contract internal edges)};
					\node [style=none] (48) at (-5, -1) {morphism in $\Graphs$};
					\node [style=none] (49) at (3, -1) {source object};
					\node [style=none] (50) at (16, -1) {target object};
				\end{pgfonlayer}
				\begin{pgfonlayer}{edgelayer}
					\draw (6) to (3.center);
					\draw (6) to (5.center);
					\draw (6) to (4.center);
					\draw [bend right=90, looseness=1.50] (6) to (7);
					\draw (7) to (9.center);
					\draw (6) to (7);
					\draw [bend right=270, looseness=1.75] (8) to (10.center);
					\draw [bend left=90, looseness=1.75] (8) to (11.center);
					\draw [bend right=90, looseness=1.75] (8) to (10.center);
					\draw [bend right=90, looseness=1.75] (8) to (11.center);
					\draw (19) to (16.center);
					\draw (19) to (18.center);
					\draw (19) to (17.center);
					\draw (20) to (22.center);
					\draw (19) to (26.center);
					\draw (25.center) to (19);
					\draw (20) to (28.center);
					\draw (20) to (27.center);
					\draw (21) to (29.center);
					\draw (21) to (30.center);
					\draw (21) to (31.center);
					\draw (21) to (32.center);
					\draw [style=end arrow] (33.center) to (34.center);
					\draw (41) to (38.center);
					\draw (41) to (40.center);
					\draw (41) to (39.center);
					\draw (42) to (44.center);
				\end{pgfonlayer}
			\end{tikzpicture}
		\caption{An example of a morphism in $\Graphs$.}
		\label{figexgraph}
	\end{figure}

	A \emph{modular operad} with values in a (higher) symmetric monoidal category $\cat{M}$ is a symmetric monoidal functor $\cat{O} : \Graphs \to \cat{M}$ that also comes with operadic identities (we explain the latter in Remark~\ref{remoperadicidentities} below). It is always understood 
	that $\cat{O}$ is a symmetric monoidal functor up to coherent isomorphism (the symmetric monoidal 1-category $\Graphs$ is seen here as symmetric monoidal higher category with no non-trivial higher morphisms). In this article, the case in which $\cat{M}$ is a symmetric monoidal bicategory in the sense of \cite[Section~2]{schommerpries} suffices; this case is the one worked out in \cite[Section~2.1]{cyclic} to which we refer to details.
	
	One defines $\Forests \subset \Graphs$ as the symmetric monoidal subcategory with the same objects,
	 but only those morphisms coming from graphs whose connected components are contractible, so-called \emph{forests}. A \emph{cyclic operad} with values in $\cat{M}$ is now defined as a symmetric monoidal functor $\cat{O}:\Forests \to \cat{M}$.

	If $\cat{O}$ is a cyclic or modular operad and $T$ a corolla with $n\ge 0$ legs, the object $\cat{O}(T) \in \cat{M}$ has to be understood
	as the object of operations of arity $n-1$ or total arity $n$,
	to be thought of as $n-1$ inputs and one output, even though this distinction is not made for cyclic and modular operads. In order to make this distinction, one needs to equip each corolla with the choice of a preferred leg, a so-called \emph{root}.  This leads to the symmetric monoidal category $\RForests$, and symmetric monoidal functors out of it describe ordinary, non-cyclic operads (but we will not need this in this paper).
	
	If $\cat{O}$ is a cyclic operad,
	the permutation group $\Sigma_n$ on $n$ letters acts through the permutation of the $n$ legs of $T$. As a result, $\cat{O}(T)$ comes with a homotopy coherent $\Sigma_n$-action. The evaluation of $\cat{O}$ on forests whose underlying graph is not a disjoint union of corollas gives us the operadic composition. If $\cat{O}$ is modular, we additionally obtain composition operations by evaluation of $\cat{O}$ on morphisms that are not forests. These give us the  self-compositions of operations. 
	
	\begin{remark}\label{remoperadicidentities}
	 The definition of cyclic and modular operads as symmetric monoidal functors out of $\Forests$ and $\Graphs$, respectively, does not yet include operadic identities. These have to be encoded by hand~\cite[Definition~2.3]{cyclic}. In this paper, all operads are equipped with operadic identities, i.e.\ an operation of total arity two whose defining feature is that it behaves neutrally with respect to operadic composition up to coherent isomorphism. 
	Moreover, morphisms of cyclic and modular operads (defined as symmetric monoidal transformations) are required to preserve operadic identities up to coherent isomorphism. 
	\end{remark}

	\subsection{Our target category: finite linear categories\label{secrexf}}
	If we are given a category-valued cyclic or modular operad $\cat{O}$, 
	we can consider cyclic or modular $\cat{O}$-algebras, 
	respectively, in any symmetric monoidal 
	bicategory $\cat{S}$ as defined in 
	\cite[Section~2.2-2.4]{cyclic}.
	Before recalling this in Section~\ref{seccycmodalgss} below, let us mention our
	running example for 
	the symmetric monoidal bicategory 
	the cyclic and modular algebra
	take their values in: $\Rexf$,
	the symmetric monoidal bicategory of finite linear categories over our fixed algebraically closed field $k$ (for short: finite categories because the linear context is implicit) and right exact functors whose definition was already given in the introduction.
	The symmetric monoidal product is the Deligne product $\boxtimes$; the monoidal unit is the category $\vect$ of finite-dimensional $k$-vector spaces.  
	By replacing right exact functors with left exact ones, one can also define the symmetric monoidal bicategory $\Lexf$, but $\Rexf$ will have the advantage that it will facilitate the connection to the works~\cite{bzbj,bzbj2,brochierwoike}.
	
Later we will need a standard fact that follows for instance from~\cite[Section~1.1]{skeinfin} or~\cite[Proposition~5.6]{sn}:
For a finite category $\cat{A}$ and right exact functors $F,G:\cat{A}\to\vect$ the natural transformations $\Nat(F,G)=\int_{X\in\cat{A}}F(X)^*\otimes G(X)$ from $F$ to $G$ can be calculated via restriction to projective objects, i.e.\ the restriction map
\begin{align}
	\catf{Nat}(F,G)\ra{\cong} \int_{P \in\Proj \cat{A}} F(P)^* \otimes G(P) \ \label{eqnnat}
\end{align}
is a canonical isomorphism.
	
	\subsection{Cyclic and modular algebras\label{seccycmodalgss}}
	Classically, an algebra $\cat{A}$ over an operad $\cat{O}$ is an operad map from $\cat{O}$ to the endomorphism operad of the underlying object of the algebra $\cat{A}$. In order to generalize this to cyclic and modular operads, the endomorphism operad needs to be made modular (and therefore in particular cyclic). This is done in \cite[Section~2]{cyclic} for the bicategorical case, following the principles of~\cite{gk,gkmod,costello}: 
	First we recall 
	the notion of a non-degenerate symmetric pairing 
	$\kappa : \cat{A}\boxtimes\cat{A}\to I$ on an object $\cat{A}$ in a symmetric monoidal bicategory $\cat{S}$ with monoidal product $\boxtimes$ and monoidal unit $I$, where 
	\begin{itemize}
		\item \emph{non-degeneracy} means that $\kappa$ exhibits $\cat{A}$ as its own dual in the homotopy category of $\cat{S}$,
		\item and \emph{symmetry} means that $\kappa$ is a homotopy fixed point in $\cat{S}(\cat{A}\boxtimes\cat{A},I)$ with respect to the $\mathbb{Z}_2$-action coming from the symmetric braiding of $\cat{S}$.
	\end{itemize}
	The non-degeneracy can also be expressed by saying  that there is a coevaluation object $\Delta : I \to \cat{A}\ot \cat{A}$ 
	that together with $\kappa$ satisfies the zigzag identities up to isomorphism.
	Thanks to the non-degenerate symmetric pairing, we can define a symmetric monoidal
	functor $\Graphs \to \Cat$ sending a corolla $T$ to the morphism category $\cat{S}(\cat{A}^{\boxtimes \Legs(T)},I)$
	from the unordered monoidal product $\cat{A}^{\boxtimes \Legs(T)}$ of $\cat{A}$ over the set $\Legs(T)$ of legs of $T$
	to the monoidal unit $I\in \cat{S}$. This $\Cat$-valued modular operad is the \emph{endomorphism operad of $(\cat{A},\kappa)$} and denoted by $\End_\kappa^\cat{A}$.
	
	A \emph{modular $\cat{O}$-algebra} in $\cat{S}$ is now defined as an object $\cat{A}$ with non-degenerate symmetric pairing $\kappa$ plus a map $\cat{O}\to \End_\kappa^\cat{A}$ of modular operads. A cyclic algebra is defined analogously. Both cyclic and modular algebras over a fixed cyclic or modular operad, respectively, with values in a symmetric monoidal bicategory form themselves bicategories in which all morphisms turn out to be invertible, i.e.\ they form 2-groupoids~\cite[Proposition~2.18]{cyclic}.  
	
	Let us partly unpack the definition of a $\Rexf$-valued modular $\cat{O}$-algebra (all of this is worked out in detail in~\cite[Section~2.5]{cyclic} where the dual case of $\Lexf$ is treated): It has an underlying object $\cat{A} \in \Rexf$, i.e.\ a finite category,
	and a non-degenerate symmetric pairing $\kappa : \cat{A}\boxtimes \cat{A} \to \vect$. By virtue of right exactness, $\kappa$ determines an equivalence $D:\cat{A}\to\cat{A}^\op$ via $\kappa(X,Y)=\cat{A}(X,DY)^*$, and the symmetry of $\kappa$ induces an isomorphism $D^2 \cong \id_\cat{A}$. 
	The coevaluation is a right exact functor
	$\Delta : \vect \to \cat{A}\boxtimes \cat{A}$ and is hence an object in $\cat{A}\boxtimes\cat{A}$
	given by the end \begin{align} \label{eqncoev} \Delta = \int_{X \in \cat{A}} DX \boxtimes X\ ; \end{align} we refer to \cite{fss} for an introduction to (co)ends in finite linear categories.
	Moreover, the structure of a modular $\cat{O}$-algebra on $\cat{A}$ provides us for each operation $o\in \cat{O}(T)$ for a corolla $T$ with a right exact functor $\cat{A}_o : \cat{A}^{\boxtimes \Legs(T)} \to \vect$. 
	If $T \in \Graphs$ is not connected, i.e.\ $T=\sqcup_{\ell \in L}T^{(\ell)}$ with corollas $T^{(\ell)}$, then to $o=(o_\ell)_{\ell\in L} \in \cat{O}(T)\simeq \prod_{\ell\in L} \cat{O}(T^{(\ell)})$ the modular algebra associates the family of right exact functors $(  \cat{A}_{o_\ell})_{\ell\in L}$; it gives us also a right exact functor $\boxtimes_{\ell\in L} \cat{A}_{o_\ell}\to\vect$ (by abuse of notation, we will not distinguish between these two objects). The data of a modular algebra includes also the compatibility with composition: Suppose that a morphism $\Gamma : T \to T'$ glues two legs together, 
	and denote by $o' \in \cat{O}(T')$ the image of some $o\in \cat{O}(T)$ under the operadic composition $\cat{O}(\Gamma):\cat{O}(T)\to\cat{O}(T')$.
	Then we obtain an isomorphism of right exact functors between $\cat{A}_{o'}$ and $\cat{A}_o$, but with $\Delta$ inserted into the two arguments of $\cat{A}_o$ corresponding to the legs that are glued together, i.e.\
	\begin{align} \cat{A}_{o'} \cong \cat{A}_o (\dots,\Delta',\dots,\Delta'',\dots) \label{eqnexcision}
	\end{align}
	with Sweedler notation $\Delta = \Delta' \boxtimes \Delta''$ for the coevaluation object (as usual, the notion does not imply that $\Delta$ is a pure tensor).
	In some contexts, \eqref{eqnexcision} is called \emph{excision}.

	\subsection{Explicit description of cyclic framed $E_2$-algebras\label{secexplicitribgv}}
	From \cite{WahlThesis,salvatorewahl}, one can deduce that a non-cyclic framed $E_2$-algebra in $\Rexf$ is a \emph{balanced braided monoidal category in $\Rexf$} whose definition we recalled in the introduction.
	It is the main result of \cite{cyclic}
	that a \emph{cyclic} framed $E_2$-algebra in $\Rexf$ is a ribbon Grothendieck-Verdier category in $\Rexf$ as defined by Boyarchenko-Drinfeld in~\cite{bd} building on Barr's notion of a $*$-autonomous category~\cite{barr}. Clearly, we have an underlying non-cyclic framed $E_2$-algebra, i.e.\ a balanced braided category. Additionally, there is an object $K\in \cat{A}$, the \emph{dualizing object}, making the functors $\cat{A}(-\otimes X,K)$ for $X\in \cat{A}$ representable via $\cat{A}(-\otimes X,K)\cong \cat{A}(-,DX)$ with an equivalence $D:\cat{A}\to\cat{A}^\op$, called \emph{Grothendieck-Verdier duality}, subject to the condition $\theta_{DX}=D\theta_X$ for all $X\in \cat{A}$. 
	By \cite{alsw} vertex operator algebras subject to mild conditions produce, after taking modules over the vertex operator algebra, ribbon Grothendieck-Verdier categories, with the Grothendieck-Verdier dual being the contragredient representation.

	\subsection{Rigid duality}
	There is an important special case in which the duality $D$ is \emph{rigid} in the sense that $D$ sends $X \in \cat{A}$ to an object $X^\vee$ that comes with morphisms $d:X^\vee \otimes X \to I$ and $b: I \to X\otimes X^\vee$ satisfying the classical zigzag identities. A priori, there would be two versions of rigid duality that can be defined, a left and a right duality. In the cases treated in this article, they would necessarily coincide, so we do not make this distinction. 
	A ribbon Grothendieck-Verdier category in $\Rexf$
	whose Grothendieck-Verdier duality is rigid and whose unit is simple
	is a \emph{finite ribbon category} in the sense of \cite{egno}. It is in particular a \emph{finite tensor category}~\cite{etingofostrik}, a finite linear category with rigid bilinear monoidal product and simple unit. 
	An example of a finite ribbon category is the category of finite-dimensional modules over a finite-dimensional ribbon Hopf algebra, see \cite{kassel} for a further textbook reference. A finite ribbon category is also a \emph{pivotal finite tensor category}, i.e.\ a finite tensor category equipped with a monoidal trivialization $\id_\cat{A} \cong -^{\vee\vee}$ of the double dual functor called \emph{pivotal structure}.

	The quadruple dual of a finite tensor category can be described, thanks to~\cite{eno-d},
	via $-^{\vee\vee\vee\vee} \cong \alpha \otimes-\otimes\alpha^{-1}$ with the \emph{distinguished invertible object} $\alpha$, see also \cite{shimizuunimodular}. A finite tensor category whose
	 distinguished invertible object is isomorphic to the monoidal unit is called \emph{unimodular}.

	For a braided monoidal category $\cat{A}$, we denote by $Z_2(\cat{A})$ the \emph{Müger center} of $\cat{A}$, the full subcategory of all $X\in\cat{A}$ such that the double braiding $c_{Y,X}c_{X,Y}$ is the identity $\id_{X\otimes Y}$ for all $Y\in \cat{A}$. Clearly, there is a right exact functor $\vect \to Z_2(\cat{A})$ sending the ground field to the monoidal unit. If this functor is an equivalence, we call the braiding of $\cat{A}$ \emph{non-degenerate}.
	A \emph{modular category} is a finite ribbon category with non-degenerate braiding. For equivalent characterizations of the non-degeneracy of the braiding, we refer to \cite{shimizumodular}.

	\subsection{Factorization homology\label{secfh}}
	\emph{Factorization homology} was defined in \cite{higheralgebra,AF} inspired by the concept of a \emph{chiral algebra} developed by Beilinson-Drinfeld in~\cite{bdca}. It allows us to `integrate' a framed $E_2$-algebra in $\Rex$ 
	 over a surface $\Sigma$ (for us, as always compact and oriented).
	 Here $\Rex$ is the symmetric monoidal bicategory of finitely cocomplete linear categories, right exact functors and linear natural transformations (in contrast to $\Rexf$, there are no finiteness conditions).
	For all $m\ge 0$, consider all oriented embeddings $\varphi : \sqcup_{m} \mathbb{D}^2 \to \Sigma$  of $m$ disks into $\Sigma$.
	The factorization homology $\int_\Sigma \cat{A}$ of $\cat{A}$ on $\Sigma$ is, roughly speaking, the homotopy colimit of $\cat{A}^{\boxtimes m}$ ranging over all such embeddings for all $m \ge 0$.
	It  comes therefore with a structure map $\varphi_* : \cat{A}^{\boxtimes m}\to \int_\Sigma \cat{A}$ for every such embedding.
	An introduction to factorization homology in the setting of quantum algebra can be found in~\cite[Section~1.1]{bzbj}.
	In the same article, it is also observed that, by the functoriality of factorization homology 
	with respect to oriented embeddings, the embedding $\emptyset \to \Sigma$ induces an object $\qss \in \int_\Sigma \cat{A}$, the so-called  \emph{quantum structure sheaf}
	that is a homotopy fixed point for the action of the diffeomorphism group of $\Sigma$ on $\int_\Sigma \cat{A}$.
	Its endomorphism algebra
	\begin{align} \SkAlg_\cat{A}(\Sigma):= \End_{\int_\Sigma \cat{A}}(\qss) \label{eqnskeinalgebras}
	\end{align} is the \emph{skein algebra} of $\cat{A}$ at $\Sigma$~\cite[Definition~2.6]{skeinfin}.
	This definition is in line with the classical definition of the skein algebra by the results in \cite{cooke}.
	
		\subsection{The monoidal product of cyclic and modular algebras\label{secmonoidalproductofalgebras}}
	Let $\cat{A}$ and $\cat{B}$ be cyclic or modular $\cat{O}$-algebras in a symmetric monoidal bicategory.
	Denote by $\kappa_\cat{A}:\cat{A}\boxtimes \cat{A}\to I$ and $\kappa_\cat{B}:\cat{B}\boxtimes\cat{B}\to I$ the underlying non-degenerate symmetric pairings.
	Then
	\begin{align}
		\kappa_{\cat{A}\boxtimes\cat{B}} : \cat{A}\boxtimes\cat{B} \boxtimes \cat{A}\boxtimes\cat{B} \ra{\substack{\text{symmetric braiding } \\ \text{of second and third factor}}}  \cat{A}\boxtimes\cat{A}\boxtimes\cat{B}\boxtimes\cat{B} \ra{\kappa_\cat{A}\boxtimes \kappa_\cat{B}} I
	\end{align}
	is a non-degenerate symmetric pairing whose coevaluation is
	\begin{align}
		\label{eqndeltaAA}	\Delta_{\cat{A}\boxtimes\cat{B}}: I \ra{\Delta_{\cat{A}}\boxtimes\Delta_\cat{B}} \cat{A}\boxtimes\cat{A}\boxtimes\cat{B}\boxtimes\cat{B} 
		\ra{\substack{\text{symmetric braiding } \\ \text{of second and third factor}}}     \cat{A}\boxtimes\cat{B} \boxtimes \cat{A}\boxtimes\cat{B} \ . 
	\end{align} 
	With the pairing $\kappa_{\cat{A}\boxtimes \cat{B}}$, the object $\cat{A}\boxtimes \cat{B}$
	becomes 
	an $\cat{O}$-algebra that on an operation $o$ of arity $n$ is given by
	\begin{align}
		(\cat{A}\boxtimes \cat{B})_o : (\cat{A}\boxtimes \cat{B})^{\boxtimes n} \ra{\substack{\text{exchange $\cat{A}$ and $\cat{B}$} \\ \text{via symmetric braiding}}} \cat{A}^{\boxtimes n}\boxtimes \cat{B}^{\boxtimes n}\ra{\cat{A}_o \boxtimes \cat{B}_o} I \ . 
	\end{align}
This modular algebra structure on $\cat{A}\boxtimes\cat{B}$ will be implicitly used in the sequel.
	
	\subsection{Calabi-Yau structures}	We will conclude this preliminary section with a reminder on Calabi-Yau structures on finite categories that are studied in a quantum algebra context for instance in \cite{mtrace,shibatashimizu,tracesw}.
	The following statements are all standard; we list them for later reference:
	
	\begin{lemma}\label{lemmacyfin}
		Let $\cat{A}$ be a finite linear category with a Calabi-Yau structure on $\Proj \cat{A}$, i.e.\
		natural isomorphisms $\Psi_{P,Q}:\cat{A}(P,Q)^* \cong \cat{A}(Q,P)$ for all $P,Q \in \Proj \cat{A}$ such that
		$\Psi_{P,Q}^* = \Psi_{Q,P}$ under the identification of a finite-dimensional vector space with its bidual.
		\begin{pnum}
			\item The Calabi-Yau structure extends to isomorphisms $\cat{A}(X,P)^*\cong \cat{A}(P,X)$,
			 where $P$ is projective and $X$ an arbitrary object. \label{lemmacyfini}
			\item The finite category $\cat{A}$ is self-injective, i.e.\ the projective objects coincide with the injective ones. \label{lemmacyfinii}
			\item	Any anti-autoequivalence of $\cat{A}$ preserves projective objects.\label{lemmacyfiniii}
		\end{pnum}
	\end{lemma}

	From Lemma~\ref{lemmacyfin}, one concludes quickly:
	\begin{lemma}\label{lemmanondegenerate}
		For any non-degenerate symmetric pairing $\kappa : \cat{A} \boxtimes \cat{A} \to \vect$ on a finite category $\cat{A}\in\Rexf$, the right exact functor $\kappa(X,-)\cong \kappa(-,X)$ is exact if and only
		if $X$ is projective. If $\Proj \cat{A}$ comes with a Calabi-Yau structure,
		we obtain natural isomorphisms
		\begin{align}
			\kappa(DX,DP)^* \cong  \kappa(P,X) \ , 
		\end{align}
		where $P$ is projective, $X$ an arbitrary object
		and $D: \cat{A}\to \cat{A}^\op$ is the duality induced by $\kappa(X,Y)=\cat{A}(X,DY)^*$ for $X,Y \in \cat{A}$. Moreover, for $X \in \cat{A}$, the following are equivalent:\label{lemmacyfiniv}
		\begin{itemize}
			\item $X$ is projective.
			\item $X$ is injective.
			\item $\kappa(X,-)$  is exact.
			\item $\kappa(-,X)$ is exact. \end{itemize}
	\end{lemma}

		\section{A duality operation for cyclic framed $E_2$-algebras\label{secdual}}
	In this section, we will introduce an important
	 technical tool of this article, a duality operation for cyclic algebras that we will need in Section~\ref{secreflection} to set up the concept of reflection equivariance.
		This operation takes as input a certain type of cyclic framed $E_2$-algebra $\cat{A}$ in $\Rexf$
		(we will see which kind of restriction we will have to make as we go along).
		By what was recalled in Section~\ref{secexplicitribgv}, $\cat{A}$ is a ribbon Grothendieck-Verdier category in $\Rexf$.
		The associated cyclic algebra over the operad of genus zero surfaces is given on a genus zero surface with $n$ boundary circles labeled with $X_1,\dots,X_n$ as follows:
		\begin{align} \Sigma_{0,n} \ \text{plus}\ \text{labels $X_1,\dots,X_n$} \mapsto \cat{A}(X_1\otimes\dots \otimes X_n,K)^* \ . \label{eqnconfblockgenus0}\end{align} Here $K\in \cat{A}$ is the dualizing object. 
		The vector space appearing on the right hand side of~\eqref{eqnconfblockgenus0} is the space of conformal blocks for $\cat{A}$ on the genus zero surface $\Sigma_{0,n}$. 
		The fact that the space of conformal blocks looks like this is a consequence of how the equivalence between cyclic framed $E_2$-algebras and ribbon Grothendieck-Verdier categories is constructed in~\cite[Section~5]{cyclic}. Note that we are treating here the $\Rexf$-valued case, not the $\Lexf$-valued one as in~\cite{cyclic}.
		
		The duality operation $\cat{A}\mapsto \cat{A}\dg$
		 that we want to define in this section takes unsurprisingly the linear dual of the spaces of conformal blocks while simultaneously dualizing the labels. 
		This means that we replace
		$\cat{A}(X_1\otimes\dots \otimes X_n,K)^*$
		with  
		$\cat{A}(DX_n\otimes\dots \otimes DX_1,K)$,
		where $D$ denotes the Grothendieck-Verdier duality.
		The order is reversed because we have to apply an iterated duality on Deligne tensor powers of $\cat{A}$. We will formalize this in  Definition~\ref{defdgr} below. 
		There is  the obvious problem that 
		$\cat{A}(DX_n\otimes\dots \otimes DX_1,K)$
		is not even right exact in the labels $X_i$, which is needed to get a cyclic $\framed$-algebra in $\Rexf$. 
		The construction presented in this section will resolve this issue under a few technical assumptions on $\cat{A}$.
		
		\subsection{Left and right exact extensions, self-injectivity}
		In this subsection, we collect some facts about Nakayama functors of finite categories. 
		For  a finite category $\cat{A}$,
		recall that with the definition of the right Nakayama functor $\nakar : \cat{A}\to \cat{A}$
		via the coend $\nakar = \int^{X \in\cat{A}} \cat{A}(-,X)^* \otimes X$
		and the left Nakayama functor via the end $\nakal = \int_{X \in \cat{A}} \cat{A}(X,-)\otimes X$
		\cite[Section~3.5]{fss},
		one can observe \cite[Lemma~4.7]{shibatashimizu}, see also \cite[Corollary~2.3]{tracesw}, that one obtains natural isomorphisms
		\begin{align}
			\cat{A}(P,X)^*&\cong \cat{A}(X,\nakar P)  \quad  \quad \text{for}\quad P \in \Proj\cat{A} \ , \quad X \in \cat{A} \ \ ,   \label{eqncystructurewithN} \\
			\cat{A}(X,J)^*&\cong \cat{A}(\nakal J,X) \quad \quad \text{for} \quad J\in\Inj\cat{A}\ , \quad X \in \cat{A} \ . \label{eqncystructurewithNN}
		\end{align}
		Let $F:\cat{A}\to \cat{B}$ be a
		right exact functor between finite categories $\cat{A}$ and $\cat{B}$.
		The functor \begin{align}
			X \mapsto \int_{J\in\Inj \cat{A}} \cat{A}(\nakal J,X) \otimes FJ\end{align} is left exact as functor $\cat{A} \to \cat{B}$. Thanks to \eqref{eqncystructurewithNN} and the Yoneda Lemma, it is the essentially unique left exact functor agreeing with $F$ on injective objects. 
		Then the functor
		\begin{align}
			\label{eqnwidetildeF}	\widetilde F : \cat{A}^\op \to \cat{B}^\op \ , \quad X \mapsto \int_{J\in\Inj \cat{A}} \cat{A}(\nakal J,X) \otimes FJ
		\end{align}
		is right exact and
		\begin{align}
			\widetilde{F}(J)\cong F(J)\label{eqntildeonproj}
		\end{align}
		by a natural isomorphism
		for any injective object $J \in \cat{A}$. 
		A 2-isomorphism $\alpha : F \to G$ between right exact functors $F,G:\cat{A}\to\cat{B}$ induces a 2-isomorphism
		$\widetilde G \to \widetilde F$, i.e.\
		$\alpha^{-1}$ induces a 2-isomorphism $\widetilde \alpha : \widetilde F \to \widetilde G$. 
		These assignments assemble into
		a  functor \begin{align}\widetilde{-}:\Rexf(\cat{A},\cat{B})\to \Rexf(\cat{A}^\op,\cat{B}^\op) \label{eqnwidetildelexf}\end{align} 
		natural in $\cat{A}$ and $\cat{B}$.
		Here $\Rexf$ is treated as $(2,1)$-category, i.e.\ we only allow 2-morphisms that are isomorphisms.

			\begin{remark}\label{remnakal} 
				Recall that a finite category $\cat{A}$ is called \emph{self-injective} if 
				 the classes of projective and injective objects coincide. By \cite[Proposition~IV.3.1]{ars} and~\cite[Proposition~3.24]{fss} the self-injectivity of $\cat{A}$ is equivalent to $\nakal$ being an equivalence whose quasi-inverse is $\nakar$.
		\end{remark}

		\begin{lemma}\label{lemmawidetilde}
			The subcategory $\Rexfsi\subset \Rexf$ of self-injective finite categories is the largest monoidal
			sub\-cate\-gory such that the functors~\eqref{eqnwidetildelexf}, when restricted to this subcategory, are equivalences.
			After restriction to self-injective finite categories, $\widetilde{-}$ consists of homotopy involutions.
		\end{lemma}

		\begin{proof}
			The largest monoidal subcategory $\cat{U} \subset \Rexf$ on which~\eqref{eqnwidetildelexf} is an equivalence is exactly the one for which the functors 
			$\varphi_{\cat{A},\cat{B}}:	\Rexf(\cat{A},\cat{B}) \to \Lexf(\cat{A},\cat{B})$, that 
			agree with~\eqref{eqnwidetildelexf}, but without taking the opposite category in the last step,
			are equivalences. 
			Since $\vect \in \cat{U}$, we need the functor
			\begin{align}
				\varphi_{\cat{A},\vect}:	\Rexf(\cat{A},\vect) \to \Lexf(\cat{A},\vect) \quad \text{for}\quad \cat{A}\in\cat{U}
			\end{align}
			to be an equivalence.
			First we observe for  $J,J'\in \Inj \cat{A}$
			\begin{align} \varphi_{\cat{A},\vect} \left(      \cat{A}(-,J)  ^*     \right) (J') = \cat{A}(J',J)^*\cong \cat{A}(\nakal J,J')
				\ , \end{align}
			and hence
			\begin{align} \varphi_{\cat{A},\vect} \left(      \cat{A}(-,J)^*       \right)  
				\cong \cat{A}(\nakal J,-) \ . \end{align}But this implies that under the 
			equivalence $\cat{A}^\op \simeq \Rexf(\cat{A},\vect)$ sending $X$ to $\cat{A}(-,X)^*$ and the equivalence $\cat{A}^\op \simeq \Lexf(\cat{A},\vect)$ sending $X$ to $\cat{A}(X,-)$, the equivalence $\varphi_{\cat{A},\vect}$ is an equivalence $\cat{A}^\op \to \cat{A}^\op$ sending $J \in \Inj \cat{A}$ to $\nakal J$. The equivalence
			$\varphi_{\cat{A},\vect}$ preserves projective and injective objects. This implies
			$\nakal (\Inj \cat{A})\subset \Inj \cat{A}$. The Nakayama functor $\nakal$ induces an equivalence $\Inj \cat{A} \simeq \Proj \cat{A}$ \cite[Section~3]{ivanov}. Let now $P \in \cat{A}$ be projective, i.e.\ $P\cong \nakal J$ for some injective object $J$. Then by what we just observed $P$ is also injective, i.e.\ $\Proj \cat{A}=\Inj \cat{A}$, making $\cat{A}$ self-injective. This shows  $ \cat{U}\subset \Rexfsi$.
			
			In fact, we also have $\Rexfsi \subset \cat{U}$ because
			the map $\Rexf(\cat{A},\cat{B}) \to \Lexf(\cat{A},\cat{B})$ for $\cat{A},\cat{B} \in \Rexfsi$ is the equivalence
			\begin{align}
				\Rexf(\cat{A},\cat{B}) \simeq \catf{Lin}(\Proj \cat{A},\cat{B})=\catf{Lin}(\Inj \cat{A},\cat{B})\simeq \Lexf(\cat{A} , \cat{B})\ ,
			\end{align}
		where $\catf{Lin}(-,-)$ denotes the category of linear functors.
			This is true because $\cat{A}$ is the finite free cocompletion of $\Proj \cat{A}$ and the finite free completion of $\Inj \cat{A}$; for these standard facts, see \cite{daylack} or, specifically in the finite setting, \cite[Section~5]{sn}. 
			This shows that the maximal monoidal subcategory on which~\eqref{eqnwidetildelexf} is an equivalence is $\Rexfsi$.
			
			The equivalence $\widetilde{-}:\Rexf(\cat{A},\cat{B})\to \Rexf(\cat{A}^\op,\cat{B}^\op)$
			takes a right exact functor $F:\cat{A}\to\cat{B}$ between self-injective finite categories restricts it to $\Inj \cat{A}=\Proj \cat{A}$ and extends it back, but as a left exact functor $\cat{A}\to\cat{B}$ which is then seen as right exact functor $\cat{A}^\op \to \cat{B}^\op$. From this description, it is clear that the functors~\eqref{eqnwidetildelexf} are homotopy involutions.
		\end{proof}

		\subsection{The $\dgr$-operation on cyclic framed $E_2$-algebras}
		We can now define the $\dgr$-operation on cyclic framed $E_2$-algebras subject to a few technical conditions.

		\begin{definition}\label{defdgr}
			Denote by \begin{align} \E := \begin{array}{c}\text{full 2-groupoid of $\Rexf$-valued cyclic framed $E_2$-algebras} \\ \text{whose underlying finite category is self-injective} \\ \text{and whose monoidal product is exact.}\end{array} \end{align} 
			We will treat throughout a version of $\framed$ with $\framed(\bullet)=\emptyset$, i.e.\ we do not include operations of total arity zero (this just excludes the closed sphere, and is not an issue because once an operation is defined everywhere else, we can extend it uniquely to the sphere). 
			Let $\cat{A}\in \E$. Denote by $D:\cat{A}\to\cat{A}^\op$ the underlying duality.
			We define 
			$\cat{A}\dg$  
			on	 an operation $o$ of total arity $n \ge 1$ 
			to be the right exact functor
			\begin{align}
				\cat{A}\dg_o : \cat{A}^{\boxtimes n}   \ra{D_n^\text{rev}} \left(   \cat{A}^\op       \right)^{\boxtimes n} \simeq\left( \cat{A}^{\boxtimes n}\right)^\op \ra{\widetilde{\cat{A}_o}} \vect^\op \ra{*}\vect \ , 
			\end{align} where
		$D_n^\text{rev}$ is defined via
			\begin{align}
			D_n^\text{rev} : \cat{A}^{\boxtimes n} \ra{\substack{\text{invert the order}\\\text{via symmetric braiding}}} \cat{A}^{\boxtimes n} \ra{D ^{\boxtimes n}}(\cat{A}^\op)^{\boxtimes n}\ , \quad X_1\boxtimes \dots\boxtimes X_n \mapsto DX_n \boxtimes \dots\boxtimes DX_1  \ . 
		\end{align}
			In other words,
			\begin{align} \cat{A}_o\dg(P):=\cat{A}_o(D_n^\text{rev} P)^* \quad \text{for}\quad P \in \Proj \cat{A}^{\boxtimes n} \ . 
			\end{align}
			We call $\cat{A}\dg$ 
			the \emph{dual} of $\cat{A}$.
		\end{definition}
	
	\begin{remark}\label{remdnaeq}Denote the internal hom of $\Rexf$ by $\Rexf[-,-]$.
		The equivalence
		\begin{align}
			\Rexf[\cat{A}^{\boxtimes n},\vect]\ra{\simeq} \Rexf[\vect,\cat{A}^{\boxtimes n}]\simeq \cat{A}^{\boxtimes n}      \label{eqndrev}
			\end{align}
		induced by the cyclic structure of $\cat{A}$ becomes $\mathbb{Z}_2$-equivariant if $\Rexf[\cat{A}^{\boxtimes n},\vect]$ is equipped with $\dgr$ and $\cat{A}^{\boxtimes n}$ with $D_n^\text{rev}$. 	
		Indeed, to describe for $F\in \Rexf[\cat{A}^{\boxtimes n},\vect]$
		the object in $\cat{A}^{\boxtimes n}$
	that	 $F\dg$ is sent to, we may assume that $F$ is projective in $\Rexf[\cat{A}^{\boxtimes n},\vect]$, i.e.\ exact. This is enough because the equivalence will preserve colimits.
		We denote the $i$-th copy of $\Delta$ by $\Delta_i = \int_{P\in\Proj\cat{A}} DP\boxtimes P = \Delta_i'\boxtimes \Delta_i''$ in Sweedler notation, see \cite[Proposition~5.1.7]{kl} for why the reduction of the coend to projective objects is possible.
		Then
		\begin{align}
			F\dg (\Delta'_1,\dots,\Delta_n') \otimes \Delta_n''\boxtimes \dots \boxtimes \Delta_1'' &= F (D\Delta'_n,\dots,D\Delta_1')^* \otimes \Delta_n''\boxtimes \dots \boxtimes \Delta_1''\\ &\cong 
			F (\Delta'_n,\dots,\Delta_1')^* \otimes D\Delta_n''\boxtimes \dots \boxtimes D\Delta_1'' \\ &\phantom{\cong}
		\quad 	\text{(exactness of $F$, symmetry of $\Delta$)} \\&\cong D_n^\text{rev} \left( 	F (\Delta'_n,\dots,\Delta_1') \otimes \Delta_1''\boxtimes \dots \boxtimes \Delta_n''  \right) \ . 
			\end{align}
		This proves the claim.
		\end{remark}

		It is not obvious that $\cat{A}\dg$ is a cyclic $\framed$-algebra. To see this, we will need some additional considerations:

		\begin{lemma}\label{lemmamonideal}
			Let $\cat{A}$ be a 
			pivotal Grothendieck-Verdier category in $\Rexf$. Suppose that $\cat{A}$ is self-injective and with exact monoidal product.
			Then  $\Proj \cat{A}$ is a monoidal ideal, i.e.\ $P\otimes X$ and $X\otimes P$ are projective for any $X\in \cat{A}$
			whenever $P$ is projective. 
		\end{lemma}
		
		\begin{proof}
			For $P \in \Proj \cat{A}$ and $X,Y \in \cat{A}$,
			we have by $\otimes$-invariance of $\kappa$ and \eqref{eqncystructurewithN}
			\begin{align}
				\kappa(P\otimes Y,X)\cong \kappa(P, Y \otimes X) \cong \cat{A}(P,D(Y\otimes X))^*\cong \cat{A}(D(Y\otimes X),\nakar P) \cong \cat{A}(D\nakar P , Y \otimes X) \ . 
			\end{align}
			First note that $D\nakar P$ is projective. Since $\otimes$ is exact, this allows us to conclude that $ \cat{A}(D\nakar P , Y \otimes -):\cat{A} \to \vect$
			is exact. 
			This makes $\kappa(P\otimes Y,-)$  exact, and therefore $P\otimes Y$ projective by Lemma~\ref{lemmanondegenerate}.
			Similarly, one sees that $Y\otimes P$ is projective. 
		\end{proof}

		We can now establish
		that $\cat{A}\dg$ is again a cyclic framed $E_2$-algebra in $\Rexf$ if we assume $\cat{A}\in\E$.
		By definition its non-degenerate symmetric pairing is \begin{align}\label{eqnkappadg} \kappa\dg (X,Y)=\kappa(\nakal X, Y)\end{align} as follows directly from the definitions. Its coevaluation object, as follows from Remark~\ref{remdnaeq}, is $D_2 \Delta$, i.e.\ the coend $\int^{X \in \cat{A}} DX \boxtimes X=\int^{P\in\Proj\cat{A}} DX\boxtimes X$. 
		The only non-obvious part is the compatibility with gluing.
		It suffices to consider the case where operations $o$ and $p$ of arity $m\ge 1$ and $n\ge 1$, respectively, are glued together to form an operation $q$ of arity $m+n-2$. Since we must have $m+n-2\ge 1$, it follows that $m$ or $n$ is at least 2; without loss of generality, we assume $m\ge 2$.
		Now we observe with Sweedler-like notation $D_2\Delta = \int^{S \in \Proj\cat{A}} DS \boxtimes S=(D_2\Delta)'\boxtimes (D_2\Delta)''$ for the dual of the coevaluation object $\Delta$
		(we can assume that the respective last arguments are affected by the gluing)
		\begin{align}
			&\cat{A}\dg_o (    P_1 \boxtimes \dots \boxtimes P_{m-1} \boxtimes (D_2\Delta)'   ) \otimes \cat{A}\dg_p (Q_1\boxtimes\dots\boxtimes Q_{n-1} \boxtimes (D_2 \Delta)'') \\ \cong& \int^{S \in \Proj \cat{A}} 
			\cat{A}\dg_o (    P_1 \boxtimes \dots \boxtimes P_{m-1} \boxtimes DS   ) \otimes \cat{A}\dg_p (Q_1\boxtimes\dots\boxtimes Q_{n-1} \boxtimes S) \\ \quad \text{with}\quad & P_1,\dots,P_{m-1},Q_1,\dots,Q_{n-1} \in \Proj \cat{A}
		\end{align}
		because $\cat{A}\dg_o$ and $\cat{A}\dg_p$ are right exact and hence preserve finite coends.
		With the definition of $\cat{A}\dg$ and \eqref{eqnconfblockgenus0}, we arrive at
		\begin{align}
			&\cat{A}\dg_o (    P_1 \boxtimes \dots \boxtimes P_{m-1} \boxtimes (D\Delta)'   ) \otimes \cat{A}\dg_p (Q_1\boxtimes\dots\boxtimes Q_{n-1} \boxtimes (D \Delta)'') \\ \cong&
			\int^{S \in \Proj \cat{A}} 
			\cat{A} (S \otimes DP_{m-1} \otimes \dots \otimes DP_1  ,K ) \otimes \cat{A} (   DS \otimes DQ_{n-1}\otimes \dots \otimes D Q_1,K)
			\\ \cong&
			\int^{S \in \Proj \cat{A}} 
			\cat{A} \left(S, \underbrace{D(DP_{m-1} \otimes \dots \otimes DP_1)}_{\substack{\text{projective because $m\ge 2$} \\   \text{and because of Lemma~\ref{lemmamonideal}}  }}   \right) \otimes \cat{A} (   DS , D( DQ_{n-1}\otimes \dots \otimes D Q_1))
			\\ \cong&
			\cat{A} (   DP_{m-1} \otimes \dots \otimes DP_1 , D( DQ_{n-1}\otimes \dots \otimes D Q_1))
			\\ \cong&
			\cat{A} (   DP_{m-1} \otimes \dots \otimes DP_1 \otimes DQ_{n-1}\otimes \dots \otimes D Q_1,K) \\\cong &\cat{A}\dg_q (Q_1 \boxtimes \dots \boxtimes Q_{n-1}\boxtimes P_1 \boxtimes \dots \boxtimes P_{n-1}) \ , \label{eqncyclicalgdgrcomp}
		\end{align}
		where we have used 
		the Yoneda Lemma.
		This establishes the gluing property.
		Therefore $\dgr$ is a well-defined operation on cyclic framed $E_2$-algebras with values in self-injective finite categories and with exact monoidal product.
		Since $\widetilde{-}$, $D$ and the dual of finite-dimensional vector spaces square to the identity up to a canonical isomorphism, we obtain in summary:

		\begin{proposition}\label{propdginv}
			Through $\cat{A}\mapsto \cat{A}\dg$, we obtain a homotopy involution
			on the 2-groupoid $\E$, with the coherence data of the homotopy coherent $\mathbb{Z}_2$-action induced by
			the trivialization $\omega : \id_\cat{A}\to D^2$ of $D^2$ coming from the symmetry of the pairing $\kappa$ underlying $\cat{A}$, the coherence data of the homotopy involutions~\eqref{eqnwidetildelexf}
			 and the natural isomorphism $V^{**}\cong V$ for finite-dimensional vector spaces $V$.
		\end{proposition}
		
		\begin{remark}\label{remcycassdgr} The considerations above show that the
			$\dgr$-operation cannot only be defined for cyclic framed $E_2$-algebras, but also cyclic associative
			algebras if self-injectivity and the exactness of the monoidal product are assumed. 
			Proposition~\ref{propdginv} remains correct in that case.
			More generally, a  version of Proposition~\ref{propdginv}
			can be stated for other cyclic operads such that the $\dgr$-operation can be defined.
			In the present paper, we have no application for this and will not pursue this further.
		\end{remark}

		\begin{corollary}\label{cordgrallobject}
			Let $\cat{A}\in \E$.
			The isomorphism \begin{align} \cat{A}_o\dg(P)\cong \cat{A}_o(D_n^\text{rev}P)^*\label{eqnforallobj}\end{align} extends
			from projective objects to all objects of $\cat{A}$ for all operations $o$ 
			if and only if $\cat{A}$ is semisimple.
		\end{corollary}
		
		\begin{proof}
			If $\cat{A}$ is semisimple, i.e.\ if all objects of $\cat{A}$ are projective, then~\eqref{eqnforallobj} clearly extends to all objects.
			
			If conversely~\eqref{eqnforallobj} holds for all objects, then in particular $X \mapsto \cat{A}(DX,K)\cong \cat{A}(I,X) $ by virtue of being isomorphic to $\cat{A}\dg( \Sigma_{0,1};-)$ needs to be right exact. But this implies that $\cat{A}(I,-)$ is  exact and that $I$ is projective.
			Now every $X\in \cat{A}$ is projective by Lemma~\ref{lemmamonideal} since $X\cong X \otimes I$.
		\end{proof}

		\begin{remark}
			\label{remdgrnonmod}
			One could ask why we do not define $\cat{A}\dg$ without any restrictions on the total arity being at least one, or even directly for arbitrary modular algebras.
			This is because of the following problems: Suppose that $\cat{A}$ is a modular $\cat{O}$-algebra and $o \in \cat{O}(\bullet)$, where $\bullet$ is the corolla without legs. In this case, $\cat{A}_o$ is simply a vector space (because it is a right exact functor $\vect \to \vect$), and Definition~\ref{defdgr}, when extended to that case, would give us the dual vector space $\cat{A}\dg_o=\cat{A}_o^*$. But now imagine that $o$ is obtained by taking the trace of another operation $p \in \cat{O}(T)$, where $T$ is a corolla with two legs (in the example $\cat{O}=\Surf$, the operation $p$ could be $\mathbb{S}^1 \times [0,1]$, and $o$ would then be $\mathbb{T}^2$). If $\cat{A}\dg$, when directly defined by extending Definition~\ref{defdgr} to all operations, were a modular algebra, then we would have $\cat{A}\dg_o = \cat{A}\dg_p (D_2 \Delta)$ by excision. On the other hand, this needs to be canonically
			isomorphic to $\cat{A}_o^*$.
			It would be tempting to calculate $\cat{A}\dg_p (D_2 \Delta)$ by simply applying $D$ to the argument and taking the dual vector space which indeed would give us $\cat{A}_p(\Delta)^* \cong \cat{A}_o^*$.  But we are not allowed to do this because $\Delta \in \cat{A}\boxtimes \cat{A}$ is generally not a projective object!
			In the semisimple case, this issue would of course be resolved, but in the non-semisimple case, we cannot define the $\dgr$-operation directly on cyclic and modular algebras \emph{without} any restriction on the arity.
		\end{remark}
		
		\subsection{A reminder on ansular functors\label{secansular}}
		If we replace in the definition of the surface operad all surfaces with parametrized boundary by three-dimensional handlebodies with parametrized oriented disks embedded in the boundary (again, we remain in the oriented and compact setting; the embeddings of the disks are always orientation-preserving), one obtains the modular operad $\Hbdy$ of handlebodies. 
		In \cite{mwansular} modular $\Hbdy$-algebras
		in a symmetric monoidal bicategory $\cat{S}$
		are given the name \emph{ansular functor}, and it is proven, building on the results of \cite{giansiracusa,cyclic},
		that the genus zero restriction from ansular functors to cyclic framed $E_2$-algebras (this makes sense because the genus zero part of $\Hbdy$ is equivalent, as cyclic operad, to $\framed$) is an equivalence. Its inverse is the modular extension:
		\begin{equation}\label{eqnequivcycframed0}
			\begin{tikzcd}
				\catf{CycAlg}(\framed) \ar[rrrr, shift left=2,"\text{modular extension}"] &&\simeq&& \ar[llll, shift left=2,"\text{restriction}"] 
				\catf{ModAlg}\left( \Hbdy \right) \quad \text{for bicategorical algebras.} 
			\end{tikzcd}
		\end{equation} 
		We denote the modular extension of $\cat{A}\in \catf{CycAlg}(\framed)$ by $\widehat{\cat{A}}$.
		The ansular functor associated to a cyclic framed $E_2$-algebra can be concretely described~\cite{cyclic,mwansular}. For a ribbon category, there is even a description in terms of admissible skeins~\cite{asm} as shown in~\cite{mwskein}.

		\subsection{The $\dgr$-operation on ansular functors\label{secansularfunctorrefl}}	Suppose that $\cat{B}$ is a modular $\Hbdy$-algebra in $\Rexf$, i.e.\ an ansular functor, whose underlying cyclic framed $E_2$-algebra (the one obtained by genus zero restriction) is in $\E$.
			Then we would like to define $\cat{B}\dg$, but we cannot na\"ively extend Definition~\ref{defdgr} because of the problems explained in Remark~\ref{remdgrnonmod}.     
			We can however restrict $\cat{B}$ to genus zero and obtain cyclic framed $E_2$-algebra $\cat{A}$ that lies in $\E$ by definition.
			Then we can define $\cat{B}\dg$
			as the unique ansular functor extending $\cat{A}\dg$, i.e.\ we set $\cat{B}\dg := \widehat{   \cat{A}\dg  }$, where $\widehat{-}$ denotes the modular extension.
			This makes the equivalence~\eqref{eqnequivcycframed0}
			 equivariant with respect to $\dgr$ once we restrict to self-injective underlying categories and exact monoidal products.	
			We have to underline that $\cat{B}\dg$, when defined this way, 
			cannot necessarily be described
			by extending Definition~\ref{defdgr} to modular algebras --- simply because it would not be clear that this produces a modular algebra
			as we saw in Remark~\ref{remdgrnonmod}.
			Nonetheless, with a computation similar to~\eqref{eqncyclicalgdgrcomp}, one can show that
			\begin{align}
				\cat{B} \dg (H;P) \cong \cat{B}(H;D_n^\text{rev}P)^* \quad \text{for}\quad P \in \Proj \cat{A}^{\boxtimes n} \ ,       \label{eqndgransular}
			\end{align} where $H$ has at least one embedded disk per connected component.
		
		\needspace{15\baselineskip}
		
		\section{Reflection equivariance\label{secreflection}}
		
		Using the definition of the $\dgr$-operation, we can now define the notion of reflection equivariance.

		\subsection{Homotopy involutions of cyclic and modular operads}
		First we collect some facts on homotopy involutions of modular operads.
		
		\begin{definition}
			An \emph{involution} on a cyclic (or non-cyclic or modular) operad $\cat{O}$ in a (higher) symmetric monoidal category is a $\mathbb{Z}_2$-action on $\cat{O}$ up to coherent isomorphism.
			If the generator of $\mathbb{Z}_2$ is sent to the automorphism $J:\cat{O}\to\cat{O}$ of cyclic (or non-cyclic or modular) operads, we say that $\cat{O}$ is a \emph{cyclic (or non-cyclic or modular) operad with involution $J$}.
		\end{definition}
		
		The following is a straightforward observation:
		\begin{lemma}	The reflection operation (or also called orientation reversal operation) $\refl : \Sigma \mapsto \bar \Sigma$
			yields a homotopy  involution
			on \begin{itemize} \item the cyclic framed $E_2$-operad $\framed$, 
				\item 	the modular handlebody operad $\Hbdy$, 
				\item 	and the modular surface operad $\Surf$ \end{itemize}
			that is compatible with the genus zero restriction and the boundary map $\partial : \Hbdy \to \Surf$ in the sense that the genus zero restriction and the boundary map are $\mathbb{Z}_2$-equivariant up to coherent isomorphism.\end{lemma}

		For a  (non-)cyclic or modular algebra $\cat{A}$ over $\framed$, $\Hbdy$ 
		or $\Surf$, we denote the restriction of $\cat{A}$ along the reflection by $\bar{\cat{A}}$.
		
		\begin{remark}
			The reflection operation on algebras
			clearly makes the modular extension procedure 
			\begin{align}
			\catf{CycAlg}(\framed)\simeq \catf{ModAlg}(\Hbdy)\end{align} from~\eqref{eqnequivcycframed0}
			also $\mathbb{Z}_2$-equivariant. 
		\end{remark}
		
		\begin{remark}\label{rembarA}
			For a framed $E_2$-algebra $\cat{A}$, a concrete description of $\bar{\cat{A}}$ is well-known, see e.g.~the explanations in~\cite[Section~2.1]{bjss}: It can be modeled as the same underlying $E_1$-algebra, but with inverse braiding and inverse balancing (and unless otherwise stated, we will refer to this model when using the notation $\bar{\cat{A}}$). A different model is $\cat{A}^{\otimes \op}$ that just equips $\cat{A}$ with the opposite monoidal product.
			These two models correspond to 
			the reflections
			along two different axes; they are
			 equivalent by a 180 degree rotation of the large disk containing the little disks.
		\end{remark}

		\subsection{Homotopy fixed points}
		The orientation reversal can be combined with the $\dgr$-operation.
		What we will be looking for are the homotopy fixed points of this combined operation.
		\begin{definition}\label{defhomotopyfixedpoint}
			Let $\cat{O}$ be a cyclic (or non-cyclic) $\Grpd$-valued operad
			such that the $\dgr$-operation on a subcategory $\E$ of its algebras with values in self-injective finite linear categories is well-defined.
			For $\cat{O}=\framed$,  we fix $\E$ to be the subcategory of self-injective  categories with exact monoidal product from Definition~\ref{defdgr}.
			Let $J$ be an involution on $\cat{O}$ 
			and $\cat{A}$ a cyclic (or non-cyclic) $\Rexf$-valued $\cat{O}$-algebra.
			A \emph{$\dgr$-$J$-homotopy fixed point structure on $\cat{A}$}
			\begin{pnum}
				\item includes the assumption that $\cat{A}$ is in the subcategory $\E$, and hence in particular self-injective \label{htpfixedpoint1}
				(this is not an additional assumption, but needed for~\ref{htpfixedpoint2} to make sense),
				\item consists of
				the structure of a homotopy $\mathbb{Z}_2$-fixed point on $\cat{A}$ for the $\mathbb{Z}_2$-action $\cat{B}\mapsto J^*\cat{B}\dg$ on $\E$ (here $J^*$ is the restriction along $J$).
				\label{htpfixedpoint2}
			\end{pnum}
		\end{definition}
		
		\begin{remark}Note that $J^* \cat{B}\dg \simeq (J^* \cat{B})\dg$ by a canonical equivalence. Therefore, the fact that $\cat{B} \mapsto J^* \cat{B}\dg$ is a homotopy involution follows from Proposition~\ref{propdginv}.
		\end{remark}
		
	\subsection{Connection to projective Calabi-Yau structures}
		If we evaluate a homotopy fixed point structure as introduced in Definition~\ref{defhomotopyfixedpoint} at the operadic identity,
		we obtain a Calabi-Yau structure on the projective objects:
		
		\begin{proposition}\label{propcy}
			Let $\cat{A}$ be a cyclic  $\Rexf$-valued
			$\cat{O}$-algebra with $\dgr$-$J$-homotopy fixed point structure
			with respect to a homotopy involution $J$ of $\cat{O}$ as in Definition~\ref{defhomotopyfixedpoint}.
			The evaluation of the homotopy fixed point
			structure at the operadic identity amounts exactly natural isomorphisms \begin{align}
				\label{eqnvarphiiso}	\varphi_{P,Q} : \kappa(DQ,DP)^* \to  \kappa(P,Q)\quad \text{for}\quad P,Q \in \Proj \cat{A}
			\end{align}
			for the pairing $\kappa$ of $\cat{A}$ 
			making the square	
			\begin{equation}\label{eqncysymcond}
				\begin{tikzcd}
					\kappa (P,Q)^*   \ar[]{rr}{\omega} \ar[swap]{dd}{\varphi_{P,Q}^*} & &  	\kappa (D^2 P, D^2 Q)^* \ar{dd}{   \varphi_{DQ,DP}  }    \\
					\\
					\kappa (DQ,DP)^{**} \ar[]{rr}{\cong}	& &  
					\kappa (DQ,DP) 
				\end{tikzcd} 
			\end{equation}
			commute. In other words,
			the isomorphisms~\eqref{eqnvarphiiso} induce a Calabi-Yau structure on $\Proj \cat{A}$. 
		\end{proposition}

		\begin{proof} First of all, $\Proj \cat{A}=\Inj \cat{A}$ by the assumption~\ref{htpfixedpoint1} in Definition~\ref{defhomotopyfixedpoint}. This implies that $D$ preserves projective objects.
			Next observe that
			we have an equivalence
			\begin{align}
				\cat{A}\dg \simeq J^* \cat{A} \ , 
			\end{align}
			so in particular
			for every operation $o\in \cat{O}(T)$, with $T$ a corolla, an isomorphism
			\begin{align}
				\cat{A}_o\dg \cong \cat{A}_{J(o)} \ . 
			\end{align}
			The evaluation at the operadic identity refers to the special case where
			$T$ is the corolla with two legs and $1_\cat{O}$ the operadic identity. In that case, we
			have by the unitality of $\cat{A}$ that $\cat{A}_{1_\cat{O}}\cong \kappa$,
			 and by unitality of $J$ that $J(1_\cat{O})\cong 1_\cat{O}$ by a canonical isomorphism. Hence, we obtain with the definition of $\cat{A}\dg$ 
			exactly
			natural isomorphisms 
			\begin{align}
				\varphi_{P,Q}:	\kappa (DQ,DP)^* \cong \kappa (P,Q) \quad \text{for} \quad P, Q \in \Proj \cat{A} \ . \label{eqnvarphicy}
			\end{align}
			The fact that~\eqref{eqnvarphicy} is part of a $\mathbb{Z}_2$-fixed point structure 
			gives us exactly the commutativity of~\eqref{eqncysymcond} as follows from
			Proposition~\ref{propdginv}.
		\end{proof}

		\subsection{The dual cyclic algebra and Boyarchenko-Drinfeld's second monoidal product}
		Let $\cat{A} \in \E$, i.e.\ a ribbon Grothendieck-Verdier category in $\Rexf$ such that $\cat{A}$ is self-injective and $\otimes :\cat{A}\boxtimes\cat{A}\to\cat{A}$ exact.
		Following \cite[Section~4.1]{bd}, we can define on $\cat{A}$ a \emph{second (or dual) monoidal product} by
		\begin{align}
			X \odot Y := D (DY \otimes DX) \quad \text{for}\quad X,Y \in \cat{A} 
		\end{align}
		with dualizing object $I$ and monoidal unit $K$.
		Since $\cat{A}$ is pivotal, we do not make a distinction between $D$ and $D^{-1}$ here.
		Also note that if $\cat{A}$ were just any 
		ribbon Grothendieck-Verdier category in $\Rexf$, the product $\odot$ would be left exact.
		In order to not leave the symmetric monoidal bicategory $\Rexf$, we need $\otimes$ to be exact. 
		If $D$ is a rigid duality, $\otimes\cong \odot$ by a canonical monoidal isomorphism.
		Conversely, if a \emph{specific comparison map} between $\otimes$ and $\odot$ happens to be an isomorphism, the category is rigid~\cite[Proposition~4.4]{bd}.
		The dual monoidal product $\odot$ features in the following concrete characterization of $\cat{A}\dg$:
		\begin{proposition}\label{proprevor}
			Let $\cat{A} \in \E$.
			Then the dual cyclic algebra $\cat{A}\dg$ is equivalent to  \begin{pnum}
				\item $\cat{A}$ equipped with the  dual monoidal product $\odot$ characterized by $D(X\otimes Y)=DY \odot DX$,\label{proprevori}
				\item  dualizing object $\nakar I\in\cat{A}$, monoidal unit $K$, Grothendieck-Verdier duality $D\dg = \nakar D$,
				\item and with the braiding and balancing that is induced by the inverse braiding and the inverse balancing of $\cat{A}$.\label{proprevoriii} \end{pnum}
		\end{proposition}

		\begin{proof}
			If we evaluate $\cat{A}\dg$ on a genus zero surface $\Sigma_{0,n}$
			with $n$ boundary components, we obtain the functor $\cat{A}^{\boxtimes n}\to \vect$ that is described on $P=P_1 \boxtimes \dots \boxtimes P_n \in\Proj \cat{A}^{\boxtimes n}$
			(this notation should not imply that $P$ is a pure tensor; it is to be understood as the usual Sweedler-type notation) by
			\begin{align}
				\cat{A}\dg (\Sigma_{0,n};P) = \cat{A}(  DP_n \otimes \dots \otimes DP_1 ,K ) \ , \label{eqndgconcrete}
			\end{align}
			see \eqref{eqnconfblockgenus0}. 
			The braiding and the balancing generators act again by the inverse braiding and the inverse balancing of $\cat{A}$ by definition of $\cat{A}\dg$
			(they act in the contragredient representation, even though we do not see a star `$*$' in~\eqref{eqndgconcrete} because the original hom space in~\eqref{eqnconfblockgenus0}
			already had one).
			
			Moreover, $D(\Proj \cat{A})\subset \Proj \cat{A}$ by self-injectivity and hence $DP_n \otimes \dots \otimes DP_1 \in \Proj \cat{A}$ by Lemma~\ref{lemmamonideal}.
			Since $DP_n \otimes \dots \otimes DP_1 \cong D(P_1\odot \dots \odot P_n)$, the object $P_1 \odot \dots \odot P_n$ is also projective.
			We conclude that
			\begin{align}	\cat{A}\dg (\Sigma_{0,n};P) &\cong \cat{A}(  D(P_1\odot \dots \odot P_n) ,K )
				\\&\cong \cat{A}( I, P_1\odot \dots \odot P_n       )\\&\cong \cat{A} (       \nakal (    P_1\odot \dots \odot P_n   )   ,I    )^*\\&\cong \cat{A}(P_1\odot \dots \odot P_n ,\nakar I )^* \ .  \label{eqnrefllemmapx}  \end{align} 
			From~\eqref{eqnrefllemmapx} we can, again by invoking \eqref{eqnconfblockgenus0}, read off the monoidal product of $\cat{A}\dg$ to be $\odot$ (the monoidal unit is therefore $K$), and the dualizing object to be $\nakar I$. The Grothendieck-Verdier duality $D\dg$ can now be read off from~\eqref{eqnkappadg}.	The action of the braiding and the balancing generators through their respective inverses on the left hand side
			translate to the inverse action of the braiding and the balancing on the very right hand side by the naturality of the isomorphisms~\eqref{eqncystructurewithN}. This proves the assertion.
		\end{proof}

		\begin{corollary}\label{cormulfinrib}
			Any finite ribbon category $\cat{A}$, as non-cyclic framed $E_2$-algebra, has a canonical structure of a homotopy fixed point with respect to the homotopy involution $\cat{A}\mapsto \bar{\cat{A}}\dg$. 
			Phrased more generally:
			On the  2-groupoid of finite ribbon categories, seen as full sub-2-groupoid of non-cyclic $\Rexf$-valued framed $E_2$-algebras that are self-injective and have an exact monoidal product,
			the homotopy involution $\cat{A}\mapsto \bar{\cat{A}}\dg$ is coherently trivial.
		\end{corollary}
		
		\begin{proof}
			By Proposition~\ref{proprevor} and Remark~\ref{rembarA}
			$\bar{\cat{A}}\dg$
			is, as \emph{non-cyclic framed $E_2$-algebra}, again $\cat{A}$  with monoidal product $\otimes$ 
			(since $\otimes$ and $\odot$ are identified thanks to rigidity) together with the original braiding and the original balancing.
		\end{proof}

		\subsection{Cyclic reflection invariance relative to a rigid duality}
		If we are given a finite ribbon category $\cat{A}$ and search for a homotopy fixed point structure with respect to $\cat{A}\mapsto \bar{\cat{A}}\dg$ as cyclic framed $E_2$-algebra, then
		Corollary~\ref{cormulfinrib} tells us that we already have such a homotopy fixed point structure as \emph{non-cyclic} algebra at our disposal.
		In that case, it makes sense to ask the homotopy fixed point structure at the cyclic level to extend the non-cyclic one.
		This will lead us to the following rather tedious definition which, however, encapsulates exactly the motivation given in the introduction:
		
		\begin{definition}\label{deffixedpointrel}
			Let $\cat{A}$ be a finite ribbon category in $\Rexf$.
			We say that $\cat{A}$ is \emph{reflection equivariant}, as a \emph{cyclic} framed $E_2$-algebra,
				 	if
				$\cat{A}$ is equipped with a $\dgr$-$\refl$-homotopy fixed point structure as \emph{cyclic} $\framed$-algebra in the sense of Definition~\ref{defhomotopyfixedpoint},
				 that is
				\emph{relative to the underlying non-cyclic homotopy fixed point structure from Corollary~\ref{cormulfinrib}}, i.e.\ a homotopy fixed point structure on the fibers of the forgetful functor from cyclic to non-cyclic algebras over framed $E_2$.
		\end{definition}

	\begin{remark}
		Definition~\ref{deffixedpointrel} covers reflection equivariant structures relative to a rigid duality.
		For this article, that focuses to a large extent on the rigid situation, this is doubtlessly the reasonable notion.
		Beyond that, one would have to address the rather difficult question 
		 whether the existence of \emph{any} cyclic reflection equivariant structure already implies rigidity (this would assume that one would actually be able to make sense of the operation $\dgr$).
		This is of course related to the rather fundamental question under which conditions Grothendieck-Verdier dualities in the sense of~\cite{bd} or weak dualities in the sense of \cite{baki} are rigid.
	Etingof and Penneys proved in~\cite{etingofpenneys} that at least every finitely semisimple braided weakly rigid monoidal category is rigid. 
		\end{remark}

		\subsection{A reminder on modified traces}
		It is one of the key insights of this section to  relate the structure from Definition~\ref{deffixedpointrel} to \emph{modified traces}~\cite{geerpmturaev,mtrace1,mtrace2,mtrace3,mtrace}
		whose definition we recall here:
		For a pivotal finite tensor category $\cat{A}$,
		a \emph{two-sided modified trace}
		on the tensor ideal $\Proj \cat{A}$ of projective objects consists of linear maps
		\begin{align}
			\trace_P:\cat{A}(P,P)\to k  \quad \text{for}\quad P \in \Proj \cat{A}
		\end{align}
		that are
		\begin{itemize}
			\item  \emph{cyclic}, i.e.\ $\trace_P(g\circ f)=\trace_Q(f \circ g)$ for morphisms $f : P \to Q$ and $g :Q \to P$ for $P,Q \in \Proj \cat{A}$,
			\item \emph{non-degenerate}, i.e.\ the pairings
			\begin{align}
				\label{eqnprojcystructure}	\cat{A}(P,Q) \otimes \cat{A}(Q,P) \to k \ , \quad f \otimes g \mapsto \trace_P(g \circ f)
			\end{align} for the morphism spaces $\cat{A}(-,-)$ of $\cat{A}$
			are non-degenerate, thereby endowing $\Proj \cat{A}$ with a Calabi-Yau structure,			\item and satisfy the \emph{two-sided partial trace property}, i.e.\
			\begin{itemize}\item the \emph{right partial trace property}\begin{align}
					\trace_{P\otimes X} (f) = \trace_P \left(  (\id_P \otimes d_{X^\vee}(\omega_X \otimes \id_{X^\vee})) \circ (f\otimes \id_{X^\vee}) \circ (\id_{P} \otimes b_X)  \right) 
				\end{align}holds for every morphism $f:P\otimes X \to P \otimes X$ with $P\in \Proj \cat{A}$ and $X\in \cat{A}$, 
				where  $b_X : I \to X \otimes X^\vee$ is the coevaluation at $X$,
				$d_{X^\vee} : X ^{\vee\vee} \otimes X^\vee \to I$ the evaluation at $X^\vee$
				and $\omega_X : X \to X^{\vee\vee}$ the pivotal structure at $X$,
				\item and  the similarly defined \emph{left partial trace property holds}.
			\end{itemize}
		\end{itemize}
		
		\subsection{The main result of cyclic reflection equivariance}
		We are now finally in a position to relate cyclic reflection equivariance to modified traces:
		\begin{theorem}\label{thmmain0}
			A finite ribbon category $\cat{A}$ can be made reflection equivariant as cyclic framed $E_2$-algebra in the sense of Definition~\ref{deffixedpointrel}
			if and only if $\cat{A}$ is a unimodular, and a reflection equivariant structure
			amounts exactly to the choice of a two-sided modified trace on $\Proj \cat{A}$.
			In particular, a modular category can be endowed with cyclic reflection equivariant structure.  
		\end{theorem}

		\begin{proof}
			If $\cat{A}$ is a finite ribbon category, then we have by Proposition~\ref{proprevor} and the simplifications implied by rigidity (Corollary~\ref{cormulfinrib})
			\begin{align}
				\label{eqnadgdescription}	 	\cat{A}\dg (\Sigma_{0,n};P)= \cat{A}(  P_1 \otimes \dots \otimes P_n ,\nakar I )^* \quad \text{for} \quad P_1,\dots,P_n\in \Proj \cat{A} \ . 
			\end{align}
			Since the functor on the right hand side is right exact, we find in fact~\eqref{eqnadgdescription} for all objects.
			Since $\cat{A}$ is a finite tensor category, we can use \cite[Theorem~4.26]{fss} to conclude that $\nakar I$ is the inverse $\alpha^{-1}$ of the
			distinguished invertible object $\alpha$ from~\cite{eno-d} describing the quadruple dual via $-^{\vee\vee\vee\vee} \cong \alpha \otimes-\otimes\alpha^{-1}$. 
			Therefore,
			\begin{align}
				\cat{A}\dg (\Sigma_{0,n};P)= \cat{A}(  P_1 \otimes \dots \otimes P_n ,\alpha ^{-1})^*  \ . \label{eqnadaggertwisted}
			\end{align}
			From Proposition~\ref{proprevor}~\ref{proprevoriii}, it follows now that 
			$\cat{A}\dg$ agrees with $\bar{\cat{A}}$ as framed $E_2$-algebra
			(same monoidal product, but braiding and balancing are inverted), but
			with $\alpha^{-1}$ as dualizing object; in other words, $\cat{A}\dg$ is obtained by twisting the rigid Grothendieck-Verdier duality of $\bar{\cat{A}}$ by $\alpha^{-1}$. 
			
			The key fact that we will use is that actually all ribbon Grothendieck-Verdier structures relative to a balanced braided structure
			are obtained by such a twist.
			To state this precisely, 
			denote  for any ribbon Grothendieck-Verdier category $\cat{A}$  by
			$Z_2^\catf{bal}(\cat{A})$
			the balanced Müger center of $\cat{A}$, defined as the full subcategory of $\cat{A}$
			consisting of those objects of $\cat{A}$ 
			that trivially double braid with all other objects and have trivial balancing.
			By $\catf{PIC}(Z_2^\catf{bal}(\cat{A}))$ we denote the Picard groupoid of $Z_2^\catf{bal}(\cat{A})$, defined as the full subgroupoid of $\otimes$-invertible objects.
			Moreover, we define
			\begin{align} \catf{F} :=   \begin{array}{c}
					\text{homotopy fiber of the non-cyclic algebra $\cat{A}$} \\ \text{with respect to the functor} \\ \text{forgetting the cyclic structure}
				\end{array}
			\end{align}
			By \cite[Theorem~4.2]{mwcenter}
			the functor
			\begin{align}\label{eqnpicardequiv}
				\catf{PIC}(Z_2^\catf{bal}(\cat{A})) \ra{\simeq}
				\catf{F}  \ , \quad  \beta \mapsto \begin{array}{c} \cat{A} \ 
					\text{as non-cyclic algebra,} \\  \text{equipped with Grothendieck-Verdier duality} \ \beta \otimes D \end{array}
			\end{align}
			is an equivalence of groupoids. 
			For this statement, it is not relevant that the duality
			 $D$ is rigid, so we refrain from denoting it by $-^\vee$.
			Under the equivalence $\catf{PIC}(Z_2^\catf{bal}(\cat{A})) \simeq \catf{F}$ the unit $I\in  \catf{PIC}(Z_2^\catf{bal}(\cat{A}))$ is mapped to $\cat{A}$ while the inverse of the
			distinguished invertible object $\alpha$ is mapped to $\bar{\cat{A}}\dg$ as follows from~\eqref{eqnadaggertwisted}. 
			Note that the fact that $\alpha$ lies actually in the balanced Müger center holds because it corresponds 
			to the ribbon Grothendieck-Verdier structure of $\cat{A}\dg$, but it is also a priori clear~\cite[Remark~3.3]{mwskein}.
			
			A cyclic reflection equivariant structure for $\cat{A}$ is by definition a 
			homotopy fixed point structure for $\bar{-}\dg$ on the point $\cat{A}\in\catf{F}$ through a morphism in the fiber $\catf{F}$.
			To this end, we first need to compute the homotopy involution $\bar{-}\dg$ on $\catf{F}$ --- of course with the model provided by the equivalence~\eqref{eqnpicardequiv}.
			We denote for $\beta \in \catf{PIC}(Z_2^\catf{bal}(\cat{A}))$ by $\cat{A}_\beta$ the ribbon Grothendieck-Verdier category whose underlying balanced braided category is the one of $\cat{A}$ with ribbon Grothendieck-Verdier duality $\beta \otimes D$.
			Then, with the same computation leading to~\eqref{eqnadaggertwisted}, we find
			\begin{align}
				\cat{A}_\beta\dg (\Sigma_{0,n};P)\cong \cat{A}(P_1 \otimes \dots \otimes P_n ,\nakar \beta^{-1}    )  \ , \label{eqnadaggertwisted2}
			\end{align}
			thereby proving that, under the equivalence 
			$\catf{F}\simeq \catf{PIC}(Z_2^\catf{bal}(\cat{A}))$, the homotopy involution on $\catf{F}$ is given by $\beta \mapsto \nakar \beta^{-1}$,
			with the coherence data being the 
			isomorphism $\nakar (\nakar \beta^{-1})^{-1} \cong \nakar \nakal \beta \cong \beta$.  
			The cyclic equivariance is therefore equivalent to a homotopy fixed point structure 
			on $I$ for this action, which is a trivialization of $\alpha^{-1}$ and therefore a right modified trace on $\Proj \cat{A}$ by \cite[Theorem~6.2]{mtrace}, see also \cite[Theorem~5.8 \& 6.8]{shibatashimizu} and \cite[Theorem~3.6]{tracesw}. 
			By \cite[Theorem~1.4]{shibatashimizu} this a priori right modified trace must be two-sided.
				The addition on modular categories is a consequence of \cite[Proposition~4.5]{eno-d}.
		\end{proof}

			\begin{example}\label{exvoa}
			Let us look at the implications of Theorem~\ref{thmmain0} for 	 the theory of vertex operator algebras~\cite{fhl,frenkelbenzvi}, one of the main sources for balanced braided monoidal categories.
			Suppose that $V$ is an $\mathbb{N}$-graded, simple, self-contragredient, $C_2$-cofinite vertex operator algebra, and denote by $\cat{A}_V$ the category of grading-restricted $V$-modules. This is a balanced braided monoidal category \cite{hlzi} for which we assume that it lives in $\Rexf$. 
			Then the following conditions are equivalent:
			\begin{pnum}
				\item $\cat{A}_V$ admits a cyclic reflection equivariant structure relative to a rigid duality in the sense of Definition~\ref{deffixedpointrel} (for \emph{any} cyclic structure that it might have). 	\label{exvoai}
				\item $\cat{A}_V$ is modular.\label{exvoaii}
			\end{pnum}
			Indeed, if $\cat{A}_V$ is even just rigid, then $\cat{A}_V$ is already modular by \cite[Theorem~1]{mcrae}.
			The converse follows from Theorem~\ref{thmmain0}.
			Note that in~\ref{exvoai}, we look a priori at all cyclic structures, but it must be necessarily the one coming from the contragredient representation as in \cite{alsw} because, thanks to the equivalent condition \ref{exvoaii}, the space of cyclic structures on $\cat{A}_V$ is connected, even though it is not contractible~\cite[Corollary 4.5]{mwcenter}.
		\end{example}

		\subsection{A reminder on handlebody skein modules\label{sechbdyskein}}
		Let $\cat{A}$ be a ribbon Grothendieck-Verdier category in $\Rexf$ (or more generally a cyclic framed $E_2$-algebra in some symmetric monoidal bicategory). Fix a handlebody $H$ with $n$ embedded disks and boundary surface $\Sigma=\partial H$.
		The operation $\partial$, when applied to a handlebody with embedded disks, takes the boundary surface and converts embedded disks into boundary circles of the boundary surface.
		For $X_i \in  \cat{A}$ for $1\le i\le m$, the \emph{generalized skein module}~\cite[Section~4.1]{brochierwoike}   
		$\PhiA(H):\int_\Sigma \cat{A}\to\cat{A}^{\boxtimes n}$ 
		is determined by
		\begin{align} \label{defPhiA}\PhiA(H ;    \varphi_* (X_1 \boxtimes \dots \boxtimes X_m) )   = \widehat{\cat{A}}(H^\varphi;X_1 \boxtimes \dots \boxtimes X_m) \ , 
		\end{align} for any oriented embedding $\varphi$ of a family of disks into $\Sigma$, where $\widehat{\cat{A}}$ is the ansular functor of $\cat{A}$ and the $n$ boundary components are seen as outgoing (we can make this choice as we see fit using the Grothendieck-Verdier duality).
		Both source and target of $\PhiA(H):\int_\Sigma \cat{A} \to \cat{A}^{\boxtimes n}$ are modules over the algebra $\int_{\partial \Sigma \times [0,1]} \cat{A}$, and $\PhiA(H)$ is a module map.
		As a special case, the $\PhiA$-map gives us an action 
		\begin{align}
			\Add_{\Sigma,H} : \SkAlg_\cat{A}(\Sigma) \to \End( \widehat{\cat{A}}(H) ) \        \label{eqnaddmap}  
		\end{align}of the skein algebra $\SkAlg_\cat{A}(\Sigma)=\End_{\int_\Sigma \cat{A}}(\qss)$ on the value $\widehat{\cat{A}}(H)$ of the ansular functor for $\cat{A}$ on $H$, thereby justifying the fact that $\widehat{\cat{A}}$ can be seen as skein module.
		The notation `$	\Add_{\Sigma,H}$' is borrowed from \cite{masbaumroberts}.
		Extracting~\eqref{eqnaddmap} just uses the trivial fact that for any functor $F$ the value $F(X)$ on some object $X$ is a module over the endomorphism algebra of $X$.
		In more detail, \eqref{eqnaddmap} makes the following triangle commute:
		\begin{equation}\label{eqndefadd}
			\begin{tikzcd}
				&	\SkAlg_\cat{A}(\Sigma)=\End(\qss)	 \ar[rrrd,"\PhiA(H)\circ -"] \ar[dd,"	\Add_{\Sigma,H} ",swap]  \\ &&&& \End(\PhiA(H)\circ\qss) \ ,   \\ 
				&	\End(\widehat{\cat{A}}(H)) \ar[rrru,"\text{isomorphism from $\PhiA(H)\circ\qss\cong \widehat{\cat{A}}(H)$ from \cite[Theorem~4.2]{brochierwoike}}", swap] 
			\end{tikzcd}
		\end{equation}
		It is not obvious that the  skein modules defined here recover in more classical situations the construction from \cite{turaevck,turaevskein,hp92,Przytycki,Walker,masbaumroberts}, but this is  implied by the results in \cite{cooke,asm,brownhaioun,mwskein}.

		\subsection{Reflection equivariance for handlebody skein modules\label{sechandlebodyrefl}}
		From the definition of the $\dgr$-operation, \eqref{eqndgransular}, Remark~\ref{remdnaeq} and~\eqref{defPhiA}, we obtain: 
		\begin{lemma}\label{lemmaconnecteddg}
			Let $\cat{A}$ be a self-injective cyclic framed $E_2$-algebra in $\Rexf$ with exact monoidal product, thereby making $\cat{A}\dg$ well-defined.
			Then $\PhiAdg(H):\int_{ \Sigma } \cat{A} \dg \to \cat{A}^{\boxtimes n}$ for a handlebody with $n\ge 0$ embedded disks and $\Sigma = \partial H$ is determined by
			\begin{align} \PhiAdg(H ;    \varphi_* (P_1 \boxtimes \dots \boxtimes P_m) )   =  D_n^\text{rev}\PhiA(H ;     \varphi_* (DP_m \boxtimes \dots \boxtimes DP_1) )	 \quad \text{with}\quad P_1,\dots,P_m \in \Proj \cat{A}     
			\end{align} for every oriented embedding $\varphi : \sqcup_{m} \mathbb{D}^2 \to \Sigma$ with $m\ge 1$. 
		\end{lemma}
		
		Together with Theorem~\ref{thmmain0}, this Lemma implies:

		\begin{proposition}\label{propphimaps}
			Let $\cat{A}$ be a unimodular finite ribbon category in $\Rexf$ with chosen modified trace.
			Then $\PhiA(\bar H):\int_{ \bar \Sigma } \cat{A}  \to \cat{A}^{\boxtimes n}$ for a handlebody with $n\ge 0$ embedded disks and $\Sigma = \partial H$ is determined by 
			\begin{align} \PhiA(\bar H ;    \check{\varphi}_* (P_1 \boxtimes \dots \boxtimes P_m) )  \cong  D_n^\text{rev}\PhiA(H ;     \varphi_* (DP_m \boxtimes \dots \boxtimes DP_1) ) \quad \text{with}\quad P_1,\dots,P_m \in \Proj \cat{A} \ .    
			\end{align} for every oriented embedding $\varphi : \sqcup_{m} \mathbb{D}^2 \to \Sigma$ with $m\ge 1$, where $\check{\varphi}$ is the oriented embedding $\varphi$, but precomposed with a reflection of all disks. 
		\end{proposition}

		\subsection{Factorization homology description of modular functors\label{secmf}}
		According to \cite[Section~3.1]{brochierwoike} a modular functor with values in $\Rexf$ is a pair $(\cat{Q},\mathfrak{F})$ 
	\begin{itemize}
		\item of a $\Grpd$-valued modular operad $\cat{Q}$ that is a certain extension of the modular surface operad $\Surf$ (we will not discuss the details here, but for the usual modular functors, this reduces to the well-known  extensions of the mapping class groups, see \cite[Remark~7.14]{brochierwoike}),
		\item together with a $\Rexf$-valued
		modular $\cat{Q}$-algebra $\mathfrak{F}$.
	\end{itemize}
	The extension part of the pair is often suppressed in the notation.
	Note however that the datum of the extension is important to define the moduli space $\mathfrak{MF}$ of modular functors~\cite[Section~3.2]{brochierwoike}.
	
		If for each surface $\Sigma$ all $\PhiA(H)$ for all handlebodies $H$ with boundary $\Sigma$ are isomorphic as module maps, we call $\cat{A}$ \emph{connected}~\cite[Definition~5.5]{brochierwoike}.   
	The main result~\cite[Theorem 6.8]{brochierwoike} tells us that genus zero restriction provides an equivalence
	\begin{align}
		\begin{array}{c} \text{moduli space $\mathfrak{MF}$} \\ \text{of $\Rexf$-valued modular functors}\end{array} \ra{\simeq} \begin{array}{c} \text{\underline{\emph{connected}} cyclic framed} \\  \text{$E_2$-algebras in $\Rexf$}\end{array} \label{eqnclass} \ . 
	\end{align}
	An explicit inverse is provided in~\cite[Section~6]{brochierwoike}:
	If $\cat{A}$ is connected,  its essentially unique extension to a modular functor $\FA$ \cite[Theorem~7.6]{brochierwoike} can be defined over an extension $\SurfA \to \Surf$ of $\Surf$ that we can  describe as follows:
		The homotopy fiber over a surface $\Sigma$ is the groupoid of all $\PhiA(H)$ with $\partial H=\Sigma$, with the isomorphisms being isomorphisms of module maps. 
		Equivalently, we can see this homotopy fiber over $\Sigma$
		as the groupoid whose objects are all handlebodies $H$ with $\partial H=\Sigma$, with the morphisms $H\to H'$ being the module isomorphisms $\PhiA(H)\to\PhiA(H')$.
		 Connectedness of $\cat{A}$ means exactly that all homotopy fibers are connected.
		The definition of $\SurfA$ as a modular operad works always, but it is not necessarily an extension.
		A modular category is an example of a connected finite ribbon category, and the resulting modular functor $\FA$ is equivalent to Lyubashenko's modular functor~\cite[Section~8.2]{brochierwoike}.

	\subsection{The reflection operation and the $\dgr$-operation for modular functors}
	With the reflection operation, we can act as follows on modular functors:
		
		\begin{definition}
			Let $\mathfrak{F}=(\cat{Q},\cat{B})$ be a $\Rexf$-valued modular functor and $\refl_\cat{Q} : \catf{R}^* \cat{Q} \ra{\simeq} \cat{Q}$ the equivalence of modular operads obtained by homotopy pullback of $\refl : \Surf \ra{\simeq} \Surf$ along $\cat{Q}\to \Surf$.  
			We now set \begin{align}
				\bar{\mathfrak{F}}=	(\bar{\cat{Q}} , \bar{\cat{B}}):= (\refl^*\cat{Q} , \refl_\cat{Q}^* \cat{B}           ) \ , 
			\end{align} 
			thereby defining the \emph{reflection operation} as a homotopy involution
			on the 2-groupoid
			$\mathfrak{MF}$ describing the moduli space of modular functors.
		\end{definition}

		\begin{remark}\label{remcompgenuszeroreversal} 
		A cyclic framed $E_2$-algebra $\cat{A}$ is connected if and only if $\bar{\cat{A}}$ is connected, and the equivalence~\eqref{eqnclass} is $\mathbb{Z}_2$-equivariant with respect to orientation reversal: $ \bar{\mathfrak{F}}_\cat{A}\simeq \mathfrak{F}_{\bar{\cat{A}}} $.
		\end{remark}

		Lemma~\ref{lemmaconnecteddg} implies that a cyclic framed $E_2$-algebra $\cat{A}$ (under the assumption that it is self-injective and the monoidal product is exact) is connected if and only if $\cat{A}\dg$ is connected. With~\eqref{eqnclass} this implies:
		\begin{lemma}\label{lemmadefdgfa}
			Let $\mathfrak{F}=(\cat{Q},\cat{B})$ be a $\Rexf$-valued modular functor
			such that the cyclic  framed $E_2$-algebra 
			algebra $\cat{A}$ obtained genus zero restriction
			is self-injective with exact monoidal product.
			Then there is, up to a contractible choice, a unique modular functor extending $\cat{A}\dg$; we denote this modular functor by $\mathfrak{F}\dg$.
			Then $\mathfrak{F}\mapsto \mathfrak{F}\dg$ establishes a homotopy involution.
			The modular functor $\mathfrak{F}\dg$ is determined by 
			$	\mathfrak{F}\dg( \Sigma ;P ) \cong \mathfrak{F}(\Sigma ; D_n^\text{rev} P)^*$ on projective objects, where it is understood that the surface $\Sigma$ has at least one boundary component per connected component.
		\end{lemma}
		
	\subsection{Strong rigidity and reflection equivariance for modular functors}	
		One can require the category obtained from a modular functor by genus zero restriction to
		just be rigid (as a property), but this is not a terribly useful notion.
	The ribbon Grothendieck-Verdier categories extracted from the modular functors in~\cite[Example 8.8]{brochierwoike} and~\cite{sn} are rigid (as a property of the monoidal category), but the Grothendieck-Verdier duality is generally different from the rigid duality. Moreover, the rigid duality could not even be chosen as the Grothendieck-Verdier duality because this would violate the ribbon condition in this case.
	Therefore, the following definition makes sense:
	
		\begin{definition}
			We call a $\Rexf$-valued modular functor, ansular functor or genus zero modular functor \begin{pnum}\item \emph{strongly rigid} if the Grothendieck-Verdier duality of its underlying ribbon Grothendieck-Verdier category  is a rigid duality,
				\item and \emph{simple} if the unit of the underlying monoidal category is simple, i.e.\
				has a simple unit.
				\end{pnum}
		\end{definition}

		\begin{definition}\label{defrefleqmf}
			For a strongly rigid $\Rexf$-valued modular functor $\mathfrak{F}$, a \emph{reflection equivariant structure} is a homotopy fixed point structure with respect to the homotopy involution $\mathfrak{F} \mapsto \bar{\mathfrak{F}}\dg$, relative to the underlying non-cyclic framed $E_2$-algebra structure as in Definition~\ref{deffixedpointrel}. 
		\end{definition}

		\begin{remark}\label{remmfrefl}
			When spelling out the definition of $\bar{\mathfrak{F}}\dg$,
			we find that the reflection equivariant structure consists of isomorphisms
			\begin{align} \mathfrak{F}(\bar \Sigma ; D_n^\text{rev} P) \cong \mathfrak{F}(\Sigma ; P)^*       \label{eqnrefleqsplledout}
			\end{align} for all surfaces $\Sigma$ with at least one boundary component for every 
		connected component and only projective boundary labels $P \in \Proj \cat{A}^{\boxtimes n}$. 
			More precisely, these isomorphisms 
	\begin{itemize}\item are compatible with gluing,\item are involutive,\item and extend, on the level of non-cyclic framed $E_2$-algebras,
		the canonical isomorphisms coming from the rigid pivotal duality.\end{itemize}
		\end{remark}
		
		\begin{remark}
			The reflection equivariance in Definition~\ref{defrefleqmf}
			is similar to the modular functor being \emph{unitary} in the sense of \cite[Definition~5.1.14]{baki}, but we  should warn the reader that we have the isomorphism~\eqref{eqnrefleqsplledout} only for $n\ge 1$ and only for projective objects as boundary labels (in~\cite{baki} this does not play a role because only the semisimple case is considered). Moreover, \eqref{eqnrefleqsplledout} is a part of a homotopy fixed point structure with respect to a homotopy involution. This entails coherence conditions that do not seem to be considered in~\cite{baki}.
		\end{remark}

		\subsection{Characterization of reflection equivariance for strongly rigid modular functors}
		With Theorem~\ref{thmmain0}, it will be straightforward to prove:
		
		\begin{theorem}\label{thmrefleqmf}
			A strongly rigid simple
			modular functor in $\Rexf$ with underlying finite ribbon 
			 category $\cat{A}$ can be made reflection equivariant if and only if
			$\cat{A}$ is unimodular, and 
			a  reflection equivariant structure
			amounts exactly to a two-sided modified trace on $\Proj \cat{A}$. 
			In particular, the modular functor of a modular category can be made reflection equivariant 
			through the choice of a 
			two-sided modified trace. 
		\end{theorem}
		
		\begin{proof}
			The equivalence~\eqref{eqnclass} restricts to an equivalence
			\begin{align}
				\left.\begin{array}{c} \text{moduli space $\mathfrak{MF}$} \\ \text{of $\Rexf$-valued modular functors}\end{array} \right|_{\substack{\text{self-injective},\\ \otimes \ \text{exact}}}\ra{\simeq} \left.\begin{array}{c} \text{\underline{\emph{connected}} cyclic framed} \\  \text{$E_2$-algebras in $\Rexf$}\end{array} \right|_{\substack{\text{self-injective},\\ \otimes\ \text{exact}}} \label{eqnclass2} 
			\end{align}
			that is equivariant with respect to the action $(-) \mapsto (\bar{\ -\ })\dg$
			as follows from Remark~\ref{remcompgenuszeroreversal} and Lemma~\ref{lemmadefdgfa}.
			The relativity to a pivotal rigid duality of the homotopy fixed points for $(-) \mapsto (\bar{\ -\ })\dg$
			is defined in the same way in
			Definition~\ref{deffixedpointrel} and~\ref{defrefleqmf}. 
			Therefore, the statement follows from Theorem~\ref{thmmain0}.
		\end{proof}

	\spaceplease
		\section{The connection between reflection equivariance and the non-degeneracy of the braiding}
		In this section, we will formulate and prove the main result of the article (Theorem~\ref{thmrefl}).
		To this end, we need to explore the connection between reflection equivariance 
		and the non-degeneracy of the braiding.
			
		\subsection{Co-non-degeneracy and the Müger center\label{secconondeg}}
		According to \cite[Definition~7.13]{brochierwoike},
		we call a cyclic framed $E_2$-algebra $\cat{A}$ \emph{co-non-degenerate} if the handlebody skein module $\PhiA(\mathbb{B}^3):\int_{\mathbb{S}^2} \cat{A} \to I$ for the three-dimensional ball
		(here $I$ is the monoidal unit in the ambient symmetric monoidal bicategory) is an equivalence.  
		The following is a non-semisimple analogue of \cite[Proposition 4.4]{skeinfin} and a more elaborate version of \cite[Remark~7.14]{brochierwoike}:
		\begin{proposition}\label{propcodeg}
			For a cyclic framed $E_2$-algebra $\cat{A}$ in $\Rex$, the application of the internal hom $\Rex[-,\vect]$ in $\Rex$ to $\vect$
			to the map
			$\PhiA(\mathbb{B}^3):\int_{\mathbb{S}^2}\cat{A}\to \vect$ yields the map
			$\vect \to Z_2(\cat{A})$ selecting the monoidal unit of $\cat{A}$. 
			Suppose that $\cat{A}$ is a finite tensor category (with rigidity just understood as property of the monoidal product), then
			$ \int_{\mathbb{S}^2} \cat{A} \simeq Z_2(\cat{A})^\op $ by a canonical equivalence, and $\cat{A}$ is co-non-degenerate  
			if and only if it has a non-degenerate braiding.	\end{proposition}
		
		\begin{proof}
			After cutting $\mathbb{S}^2$ into two hemispheres, we find with excision \begin{align}\label{eqnexcisions2} \int_{\mathbb{S}^2}\cat{A} \simeq \cat{A}\boxtimes_{\int_{\mathbb{S}^1\times [0,1]}\cat{A}} \cat{A} \ . 
			\end{align}
			Under the equivalence~\eqref{eqnexcisions2}, $\PhiA(\mathbb{B}^3)$ is the map induced by the pairing $\kappa : \cat{A}\boxtimes\cat{A}\to \vect$. 
			Next we calculate with the tensor hom adjunction
			\begin{align}
				\Rex\left[  \cat{A}\boxtimes_{\int_{\mathbb{S}^1\times [0,1]}\cat{A}} \cat{A}     ,\vect\right] \simeq \Rex  _{\int_{\mathbb{S}^1\times [0,1]}\cat{A} }  \left[   \cat{A}    , \cat{A} \right]  \label{eqntensorhom}
			\end{align}
			as in the proof of \cite[Proposition 5.9]{brochierwoike}.
			The Müger center $Z_2(\cat{A})$, in its description in~\cite[Definition~2.6]{bjss}, is $\Rex _{\int_{\mathbb{S}^1\times [0,1]}\cat{A} } [\cat{A},\cat{A}]$, and the canonical map $\vect \to Z_2(\cat{A})\simeq \Rex _{\int_{\mathbb{S}^1\times [0,1]}\cat{A} } [\cat{A},\cat{A}]$ selects the identity of $\cat{A}$. This proves that $\PhiA(\mathbb{B}^3)$, under all these identifications, dualizes to the map $\vect \to Z_2(\cat{A})$ selecting the unit.
			
			Suppose now that $\cat{A}$ is a finite tensor category.
			Then $\int_{\mathbb{S}^1\times [0,1]}\cat{A}$ is a finite tensor category~\cite[Section~4.2]{bzbj2}.
			Thanks to \cite[Theorem~3.3~(1)]{dss}, this implies that $\int_{\mathbb{S}^2}\cat{A}$ is a finite linear category.
			For finite categories, $\Rex[-,\vect]$ is involutive and sends a category to its opposite. By applying $\Rex[-,\vect]$ to \eqref{eqntensorhom} we obtain a canonical equivalence 
			$\int_{\mathbb{S}^2} \cat{A} \simeq Z_2(\cat{A})^\op$.

			If $\cat{A}$ is co-non-degenerate, i.e.\ if $\PhiA(\mathbb{B}^3)$ is an equivalence, so is  $\vect \to Z_2(\cat{A})$. This implies that the braiding of $\cat{A}$ is non-degenerate. 
			
			Conversely, if the braiding of $\cat{A}$ is non-degenerate,  $\vect \to Z_2(\cat{A})$ is an equivalence by definition, which implies that $\Rex[ \PhiA(\mathbb{B}^3) , \vect] : \vect \ra{\simeq} \Rex\left[ \int_{\mathbb{S}^2}\cat{A},\vect   \right]$ is an equivalence, and hence `bi-dualizes' to an equivalence
			\begin{align} \Rex \left[\Rex[ \PhiA(\mathbb{B}^3) , \vect],\vect \right]  : \Rex \left[\Rex\left[ \int_{\mathbb{S}^2}\cat{A},\vect   \right],\vect \right] \ra{\simeq} \vect \ .  \end{align} 
			But this map can be identified with $\PhiA(\mathbb{B}^3)$ again because $\int_{\mathbb{S}^2} \cat{A}$ is  finite and $\Rex[-,\vect]$ is involutive on finite categories.
			This means that $\cat{A}$ is co-non-degenerate. 
		\end{proof}

			This leads to a non-semisimple generalization of \cite[Lemma 4.5]{skeinfin}
			that can be phrased in terms of the admissible skein modules from~\cite{asm}:

		\begin{corollary}\label{corS2S1}
			Let $\cat{A}$ be a unimodular finite ribbon category.
			Then the admissible skein module for $\mathbb{S}^2 \times \mathbb{S}^1$ is given by
			\begin{align} 	\skA(\mathbb{S}^2 \times \mathbb{S}^1) \cong HH_0(Z_2(\cat{A})) \ . 
			\end{align}
		\end{corollary}

		Here $HH_0$ is the zeroth Hochschild homology of a finite linear category $\cat{L}$  defined via $HH_0(\cat{L}) = \int^{P\in \Proj \cat{L}} \cat{L}(P,P)$. It satisfies the property $HH_0(\cat{L})\cong HH_0(\cat{L}^\op)$. A reminder on Hochschild homology for finite linear categories, as they occur in quantum algebra, is given in~\cite[Section~2]{dva}. 
		
		\begin{proof}[\slshape Proof of Corollary~\ref{corS2S1}]
			The uncocompleted admissible skein form a three-dimensional bimodule-valued topological field theory \cite[Theorem 2.21]{brownhaioun} sending $\mathbb{S}^2$ to the uncompleted skein category $\skcatA(\mathbb{S}^2)$. Its value on $\mathbb{S}^2 \times \mathbb{S}^1$ is $\skA(\mathbb{S}^2 \times \mathbb{S}^1)$. By general topological field theory principles, $\skA(\mathbb{S}^2 \times \mathbb{S}^1)$ is the dimension object of $\skcatA(\mathbb{S}^2)$, which is the coend $\int^{P\in \skcatA(\mathbb{S}^2)} \End_{\skcatA(\mathbb{S}^2)} (P)$. 
			Since the finite free cocompletion of $\skcatA(\mathbb{S}^2)$ is $\int_{\mathbb{S}^2} \cat{A}$ by \cite[Theorem~3.10]{brownhaioun}, we conclude
			\begin{align}
				\skA(\mathbb{S}^2 \times \mathbb{S}^1) \cong HH_0\left( \int_{\mathbb{S}^2} \cat{A}   \right) \ , 
			\end{align}
			and therefore
			\begin{align}
				\skA(\mathbb{S}^2 \times \mathbb{S}^1) \cong HH_0(Z_2(\cat{A}))
			\end{align}
			with Proposition~\ref{propcodeg}.
		\end{proof}

		With \cite[Corollary 3.10]{mwskein}, we obtain:
		
		\begin{corollary}\label{corconnectedsum}
			Let $\cat{A}$ be a unimodular finite ribbon category and
			$M$ and $N$ closed oriented three-dimensional manifolds.
			Then there is a canonical isomorphism
			\begin{align}
				\skA(M \# N) \cong \int^{P \in \Proj Z_2(\cat{A})} \skA(M\setminus B ; P^\vee) \otimes \skA(N\setminus B;P) \ ,
			\end{align}
			where $B$ is an open three-dimensional ball cut out of $M$ and $N$, respectively.
		\end{corollary}
		
		\begin{example}
			If $\cat{A}$ is a ribbon fusion category, Corollary~\ref{corconnectedsum} simplifies to
			\begin{align}
				\skA(M \# N) \cong \skA(M\setminus B) \otimes \skA(N\setminus B) \oplus \bigoplus_{i=1}^n  \skA(M\setminus B ; X_i^\vee) \otimes \skA(N\setminus B;X_i) \ ,
			\end{align}
			where $I,X_1,\dots,X_n$ is a basis of simple objects for the Müger center. 
		\end{example}
		
		\begin{example}
			Let $H$ be a finite-dimensional cocommutative unimodular ribbon
			Hopf algebra; $H$ could be non-semisimple, for example the group algebra of the cyclic group with $p$ elements in characteristic $p$.
			Then Corollary~\ref{corconnectedsum} gives us
			\begin{align}
				\skA(M \# N) \cong  \skA(M\setminus B ;H^*) \otimes_H \skA(N\setminus B;H) \ .
			\end{align}
		\end{example}
		
			\subsection{Calculating admissible skein modules}
		We will now derive a formula for admissible skein modules in terms of natural transformations between generalized skein modules.
		To this end, let us recall some notation from \cite[Section~2.3]{brochierwoike}:
		For a handlebody $H$ with $\partial H=\Sigma$ and a mapping class $f: \Sigma \to \Sigma'$, we denote by $f.H$ the handlebody defined via the following pushout:
		\begin{equation}	
			\begin{tikzcd}
				\Sigma\times\{0\} \cong \Sigma=\partial H \ar[rr] \ar[dd,"f\times 0",swap] && H \ar[dd,"\widetilde f"] \\ \\ 
				\Sigma'\times [0,1] \ar[rr, swap]  & & f.H \ .
			\end{tikzcd}
		\end{equation}

		\begin{proposition}\label{proplensspace}
			Let $\cat{A}$ be a unimodular finite ribbon category.
			Denote by $L_f$  the compact oriented three-dimensional manifold obtained by gluing two  handlebodies along a mapping class $f \in \Map(\Sigma)$ of the their boundary surface $\Sigma$. 
			Then $\skA(L_f)$ is dual to the vector space of natural transformations $\PhiA(f.H) \to \PhiA(H)$; 
			\begin{align}
				\skA(L_f)^*\cong \catf{Nat}(\PhiA(H),\PhiA(f.H))=\int_{X \in \int_\Sigma \cat{A}} \PhiA(f.H)X \otimes (\PhiA(H)X)^* \ . 
			\end{align}
		\end{proposition}
		
		\begin{proof}
			By definition $L_f = \overline{f.H} \cup_{\Sigma} H$, and hence
			\begin{align}
				\skA(L_f) \cong \int^{P\in \Proj \int_{\Sigma} \cat{A}} \PhiA(\overline{f.H};P^\vee) \otimes \PhiA(H;P)
			\end{align} by \cite[Corollary~3.10]{mwskein}.
			We have $\PhiA(\overline{f.H};P^\vee)\cong \PhibarA (f.H;P^\vee)$ and hence 
			$\PhiA(\overline{f.H};P^\vee)\cong \PhiAdg(f.H;P^\vee)$ thanks to Theorem~\ref{thmrefleqmf}.
			This leaves us with
			\begin{align}
				\skA(L_f)\cong\int^{P\in \Proj \int_{\Sigma} \cat{A}} \PhiAdg(f.H;P^\vee)\otimes \PhiA(H;P)\stackrel{\text{Lemma~\ref{lemmaconnecteddg}}}{\cong} \int^{P\in \Proj \int_{\Sigma} \cat{A}} \PhiA(f.H;P)^*\otimes \PhiA(H;P)
			\end{align}
			and hence
			\begin{align}
				\SkA(L_f)^*\cong  \int_{P\in \Proj \int_{\Sigma} \cat{A}} \PhiA(f.H;P)\otimes \PhiA(H;P)^* \stackrel{\eqref{eqnnat}}{\cong}
				\catf{Nat}(\PhiA(H),\PhiA(f.H)) \ . 
			\end{align}
		\end{proof}

	\begin{corollary}\label{corendPhiA}
		For any unimodular finite ribbon category $\cat{A}$ that extends to a modular functor, and any handlebody $H$ with $\partial H=\mathbb{T}^2$,
		the algebra of endomorphisms of $\PhiA(H):\int_{\mathbb{T}^2} \cat{A}\to\vect$ is one-dimensional:
		\begin{align}
			\End(\PhiA(H))\cong k\ . 
			\end{align}
		\end{corollary}
		
		\begin{proof}
			The following generalizes an argument from \cite[Theorem 5.3.1]{damioliniwoike} from the semisimple to the non-semisimple case:
			If $\cat{A}$ extends
			to a modular functor,
			we obtain thanks to \cite[Theorem~6.8]{brochierwoike} and
			Proposition~\ref{proplensspace}
			\begin{align}
				\skA(L)^*\cong \End(\PhiA(H)) \label{eqnindependence}
			\end{align}
			for any compact oriented three-manifold $L$ obtained from a Heegaard splitting involving a handlebody $H$, \emph{independently} of the mapping class along which we glue.

			We exploit this for two special cases:
			For $L=\mathbb{S}^3$, the skein module is one-dimensional thanks to unimodularity \cite[Corollary~3.2]{asm}.
			Since $\mathbb{S}^3$ and $\mathbb{S}^2 \times \mathbb{S}^1$ both admit a genus one Heegaard splitting, their skein modules are isomorphic, which gives us $\dim\, \skA(\mathbb{S}^2 \times \mathbb{S}^1)=1$.
		To obtain the statement, set $f=\id$ in Proposition~\ref{proplensspace} and use $\bar H\cup_{\mathbb{T}^2} H =\mathbb{S}^2 \times \mathbb{S}^1$.
			\end{proof}

		\subsection{Topological characterization of modular categories\label{sectopcharmod}}		Proposition~\ref{proplensspace} is the key calculation needed to understand when a unimodular finite ribbon category extends to a modular functor.

			\begin{proposition}\label{propconnectedmodular}
			The braiding of a unimodular finite ribbon category $\cat{A}$ is non-degenerate if and only if $\cat{A}$ extends to a modular functor.
		\end{proposition}
		
		This is a generalization of a statement appearing in \cite[Theorem~5.5.1]{baki}
		from the semisimple to the non-semisimple case, at least morally, because strictly speaking it treats a different notion of modular functor	\cite[Remark 3.4.3]{damioliniwoike}.

		\begin{proof}[\slshape Proof of Proposition~\ref{propconnectedmodular}]
			If the braiding of $\cat{A}$ is non-degenerate, $\cat{A}$ is connected and extends to a modular functor~\cite[Corollary~8.3]{brochierwoike}.
			For the proof of the converse, 
			consider a genus one handlebody $H$ and a complement $H'$ in $\mathbb{S}^3$.
			 The embedding of admissible skeins~\cite{asm,brownhaioun} gives us natural maps (we denote by $\skcatA(-)$ the categories of admissible skeins)
			\begin{align}
				\skA(H;P)\otimes \skA(H';P^\vee)\to \skA(\mathbb{S}^3)\cong k \ , \quad P\in \skcatA(\mathbb{T}^2) \ , \label{eqndrinfeldcompose0}
				\end{align}
		that descend to the coend over $P$.
		After calculating the spaces in question via excision \cite{brownhaioun,rst,mwskein}, these maps are given by
			\begin{align}
				\cat{A}(I,\mathbb{F}\otimes P) \otimes \cat{A}(P,\mathbb{F})\to \cat{A}(P,P) \ra{\text{modified trace}} k \  \label{eqndrinfeldcompose}
				\end{align}
			that apply the Hopf pairing $\omega : \mathbb{F}\otimes\mathbb{F}\to I$, see e.g.~\cite{shimizumodular}, to the coend $\mathbb{F}=\int^{X\in\cat{A}} X^\vee \otimes X$ (because the skeins described by the dummy variable of the respective coends interlink geometrically in the way prescribed by the Hopf pairing)
			and apply the modified trace (this is again \cite[Corollary~3.2]{asm}).
			(Here the notation is adapted for the case in which $P$ is \emph{one} projective object in $\cat{A}$ placed on a disk in the torus; if we have several objects, we would need to tensor them together, but this does not add any insight, so we prefer to keep the notation readable.)
			This means that, after bringing $\cat{A}(P,\mathbb{F})$ to the right hand side via duality, and using the projective Calabi-Yau structure coming from the modified trace
			\begin{align}
				\cat{A}(P,\mathbb{F})^*\cong \cat{A}(\mathbb{F},P)\cong \cat{A}(I,\mathbb{A}\otimes P) \quad \text{with}\quad \mathbb{A}=\int_{X\in\cat{A}} X^\vee \otimes X=\mathbb{F}^\vee
				\end{align}
			 we obtain natural maps
			 \begin{align}
			 \label{eqnsktransf}	\skA(H;P)\cong \cat{A}(I,\mathbb{F}\otimes P) \to \cat{A}(I,\mathbb{A}\otimes P)\cong \skA(H';P)^*
			 \end{align}
		 induced by the Drinfeld map $\mathbb{D}:\mathbb{F}\to\mathbb{A}$ 
			\cite{drinfeld}.
			The latter is induced by the Hopf pairing;
			its component $X^\vee \otimes X \to Y^\vee \otimes Y$ comes from the 
			double braiding and is given in the graphical calculus (to be read from bottom to top) as follows:
			\begin{align}\begin{array}{c}
					\begin{tikzpicture}[scale=0.5]
						\begin{pgfonlayer}{nodelayer}
							\node [style=none] (10) at (3.75, -6.75) {{\footnotesize$X$}};
							\node [style=none] (11) at (2.5, -6.75) {{\footnotesize$X^\vee$}};
							\node [style=none] (12) at (6.75, 0.75) {{\footnotesize $Y^\vee$}};
							\node [style=none] (13) at (5.25, 0.75) {{\footnotesize$Y$}};
							\node [style=none] (14) at (3.75, -4.5) {};
							\node [style=none] (15) at (5.25, -4.5) {};
							\node [style=none] (16) at (3.75, -3) {};
							\node [style=none] (17) at (5.25, -3) {};
							\node [style=none] (18) at (3.75, -1.5) {};
							\node [style=none] (19) at (5.25, -1.5) {};
							\node [style=none] (20) at (2.5, -1.5) {};
							\node [style=none] (21) at (2.5, -4.5) {};
							\node [style=none] (22) at (6.5, -4.5) {};
							\node [style=none] (23) at (5.25, 0) {};
							\node [style=none] (24) at (6.5, 0) {};
							\node [style=none] (25) at (2.5, -6) {};
							\node [style=none] (26) at (3.75, -6) {};
						\end{pgfonlayer}
						\begin{pgfonlayer}{edgelayer}
							\draw [in=-90, out=90, looseness=0.75] (15.center) to (16.center);
							\draw [in=-90, out=90, looseness=0.75] (17.center) to (18.center);
							\draw [style=over, in=-90, out=90, looseness=0.75] (14.center) to (17.center);
							\draw [style=over, in=-90, out=90, looseness=0.75] (16.center) to (19.center);
							\draw [bend right=90, looseness=1.50] (18.center) to (20.center);
							\draw (20.center) to (21.center);
							\draw [bend right=90, looseness=1.25] (15.center) to (22.center);
							\draw [style] (21.center) to (25.center);
							\draw [style] (14.center) to (26.center);
							\draw [style] (23.center) to (19.center);
							\draw [style] (24.center) to (22.center);
						\end{pgfonlayer}
				\end{tikzpicture}\end{array}
				\label{eqndrinfeld}
			\end{align}
			The Drinfeld map is non-zero because it is a map of algebras
			(this is essentially \cite{drinfeld}, see also \cite[Section~3.3]{bjss} for an account in phrased in the language used here)
			between non-zero objects (this is well-known; for a proof, use e.g.~\cite[Proposition~4.7]{mwdehn} and $\mathbb{F}\cong\mathbb{A}$ thanks to unimodularity~\cite[Theorem~4.10]{shimizuunimodular}). Therefore the transformation~\eqref{eqnsktransf} is non-zero as well. By the Yoneda Lemma \eqref{eqnsktransf} is an isomorphism if and only if this is the case for $\mathbb{D}$.

			Through the connection between the generalized skein modules $\PhiA$ from \cite{brochierwoike} and the admissible skein modules and skein categories from \cite{asm,brownhaioun} proved in \cite{mwskein} under the assumption of unimodularity, \eqref{eqnsktransf} gives us via Lemma~\ref{lemmaconnecteddg} a non-zero natural transformation
			$\PhiA(H)\to \PhiA(H')\dg$ that by reflection equivariance of $\cat{A}$ and Proposition~\ref{propphimaps} gives us a non-zero map
			$\sigma : \PhiA(H) \to \PhiA(\bar{H}')$ that is an isomorphism if and only the Drinfeld map is an isomorphism.
			(When spelling out how
			$\sigma$ is defined, one will realize that it is an a priori non-invertible version of Lyubashenko's famous $S$-transformation \cite{lyubacmp,lyu,lyulex}, but we do not use this. In fact, we cannot use it because we do not know that $\cat{A}$ is modular yet --- this is what we want to prove after all.)

			Since $\cat{A}$ is connected, there is an isomorphism $\xi: \PhiA(\bar{H}')\ra{\cong} \PhiA(H)$, and $\xi\circ \sigma = \lambda \cdot \id_{\PhiA(H)}$ by Corollary~\ref{corendPhiA} for $\lambda \in k$. 
			Since $\sigma \neq 0$, we have $\lambda \neq 0$, and therefore $\sigma = \lambda \cdot \xi^{-1}$, i.e.\ the Drinfeld map is an isomorphism, which by the characterization of the non-degeneracy of the braiding in \cite{shimizumodular} means that $\cat{A}$ is modular.
		\end{proof}
	
	Proposition~\ref{propconnectedmodular} leads us to the following topological characterization of modular categories:
		
		\begin{theorem}\label{thmmodularcats}
			Genus zero restriction provides a canonical bijection between
			\begin{itemize}
				\item equivalence classes of $\Rexf$-valued modular functors which are strongly rigid, simple and have the property of admitting a reflection equivariant structure,
				
				\item and ribbon equivalence classes of modular categories. 
			\end{itemize}
			The inverse to genus zero restriction is the Lyubashenko construction.
		\end{theorem}
		
		\begin{proof}
			By the classification of modular functors~\cite[Theorem~6.8]{brochierwoike} and the characterization of reflection equivariant structures in the rigid and simple case in Theorem~\ref{thmrefleqmf},
			 equivalence classes of
			$\Rexf$-valued
			modular functors which are rigid, simple and have the property of admitting a reflection equivariant structure are in bijection to equivalence classes of unimodular finite ribbon categories which are connected, where the equivalence is equivalence as cyclic framed $E_2$-algebras.
			
			Therefore, two things remain to be observed:
			\begin{pnum}
				\item For a unimodular finite ribbon category, connectedness and the non-degeneracy of the braiding are equivalent by Proposition~\ref{propconnectedmodular}.
				\item We need to show that modular categories up to equivalence of cyclic framed $E_2$-algebras are the same as modular categories up to ribbon equivalence. This follows from~\cite[Corollary~4.5]{mwcenter} because of the triviality of the Müger center of modular categories.
			\end{pnum}
		\end{proof}

		In the semisimple case, this can be simplified using the rigidity result of Etingof-Penneys~\cite[Corollary~1.3]{etingofpenneys} saying that every finitely semisimple braided $r$-category is rigid. This is done using the semisimple version of Proposition~\ref{propconnectedmodular} in \cite{baki}.

		\spaceplease
		\begin{corollary}\label{corssimf}
			Genus zero restriction provides a canonical bijection between
			\begin{itemize}
				\item equivalence classes of semisimple modular functors which are simple and normalized in the sense that the vector space associated to $\mathbb{S}^2$ is one-dimensional,
				
				\item and ribbon equivalence classes of modular fusion categories. 
			\end{itemize}
			\end{corollary}
		
		A similar statement appears in \cite[Corollary~1.4]{etingofpenneys}, but since it a priori refers to a different notion of modular functor (see the remarks after Proposition~\ref{propconnectedmodular}), it is probably worth spelling out a proof that uses the rigidity result  of Etingof-Penneys, but not the definitions from \cite{baki}.

		\begin{proof}[\slshape Proof of Corollary~\ref{corssimf}]
			Suppose that $\mathfrak{F}$ is a $\Rexf$-valued modular functors whose cyclic framed $E_2$-algebra $\cat{A}$ obtained by evaluation on the circle is semisimple. The value on the sphere is the dual morphism space $\cat{A}(I,K)^*$ between $I$ and the dualizing object $K$. Since $I$ is simple, $\dim\, \cat{A}(I,K)^*=1$ implies $K\cong I$. Therefore, $\cat{A}$ is an $r$-category.
			Now \cite[Corollary 1.3]{etingofpenneys} tells us that $\cat{A}$ actually rigid, which by \cite[Proposition 2.3]{bd} tells us that the ribbon Grothendieck-Verdier duality is the rigid one. Since unimodularity is automatic in the semisimple setting \cite[Corollary 6.4]{eno-d}, $\mathfrak{F}$ admits a reflection equivariant structure. Therefore, the statement follows from Theorem~\ref{thmmodularcats}.
			\end{proof}

		\subsection{The main theorem on reflection equivariance}
		We are now ready to compose all the above results on reflection equivariance:

		\begin{theorem}\label{thmrefl}
				Let $\cat{A}$ be a cyclic framed $E_2$-algebra in $\Rexf$. Assume that $\cat{A}$ is strongly rigid and simple.
			Then the following structures are equivalent:
			\begin{pnum}
				\item The choice of a two-sided modified trace on $\Proj \cat{A}$, which is unique up to an element in $k^\times$. 
				\label{thmrefli}

				\item Cyclic reflection equivariance for $\cat{A}$, i.e.\ the structure of a $\mathbb{Z}_2$-homotopy fixed point of $\cat{A}$ for the homotopy involution $\cat{A}\mapsto \bar{\cat{A}}\dg$ on cyclic framed $E_2$-algebras relative to the canonical non-cyclic homotopy fixed point structure induced by the rigid pivotal duality of $\cat{A}$.  
				\label{thmreflii}

				\item Reflection equivariance for the unique ansular functor $\widehat{\cat{A}}$ extending $\cat{A}$, i.e.\ the choice of coherent isomorphisms
				\begin{align}
					\widehat{\cat{A}}(\bar H; P_1,\dots,P_n) \cong 	\widehat{\cat{A}}(H; DP_n,\dots,DP_1)^* \qquad \text{with}\qquad P_1,\dots,P_n \in \Proj \cat{A} \label{eqnisoH}
				\end{align}
				for every handlebody $H$ with at least one embedded disk per connected component,
				such that these isomorphisms 
				\begin{align}
					\left. \begin{array}{l} - \text{are compatible with gluing,} \\ - \text{are involutive,} \\ - \text{and extend, on the level of non-cyclic framed $E_2$-algebras,} \\ \phantom{-} \text{the canonical isomorphisms coming from the rigid pivotal duality.}\end{array}\right\} \label{eqnreflconditionsx} \tag{$**$}
				\end{align}
				\label{thmrefliiim}
			\end{pnum}
			Suppose further \begin{itemize}
				\item that any of the structures \ref{thmrefli}-\ref{thmrefliii} can be chosen and is fixed,
				\item and that $\cat{A}$ extends to a modular functor $\mathfrak{F}$.\end{itemize}
			Then 
			$\cat{A}$ has a non-degenerate braiding, i.e.\ it is modular, and $\mathfrak{F}$ agrees with the Lyubashenko modular functor for $\cat{A}$,
			and the isomorphisms~\eqref{eqnisoH} extend to isomorphisms
			\begin{align}
				\mathfrak{F}(\bar \Sigma; P_1,\dots,P_n) \cong \mathfrak{F}(\Sigma; DP_n,\dots,DP_1)^* \qquad \text{with}\qquad P_1,\dots,P_n \in \Proj \cat{A}\label{eqnisosmf}
			\end{align}
			for every surface $\Sigma$ with at least one boundary component per connected component for which the analoga of the conditions~\eqref{eqnreflconditions} hold.
		\end{theorem}

		\begin{proof}
			We  have to combine the previous results: Theorem~\ref{thmmain0} tells us \ref{thmrefli} $\Longleftrightarrow$ \ref{thmreflii}.
			Under the equivalence~\eqref{eqnequivcycframed0} between cyclic framed $E_2$-algebras and ansular functors, \ref{thmreflii} transforms equivalently into~\ref{thmrefliiim}, see also Remark~\ref{secansularfunctorrefl}. 
			
			Suppose now that we have either one of the equivalent structures in \ref{thmrefli}-\ref{thmrefliii} 
			 and  that $\cat{A}$ extends to a modular functor. Then this modular functor is also reflection equivariant by Theorem~\ref{thmrefleqmf}, which can be expressed through the isomorphisms~\eqref{eqnisosmf} subject to the analoga of the conditions \eqref{eqnreflconditionsx} as follows from Remark~\ref{remmfrefl}. The non-degeneracy of the braiding follows from Proposition~\ref{propconnectedmodular}.
			The remaining statements follow from Theorem~\ref{thmmodularcats}.
		\end{proof}

	\section{The Heisenberg picture for spaces of conformal blocks}
	In this section, we establish the Heisenberg picture for the spaces of conformal blocks of a logarithmic conformal field theory, still under assumptions of finiteness and rigidity.
	More precisely, we will show that the mapping class group representations on the spaces of conformal blocks are implicitly contained in the (internal) skein algebras and their representation theory.
	A thorough motivation and contextualization was given in the introduction, so we will get right to the proofs.

	\subsection{Mapping class group equivariance of the skein action\label{secequiv}}
	As recalled in Section~\ref{secmf},
	for any cyclic framed $E_2$-algebra $\cat{A}$,
	we can build a modular operad $\SurfA$ with 
	a map $p_\cat{A}:\SurfA \to \Surf$ of operads 
	and a modular $\SurfA$-algebra $\FA$.
	On an operation $(\Sigma,H)$ in $\SurfA$, the modular algebra	
	$\FA$ is given by $\PhiA(H)\circ \qss\cong \widehat{\cat{A}}(H) \in \cat{A}^{\boxtimes n}$ if $\Sigma$ is connected with $n$ boundary components, all seen as outgoing.
	This makes the value of $\FA$ on $(\Sigma,H)$, which we will denote by $\FA(\Sigma)$ by abuse of notation,
	also a module over the skein algebra (we denote the action again by $	\Add_{\Sigma,H}$). This way the canonical isomorphism
	$\widehat{\cat{A}}(H) \cong \FA(\Sigma)$ for a fixed handlebody $H$ becomes a $\Map(H)$-equivariant isomorphism and a skein module isomorphism.
	
	For a surface $\Sigma$ and $n\ge 0$ embedded intervals in its boundary, $\int_\Sigma \cat{A}$ becomes a module over $\cat{A}^{\boxtimes n}$.
	We will denote the action in the sequel by $\act : \cat{A}^{\boxtimes n}\boxtimes \int_\Sigma \cat{A}\to \int_\Sigma \cat{A}$, see also \cite[Section~5.2]{bzbj}. 
	
	\begin{theorem}\label{thmequiv}
		Let $\cat{A}$ be a cyclic framed $E_2$-algebra and $T$ a corolla with $n$ legs. Denote by $p_\cat{A}: \SurfA(T) \to \Surf(T)$ the projection functor.
		Let $(\Sigma,H)$ be a point in $\SurfA(T)$.
		Place an interval on each of the $n$ boundary components, thereby introducing an $\cat{A}^{\boxtimes n}$-action on $\int_\Sigma \cat{A}$. 
		Then the skein action \begin{align} \label{eqnactionmap}\Add_{\Sigma,H} : \Hom_{\int_\Sigma \cat{A}} (X \act \qss ,Y \act \qss)   
			\to \Hom_{\cat{A}^{\boxtimes n}}
			(X \otimes \FA(\Sigma),Y \otimes \FA(\Sigma)) \ , \quad X,Y\in\cat{A}^{\boxtimes n}
		\end{align}
		induced by the $(\int_{\mathbb{S}^1\times[0,1]}\cat{A})^{\boxtimes n}$-module map 
		$\PhiA(H): \int_\Sigma \cat{A}\to \cat{A}^{\boxtimes n}$ is natural in $(\Sigma,H)$. 
		If $\cat{A}$ is connected, the map~\eqref{eqnactionmap} is mapping class group equivariant. 
	\end{theorem}
	
Before giving the proof, let us spell out the case of main interest:
If $\cat{A}$ is connected and $X=Y=I$, then
the statement tells us that
\begin{align}
\Add_{\Sigma,H} : \SkAlg_\cat{A}(\Sigma)\to \End_{\cat{A}^{\boxtimes n}}(\FA(\Sigma))
\end{align}
is mapping class group equivariant. Often, the subscript of $\Add_{\Sigma,H}$ is dropped which is justified because the skein modules for different handlebodies to the same surface are isomorphic.
Let us spell this out:

		\begin{corollary}\label{corequiv}
	For any connected cyclic framed $E_2$-algebra, the skein action 
	\begin{align}
	\Add : \SkAlg_\cat{A}(\Sigma)\to \End_{\cat{A}^{\boxtimes n}}(\FA(\Sigma))
	\end{align}is mapping class group equivariant:
	\begin{align}
	\Add(f.a)=\FA\left(\widetilde f\right)\Add(a)\FA\left(\widetilde f\right)^{-1} \quad \text{for}\quad f \in \Map(\Sigma) \ , \quad a \in \SkAlg_\cat{A}(\Sigma) \ , \label{eqnmcgequiv}
	\end{align} where $f.a$ denotes the mapping class group action on the skein algebra element $a$ and $\widetilde f$ is any lift of $f$ to the extension of $\Map(\Sigma)$ that $\SurfA$ gives rise to.
In particular, the extension of $\Map(\Sigma)$ acts on $\FA(\Sigma)$ by operators in the normalizer of $\Add (      \SkAlg_\cat{A}(\Sigma)     )$. 
\end{corollary}

In the traditional semisimple 	skein-theoretic framework, \eqref{eqnmcgequiv} is classical~\cite{roberts,masbaumroberts}. 
Beyond this case, this seems to be new.
Also note that by the classification in \cite{brochierwoike} all modular functors are equivalent, in the moduli space of modular functors, to a modular functor of the form $\FA$. 
This means that \eqref{eqnmcgequiv} holds really for \emph{all} modular functors.
	
	As a preparation for the proof of Theorem~\ref{thmequiv}, we first make the following basic observation:
	
	\begin{lemma}\label{lemmaauxlemma}
		Let $f: X \to Y$ be a morphism in a bicategory and let $\alpha : F \to G$ be a 2-isomorphism between 1-morphisms $F,G:Y \to Z$.
		Then the diagram	
		
		\begin{equation}
			\begin{tikzcd}
				&\End(f)	 \ar[rrr,"F\circ -"] \ar[dd,"G\circ - ",swap] &&& \End(F\circ f) \ar[llldd," \text{conjugate with $\alpha$} "]  \\   \\ 
				&	\End(G\circ f)  
			\end{tikzcd}
		\end{equation}
		commutes, and the isomorphism conjugating with $\alpha$ does not depend on the choice of $\alpha$.
	\end{lemma}

	\begin{proof}[\slshape Proof of Theorem~\ref{thmequiv}]
		A morphism $\xi  : (\Sigma,H) \to (\Sigma',H')$ in $\SurfA(T)$ is a morphism $f:\Sigma \to \Sigma'$ in $\Surf(T)$
		plus an isomorphism $\alpha : \PhiA(H) \to \PhiA(H')f_*$. 
		Since $p_\cat{A}(\xi)=f$, we need to prove that the diagram
		\begin{equation}\label{eqnsqaurecomm}
			\begin{tikzcd}
				\Hom_{\int_\Sigma \cat{A}} (X \act \qss ,Y \act \qss)   	 \ar[rrr,"\Add_{\Sigma,H}"] \ar[dd,swap,"       (*)  "]   &&& \Hom_{\cat{A}^{\boxtimes n}}
				(X \otimes \FA(\Sigma),Y \otimes \FA(\Sigma))  \ar[dd,"\text{adjoint action with $\FA(\xi)$}"] \\  \\
				\Hom_{\int_{\Sigma'} \cat{A}} (X \act \qssp ,Y \act \qssp)     \ar[rrr,swap,"\Add_{\Sigma',H'}"] &&&       \Hom_{\cat{A}^{\boxtimes n}}
				(X \otimes \FA(\Sigma'),Y \otimes \FA(\Sigma'))
			\end{tikzcd}
		\end{equation}
		commutes.
		Here the map $(*)$ applies $f_*: \int_\Sigma \cat{A} \to \int_{\Sigma'} \cat{A}$, uses the homotopy fixed point structure $f_* \qss \cong \qssp$ and the $\cat{A}^{\boxtimes n}$-module map structure of $f_*$. 
		We prove the commutativity of the square~\eqref{eqnsqaurecomm}
		by breaking up the diagram into subdiagrams (MS $=$ module structure of $\PhiA(H)$
		 is being used, FPS $=$ homotopy fixed point structure is being used):
		
		\footnotesize
		\begin{equation}
			\begin{tikzcd}
				\Hom_{\int_\Sigma \cat{A}} (X \act \qss ,Y \act \qss) \ar[r,"\Add_{\Sigma,H}"] \ar[dd,"f_*\circ -"] 
				& \Hom_{\cat{A}^{\boxtimes n}}
				(X \otimes \FA(\Sigma),Y \otimes \FA(\Sigma))\cong \Hom_{\cat{A}^{\boxtimes n}}(
				\PhiA(H) (X\act \qss) ,\PhiA(H)(Y\act \qss) ) \ar[dd,"\text{conjugate with $\alpha$}"] 
				\\  \\
				\Hom_{\int_\Sigma \cat{A}} (f_*(X \act \qss) ,f_*(Y \act \qss)) \ar[dd,"\text{MS \& FPS}"] \ar[r," \PhiA(H')\circ -"] &       	\Hom_{\int_\Sigma \cat{A}} (\PhiA(H') \circ f_*(X \act \qss) ,\PhiA(H')\circ f_*(Y \act \qss))\ar[dd,swap,"\text{MS \& FPS}"] \\ \\
				\Hom_{\int_{\Sigma'} \cat{A}} (X \act \qssp ,Y \act \qssp)  \ar[r,swap,"\Add_{\Sigma',H'}"]  &       \Hom_{\cat{A}^{\boxtimes n}}
				(X \otimes \FA(\Sigma'),Y \otimes \FA(\Sigma'))
			\end{tikzcd}
		\end{equation}
		\normalsize By construction the re-composed diagram agrees with \eqref{eqnsqaurecomm}.
		The upper square commutes thanks to Lemma~\ref{lemmaauxlemma}. The lower square commutes by definition of $\Add_{\Sigma',H'}$.	
		
		For the addendum on the connected case, we use once again Lemma~\ref{lemmaauxlemma}.
	\end{proof}

	\subsection{Invertibility of handlebody skein modules\label{secinvskenmodule}}	According to \cite[Section~3.3]{bjss}, we call a cyclic framed $E_2$-algebra $\cat{A}$ \emph{cofactorizable} if \begin{align}\label{eqnphicyl} \PhiA(\mathbb{D}^2\times [0,1]): \int_{\mathbb{S}^1\times [0,1]} \cat{A} \to \cat{A}\boxtimes \cat{A}\end{align} is an equivalence (this is the equivalent definition from 
	\cite[Section~7.2]{brochierwoike}).

	\begin{remark}[Connection to the Drinfeld map]	Let us explain the abstractly defined $\PhiA(\mathbb{D}^2\times [0,1])$ a bit further in the more familiar special case in which $\cat{A}$ is a finite ribbon category: Then it follows that under the identifications \begin{itemize}
		\item $\int_{\mathbb{S}^1\times [0,1]} \cat{A}\simeq \rmod\mathbb{F}$~\cite[Theorem~5.14]{bzbj}, the category of right modules over
		 the coend algebra $\mathbb{F}=\int^{X\in\cat{A}}X^\vee \otimes X$ from \cite[Section~2]{lyu},
		\item	and $\cat{A}\boxtimes\cat{A}\simeq \End(\cat{A})\simeq \rmod\mathbb{A}$, see \cite[Section~2.7]{lyubook} and \cite[Lemma~3.5]{shimizupiv}, the category of right modules over the end algebra $\mathbb{A}=\int_{X\in\cat{A}} X^\vee \otimes X$,
	\end{itemize}
	the functor~\eqref{eqnphicyl}
	can be identified with the induction along the Drinfeld map, the algebra map $\mathbb{D}:\mathbb{F}\to\mathbb{A}$ described in~\eqref{eqndrinfeld}.
	In formulas:
	\begin{align}
		\label{eqndrinfeldmapphi}	\PhiA(\mathbb{D}^2\times [0,1])\cong \mathbb{D}_! : \rmod\mathbb{F}\to \rmod\mathbb{A} \ . 
	\end{align}
	The description~\eqref{eqndrinfeldmapphi} is by no means obvious, but follows as a consequence of the proof of \cite[Theorem~3.20]{bjss}. 
	\end{remark}

	We will prove a key technical result, namely the statement that co-non-degeneracy and cofactorizability are enough for all handlebody skein modules to be an equivalence. 
	
	\spaceplease
	\begin{theorem}\label{thminv}
		For a cyclic framed $E_2$-algebra $\cat{A}$, the following are equivalent:
		\begin{pnum}
			\item $\cat{A}$ is cofactorizable and co-non-degenerate.\label{thminv1}
			\item $\cat{A}$ is $\Phi$-invertible, i.e.\ $\PhiA(H)$ is an equivalence for all handlebodies.\label{thminv2}
		\end{pnum}
	\end{theorem}

By Proposition~\ref{propcodeg}, \cite[Theorem~1.6]{bjss} and  \cite{shimizumodular} modular categories are co-non-degenerate and cofactorizable. This implies:
\begin{corollary}\label{cormodcatphiinv}
	Modular categories are $\Phi$-invertible.
	\end{corollary}

Before going into the proof of Theorem~\ref{thminv}, let us mention that
for anomaly-free modular fusion categories Corollary~\ref{cormodcatphiinv} follows from
\cite[Theorem~1.1]{akz}. Without the assumption on the anomaly, it follows from \cite[Theorem 9.7]{kirillovtham}. 
For not necessarily semisimple modular categories, it was stated in the introduction of the first version of \cite{brochierwoike} without spelling out the proof (simply because $\Phi$-invertibility is not necessary to build a modular functor); the proof of  Corollary~\ref{cormodcatphiinv}, and actually the more general version in Theorem~\ref{thminv}, however, is essentially contained in a gluing property of the $\Omega_\cat{A}$-groupoids given in 
\cite[Proposition 7.12 \& 
Remark 7.14]{brochierwoike} and that is proven by induction on the genus.  
Nonetheless, since Theorem~\ref{thminv} is rather important, 
 we will spell out a the full argument for the induction on the genus in a self-contained way.
	
More has to be said about the context of Corollary~\ref{cormodcatphiinv}:	By \cite{bjss} a modular category gives rise to an invertible fully extended
	framed topological field theory through the use of the cobordism hypothesis.
	It is not known whether a modular category also comes with a homotopy fixed point structure under $\catf{SO}(4)$ to make this an oriented theory.
	However, 
	if this is true, one would expect the value on a handlebody to be $\PhiA$.
	The map $\PhiA$ being an equivalence would be in line with the invertibility of the tentative oriented topological field theory.

	Finally, we mention that a different possibility for the proof of Corollary~\ref{cormodcatphiinv} would be to use the results of \cite{mwcenter,mwskein} to translate it to a statement about admissible skeins and then proceed using results of \cite{skeintft,brownhaioun}.

	\begin{proof}[\slshape Proof of Theorem~\ref{thminv}]
		Clearly, \ref{thminv2} implies \ref{thminv1} by definition. 
		In order to prove \ref{thminv1} $\Longrightarrow$ \ref{thminv2}, let us assume that $\cat{A}$ is cofactorizable and co-non-degenerate.	
		We fix $n \ge 1$, consider a handlebody $H$ with boundary $\Sigma_{0,n}$, the sphere with $n$ boundary components, and choose a boundary component of $\Sigma_{0,n}$. We can now write $\Sigma_{0,n}$ by gluing a cylinder to $\Sigma_{0,n}$ at that boundary component, i.e.\ $\Sigma_{0,n}\cong \Sigma_{0,n} \cup_{\mathbb{S}^1} (\mathbb{S}^1 \times [0,1])$. 
		Suppose now we cap off both gluing boundaries of $\Sigma_{0,n} \cup_{\mathbb{S}^1} (\mathbb{S}^1 \times [0,1])$ with a disk. Then we obtain $\Sigma_{0,n-1} \sqcup \mathbb{D}^2$, and $H$ induces a handlebody $H'$ with boundary $\Sigma_{0,n-1}$.
		With the definition of the $\PhiA$ in terms of the ansular functor of $\cat{A}$, one observes
		that the diagram
		\begin{equation}\label{eqninduccof}
			\begin{tikzcd}
				\int_{\Sigma_{0,n}} \cat{A}     \ar[dd,swap,"\text{$\PhiA(H)$}"]	&&&   \int_{\Sigma_{0,n}} \cat{A}    \boxtimes_{\int_{\mathbb{S}^1 \times [0,1]} \cat{A}} \int_{\mathbb{S}^1 \times [0,1]} \cat{A}  \ar[lll," \simeq \  \text{via excision}  "]  \ar[dd,"\text{$\id \boxtimes \PhiA(\mathbb{D}^2 \times [0,1])$}"]   \\ \\ 
				\cat{A}^{\boxtimes n}       &&& \int_{\Sigma_{0,n}} \cat{A}  \boxtimes_{\int_{\mathbb{S}^1 \times [0,1]} \cat{A}} \int_{\mathbb{D}^2 \sqcup \mathbb{D}^2} \cat{A}       \ar[dd," \simeq \  \text{via excision}  "] 	   \\ \\ &&&     \int_{\Sigma_{0,n-1}} \cat{A} \boxtimes \cat{A}  \ar[uulll,"\PhiA(H') \boxtimes \id  "]   \ . 
			\end{tikzcd}
		\end{equation}
		commutes up to an isomorphism of $\left( \int_{\mathbb{S}^1 \times [0,1]} \cat{A}\right)^{\boxtimes n}$-module maps. 
		We can now show that
		\begin{align}
			\PhiA(H):\int_{\Sigma_{0,n}} \cat{A}  \to \cat{A}^{\boxtimes n}
		\end{align} 
		is an equivalence for $n\ge 0$. 
		For $n=0$, this is true by assumption because of co-non-degeneracy.
		For $n=1$, this is clear because the $\Phi$-map is the identity of $\cat{A}$ under the identification $\int_{\mathbb{D}^2}\cat{A}\simeq \cat{A}$. For $n=2$, it holds again by assumption thanks to cofactorizability.
		For $n\ge 3$, the statement follows inductively from~\eqref{eqninduccof} (this re-uses the statement for $n=2$). 
		So far, we have established that $\cat{A}$ is $\Phi$-invertible in genus zero.

		In the next step, we extend the statement that 
		\begin{align}
			\PhiA(H):\int_{\Sigma_{g,n}} \cat{A}  \to \cat{A}^{\boxtimes n}     \label{eqnPhiinvertequiv}
		\end{align} is an equivalence for any handlebody $H$ whose boundary is a surface $\Sigma_{g,n}$ of genus $g$ with $n\ge 0$ boundary components
		from the case $g=0$ to arbitrary genus.
		To this end, let $\Sigma$ be a connected surface with at least two boundary components and $H$ a handlebody with boundary $\Sigma$. Now we glue a cylinder to $\Sigma$. We attach its two boundary components to the two selected boundary components of $\Sigma$. We denote the resulting surface by $\Sigma'$. If $\Sigma$ has genus $g$, then $\Sigma'$ clearly has genus $g+1$. The handlebody $H$ with $\partial H=\Sigma$ induces after gluing a handlebody $H'$ with $\partial H'=\Sigma'$. 
		In order to finish the proof that~\eqref{eqnPhiinvertequiv} is always an equivalence, it remains to prove
		\begin{align}
			\PhiA(H) \ \text{is an equivalence} \quad \Longrightarrow \quad \PhiA(H') \ \text{is an equivalence} 
			\ . 
		\end{align}
		For the proof of this fact,
		we consider the following diagram:
		\begin{equation}\label{eqnnondegcof}
			\begin{tikzcd}
				\int_\Sigma \cat{A}  \boxtimes_{   \left( \int_{\mathbb{S}^1\times[0,1]} \cat{A} \right)^{\boxtimes 2} }  \int_{\mathbb{S}^1 \times [0,1]} \cat{A}  \simeq \int_{\Sigma'} \cat{A}    \ar[rrr,"  \PhiA(H')  "]\ar[dd,swap,"\text{$\PhiA(H)\boxtimes\PhiA(\mathbb{D}^2 \times [0,1])$}"]	&&&  \cat{A}^{\boxtimes n} \\ \\ 
				\cat{A}^{\boxtimes (n+2)}  \boxtimes_{   \left( \int_{\mathbb{S}^1\times[0,1]} \cat{A} \right)^{\boxtimes 2} }  \cat{A}^{\boxtimes 2}	           \ar[rrr, swap," \simeq \ \text{via excision}   "] \ar[rrruu,"\text{pairing}"] &&& \cat{A}^{\boxtimes n} \boxtimes \left(\int_{\mathbb{S}^2} \cat{A}\right)^{\boxtimes 2} \ar[uu,swap,"\cat{A}^{\boxtimes n}\boxtimes \PhiA(\mathbb{B}^3)^{\boxtimes 2}"] 	    \ . 
			\end{tikzcd}
		\end{equation}
		With the definition of the $\PhiA$
		and its gluing properties \cite[Section~5.1]{brochierwoike}, one can verify that the two triangles of the square commute.
		Now if $\PhiA(H)$ is an equivalence, then we can use that $\PhiA(\mathbb{D}^2 \times [0,1])$ and $\PhiA(\mathbb{B}^3)$ are equivalences 
		to conclude that $\PhiA(H')$ is an equivalence as well.
		This proves \ref{thminv2} $\Longleftrightarrow$ \ref{thminv1}.
	\end{proof}

	\begin{corollary}\label{corprog}
		For a modular category $\cat{A}$ and surface $\Sigma$ with $n\ge 0$ boundary components on each of which we fix one marked interval, the quantum structure sheaf $\qss \in \int_\Sigma \cat{A}$ is a projective generator of the $\cat{A}^{\boxtimes n}$-module $\int_\Sigma \cat{A}$.
		Moreover, for $n=0$, all objects in $\int_\Sigma \cat{A}$ are projective.
	\end{corollary}
	
	If $\Sigma$ has at least one boundary component per connected component,
	the statement is already true if $\cat{A}$ is  ribbon by \cite[Theorem~5.14]{bzbj}.
	
	\begin{proof}[\slshape Proof of Corollary~\ref{corprog}] As just mentioned, just the case of a closed surface $\Sigma$
		remains to be treated: Since $\int_\Sigma \cat{A}\ra{\simeq}\vect$ is an equivalence by Theorem~\ref{thminv}
		that by \cite[Theorem~4.2]{brochierwoike} sends $\qss$ to $\widehat{\cat{A}}(H)$ for a handlebody $H$ with $\partial H=\Sigma$, it suffices to observe that $\widehat{\cat{A}}(H)\in \vect $ is a projective generator. But this means just that $\widehat{\cat{A}}(H)$ is non-zero. This is indeed the case by \cite[Proposition 4.7~(i)]{mwdehn}. 
	\end{proof}

	\subsection{Representations of moduli algebras: closed surfaces}
	For the discussion of the representation theory of moduli algebras, let us discuss the case of closed surfaces first. In this case, it amounts to a statement about skein algebras:

		\begin{theorem}\label{thmskeinalg}
		Let $\mathfrak{F}$ be a $\Rexf$-valued modular functor whose value on the circle is the cyclic framed $E_2$-algebra $\cat{A}$.
		Suppose that $\mathfrak{F}$ is reflection equivariant relative to a rigid duality and simple, thereby making $\mathfrak{F}$ the Lyubashenko modular functor for the modular category $\cat{A}$.
		Then the following holds:
		\begin{pnum}
			\item
			For any closed surface $\Sigma$, the space of conformal blocks $\mathfrak{F}(\Sigma)$ is entirely characterized as follows:
			\begin{itemize}
				\item	$\mathfrak{F}(\Sigma)$ is the unique simple module over $\SkAlg_\cat{A}(\Sigma)$.
				\item The projective mapping class group action on $\mathfrak{F}(\Sigma)$ is the only one making the action mapping class group equivariant.\end{itemize}
			\item The mapping class group action on $\SkAlg_\cat{A}(\Sigma)$ is by inner automorphisms.
		\end{pnum}
	\end{theorem}
	
	Instead of proving Theorem~\ref{thmskeinalg}, we will prove the following, slightly more technical result that, in combination with
	Theorem~\ref{thmrefl}, yields immediately Theorem~\ref{thmskeinalg}:

	\begin{theorem}\label{thmskeinactioniso}
	Let $\cat{A}$ be a cofactorizable and co-non-degenerate cyclic framed $E_2$-algebra.
	Then the skein algebra action on the handlebody skein modules induces a skein algebra action on the spaces of conformal blocks that is an isomorphism
	\begin{align} \Add: \SkAlg_\cat{A}(\Sigma)\ra{\cong} \End_{\cat{A}^{\boxtimes n}} \FA (\Sigma)\end{align}
	of algebras with $\Map(\Sigma)$-action for any surface $\Sigma$. Here the action on the right hand side is the adjoint action.
	If $\Sigma$ is closed, the following holds:
	\begin{pnum}\item 	The skein algebra $\SkAlg_\cat{A}(\Sigma)$ is a matrix algebra, and the mapping class group acts by inner automorphisms. \label{thmskeinactionisoi}
		
		\item  A mapping class $f \in \Map(\Sigma)$ acts as a scalar multiple of the identity  $\FA(\Sigma)$ 
		if and only if $f$ acts by the identity on the skein algebra $\SkAlg_\cat{A}(\Sigma)$. \label{thmskeinactionisoii}	
	\end{pnum}
\end{theorem}

\begin{proof}
	The action map $ \SkAlg_\cat{A}(\Sigma)\to \End\, \FA (\Sigma)$ is an isomorphism since the $\PhiA$-maps are equivalences by Theorem~\ref{thminv}.
	The $\Map(\Sigma)$-equivariance is a consequence of Corollary~\ref{corequiv}. This proves the main statement.
	
	Statement~\ref{thmskeinactionisoi} follows now directly from the Skolem-Noether Theorem. 
	In~\ref{thmskeinactionisoii}, if $f\in \Map(\Sigma)$ acts trivially on the skein algebra, its acts trivially on $\End\, \FA (\Sigma)$ which means that it commutes with all endomorphisms of $\FA(\Sigma)$. This is only possible for multiples of the identity. The converse is clear.
\end{proof}
	
	\subsection{Representations of moduli algebras: surfaces with boundary}
	For surfaces that possibly have boundary components, we will need the \emph{internal skein algebras} (or \emph{moduli algebras}) in the sense of \cite{bzbj,bzbj2,skeinfin}: Suppose that $\Sigma$ is a surface with $n$ boundary components, one each with we fix an embedded interval.
	With the internal hom $\HOM(-,-):\left(\int_\Sigma\cat{A}\right)^\op\boxtimes \int_\Sigma \cat{A}\to \cat{A}^{\boxtimes n}$ for the action $\act : \cat{A}^{\boxtimes n}\boxtimes \int_\Sigma \cat{A}\to\int_\Sigma \cat{A}$, one may define the following \emph{internal skein algebra} or \emph{(generalized) moduli algebra}
	\begin{align}\moduli_\Sigma:= \END_{ \int_\Sigma \cat{A}}(\qss) \in \cat{A}^{\boxtimes n} \ ;
		\end{align}
	see \cite[Section~5.2]{bzbj} for the connection to the algebras in \cite{alekseevmoduli,agsmoduli,asmoduli}.
	We should warn the reader that the terminology of \cite{bzbj} that we use here is actually \emph{not} in line with \cite{alekseevmoduli,agsmoduli,asmoduli}. 
		
	\begin{theorem}\label{thmmodulishort}
		Let $\cat{A}$ be a modular category and
		$\Sigma$ a surface with $n\ge 0$ boundary components on each of which we fix one marked interval.
		\begin{pnum}
			\item 	Then the 
			conformal block functor for $\Sigma$, seen as object $\FA(\Sigma)$ in $\cat{A}^{\boxtimes n}$ via the cyclic structure, is a simple module over  $\moduli_\Sigma \in \cat{A}^{\boxtimes n}$.\label{thmmodulishorti}

			\item 	The $\moduli_\Sigma$-module $\FA(\Sigma)$ generates the $\moduli_\Sigma$-modules in $\cat{A}^{\boxtimes n}$ in the sense that any other module is given by tensoring $\FA(\Sigma) \otimes X$ with $X \in \cat{A}^{\boxtimes n}$.
			This establishes a Morita trivialization $\moduli_\Sigma\catf{-mod}_{\cat{A}^{\boxtimes n}} \simeq \cat{A}^{\boxtimes n}$
			of right $\cat{A}^{\boxtimes n}$-module categories. \label{thmmodulishortiii}
			
			\item  The $\Map(\Sigma)$-action on $\moduli_\Sigma$ is through inner automorphisms, and  the action $\moduli_\Sigma \ra{\cong} \FA(\Sigma) \otimes \FA(\Sigma)^\vee $ of $\moduli_\Sigma$ on $\FA(\Sigma)$ is by a $\Map(\Sigma)$-equivariant isomorphism.\label{thmmodulishortii}
		\end{pnum}
	\end{theorem}

For the proof of Theorem~\ref{thmmodulishort}, let us begin with the following standard observation:
	\begin{lemma}\label{lemmaprogenerator}
		Any non-zero object $X\in \cat{C}$ in a pivotal
		finite tensor category $\cat{C}$ is a progenerator for the left regular $\cat{C}$-module $\cat{C}$. 
		In particular,
		\begin{align}
			\cat{C} \ra{\simeq} \rmodc X\otimes  X^\vee \ , \quad Y \mapsto \END_{\cat{C}}(X,Y)\cong Y\otimes X^\vee 
		\end{align}
		is an equivalence, and hence $X^\vee$, as right $X\otimes X^\vee$-module, and analogously $X$ as left $X \otimes X^\vee$-module, are simple.
	\end{lemma}
	
	\begin{proof}
		Indeed, $\HOM_\cat{C}(X,-)=- \otimes X^\vee$ is exact. It is moreover faithful because for $Y,Z\in\cat{C}$ the map
		\begin{align}
			\cat{C}(Y,Z)\to	\cat{C}( Y\otimes X^\vee , Z\otimes X^\vee) \label{eqninj1}
		\end{align}
		can  be identified via duality
		with the map
		\begin{align}
			\cat{C}(Y,Z) \to \cat{C}(Y, Z \otimes X^\vee \otimes X) \label{eqninj2}
		\end{align} postcomposing with $\id_Z\otimes b_X :Z\to Z \otimes X^\vee \otimes X$. By simplicity of $I$ and $X\neq 0$ the coevaluation $b_X$ is a monomorphism, and so is $\id_Z\otimes b_X $ because the monoidal product is exact. Left exactness of $\cat{C}(Y,-)$ implies now the injectivity of \eqref{eqninj2} and \eqref{eqninj1}. This proves that $X$ is a progenerator. The rest follows from reconstruction of the module category from the progenerator~\cite[Chapter~7]{egno}.
	\end{proof}

	\begin{proof}[\slshape Proof of Theorem~\ref{thmmodulishort}]
		We can assume that $\Sigma$ is connected.
		For any handlebody $H$ with $\partial H=\Sigma$, the evaluation $\widehat{\cat{A}}(H)$ of the ansular functor for $\cat{A}$ on $H$, seen as object in $\cat{A}^{\boxtimes n}$ (this means that we see the $n$ boundary components as outgoing), inherits the structure of
		a module over $\ma$, with the action defined by 
		\begin{align}
			\label{eqnactionmapma}\ma = \END_{\int_\Sigma \cat{A}}(\qss) \ra{\PhiA(H)\circ - } \END_{\cat{A}^{\boxtimes n}}(\PhiA(H)\circ \qss)\cong \END_{\cat{A}^{\boxtimes n}}(\widehat{\cat{A}}(H))\cong \widehat{\cat{A}}(H) \otimes \widehat{\cat{A}}(H)^\vee \ . 
		\end{align}
		This uses that $\PhiA(H):\int_\Sigma \cat{A}\to \cat{A}$ is a $\int_{\partial \Sigma \times [0,1]}\cat{A}$-module map~\cite[Proposition~3.3]{brochierwoike} and therefore in particular an $\cat{A}^{\boxtimes n}$-module map, where the $\cat{A}$-action on $\cat{A}$ is the regular one.
		
		Since $\cat{A}$ is connected~\cite[Corollary~8.3]{brochierwoike}, \eqref{eqnactionmapma} does not depend on the choice of the handlebody; in other words, $\widehat{\cat{A}}(H)$ becomes a $\ma$-module that is, up to isomorphism,
		independent of $H$. 
		By the explanations at the beginning of Section~\ref{secequiv} this induces an $\ma$-module structure on $\FA(\Sigma)$; it is just transported along $\widehat{\cat{A}}(H)\cong \FA(\Sigma)$.  
		Now~\eqref{eqnactionmapma} gives us an algebra map $ \varphi : \ma \to \underline{\End}_{\cat{A}^{\boxtimes n}}(\FA(\Sigma)) 
		$. 	
		By \cite[Section~8.1 \& 8.2]{brochierwoike} $\FA(\Sigma)$ carries a projective $\Map(\Sigma)$-representation
		(that  agrees with the one constructed by Lyubashenko). 
		
		With the same arguments as for Theorem~\ref{thmequiv},
		the algebra map  $\varphi : \moduli_\Sigma \to \underline{\End}_{\cat{A}^{\boxtimes n}}(\FA(\Sigma)) $ is $\Map(\Sigma)$-equivariant,
		and by Theorem~\ref{thminv} it is an isomorphism of algebras.
		This proves~\ref{thmmodulishortii}.
		
		Thanks to $\moduli_\Sigma \cong \FA(\Sigma) \otimes \FA(\Sigma)^\vee $, we find
		 with 
		Lemma~\ref{lemmaprogenerator} that $\FA(\Sigma)$ is simple. The hypothesis $\FA(\Sigma)\neq 0$ is satisfied because of \cite[Proposition 4.7~(i)]{mwdehn}.
		This proves~\ref{thmmodulishorti}.

		By construction the
		$\moduli_\Sigma$-module $\FA(\Sigma)$ is the image of the progenerator $\qss \in \int_\Sigma \cat{A}$ from Lemma~\ref{corprog} under the $\cat{A}^{\boxtimes n}$-module equivalence $\PhiA(H):\int_\Sigma \cat{A} \ra{\simeq}\cat{A}^{\boxtimes n}$, seen here as an equivalence of right module categories.
		Now the	dual version of Lemma~\ref{lemmaprogenerator} using the right regular action	 implies~\ref{thmmodulishortiii}.
	\end{proof}

		\begin{remark}[Hopf-algebraic special case for one boundary component]\label{remhopfalg}	If $\Sigma$ has genus $g$ and \emph{one boundary component}, $\mathfrak{F}(\Sigma)$, as an object in $\cat{A}$, is the $g$-th tensor power $\mathbb{A}^{\otimes g}$ of the end $\mathbb{A}=\int_{X \in \cat{A}}X^\vee \otimes X$ in $\cat{A}$.
	In the case $\cat{A}=H\text{-mod}$ for a ribbon factorizable Hopf algebra $H$ (not necessarily semisimple), this is the $g$-th tensor power $H_\text{adj}^{\otimes g}$ of the adjoint representation~\cite[Theorem~7.4.13]{kl}.
	Let us explain how Theorem~\ref{thmmodulishort} recovers, in this Hopf-algebraic case, 
	 results of Faitg:
The fact that the action map produces an isomorphism
$\moduli_\Sigma \ra{\cong} \FA(\Sigma) \otimes \FA(\Sigma)^\vee $ 
in this case and that $\FA(\Sigma)$ is simple is established in \cite[Section~3.3]{faitg} (thereby making the action by inner automorphisms), see additionally~\cite[Section~4.2]{faitgderived} for the fact that these modules live actually inside $H$-modules. These facts are also used to describe the Lyubashenko representations.
	\end{remark}

In the presence of boundary components, it is worth having a version of the previous results that is not phrased in terms of internal skein algebras and internal skein modules, but rather directly for the \emph{values} of the modular functor on specific labels.
This can be easily achieved by considering endomorphism algebras in factorization homology of object \emph{other} than the quantum structure sheaf.

\begin{corollary}\label{corboundary}
	Let $\Sigma$ be a surface with $n$ boundary components and $\underline{\Sigma}$ the surface obtained from $\Sigma$ by deleting the boundary components that comes with an embedding $\varphi: (\mathbb{D}^2)^{\sqcup n} \to \underline{\Sigma}$ whose disks are placed at the position of the boundary components of $\Sigma$.
	For a modular category $\cat{A}$ and $X\in \cat{A}^{\boxtimes n}$, denote by $\underline{X}:= \varphi_* X$ the image of $X$ under $\varphi_* : \cat{A}^{\boxtimes n}\to \int_{\underline{\Sigma}} \cat{A}$. 
	Then the following holds true:
	\begin{pnum}
		\item The endomorphism algebra $\End_{\int_{\underline{\Sigma}}   \cat{A}}(\underline{X})$ is a matrix algebra and therefore has only one simple representation (unless $\End_{\int_{\underline{\Sigma}}   \cat{A}}(\underline{X})$ is the zero algebra). \label{corboundaryi}
		
		\item The algebra $\End_{\int_{\underline{\Sigma}}   \cat{A}}(\underline{X})$ comes with an action of the mapping class group of $\Sigma$ through inner automorphisms.
		
		\item The unique simple $\End_{\int_{\underline{\Sigma}}   \cat{A}}(\underline{X})$-module is the value $\FA(\Sigma;X)$ of the Lyubashenko modular functor on $\cat{A}$ with label $X$, and the projective $\Map(\Sigma)$-representation on $\FA(\Sigma;X)$ is the only one making the action map equivariant.
		\end{pnum}
	\end{corollary}

\begin{proof}
	Any handlebody $H$ with boundary $\underline{\Sigma}$ gives us by Corollary~\ref{cormodcatphiinv} an equivalence
	$\PhiA(H):\int_{\underline{\Sigma}}   \cat{A}\ra{\simeq}\vect$ sending $\underline{X}$ to $\FA(\Sigma;X)$, see~\eqref{defPhiA}.
	The rest of the proof repeats the arguments of the proof of Theorem~\ref{thmskeinactioniso}.
	\end{proof}

	\subsection{A few comments about non-internal skein algebras for surfaces with boundary}
	Theorem~\ref{thmmodulishort} treats the \emph{internal} skein algebra
	$\moduli_\Sigma = \END_{\int_\Sigma \cat{A}} (\qss)\in\cat{A}^{\boxtimes n}$ whose algebra of invariants
	$\Hom_{\cat{A}^{\boxtimes n}} (I,  \moduli_\Sigma  )=\SkAlg_\cat{A}(\Sigma)$ is the \emph{non-internal} skein algebra from~\eqref{eqnskeinalgebras}.
	About these, we can make the following statement in the case $n=1$:
	
	\begin{proposition}\label{propskeinone}
		For a modular category $\cat{A}$,
		the skein algebra $\SkAlg_\cat{A}(\Sigma)$ for a connected surface of genus $g$ and one boundary component is isomorphic to the endomorphism algebra $\End_\cat{A}(\mathbb{A}^{\otimes g})$ for $g$-th tensor power of the end $\mathbb{A}=\int_{X \in \cat{A}} X^\vee \otimes X$. 
	\end{proposition}
	
	\begin{proof}
		Theorem~\ref{thminv} together with \eqref{eqndefadd} gives us an algebra  isomorphism $\SkAlg_\cat{A}(\Sigma) \cong  \End\, \FA(\Sigma;-)$. Now the statement follows from \cite[Corollary 8.1]{brochierwoike}.
	\end{proof}

	\begin{example}[Non-semisimple non-internal skein algebras]\label{exHmod}
		Consider a modular category given by finite-dimensional modules over a ribbon factorizable Hopf algebra $H$.
		Then $\mathbb{A}=H_\text{ad}$ is the adjoint representation, see Remark~\ref{remhopfalg}.
		By Proposition~\ref{propskeinone} the skein algebra for the torus with one boundary component is $\End_H(H_\text{ad})$. 
		For the small quantum group $\bar U_q (\mathfrak{sl}_2)$ 
		at odd root of unity, a decomposition for the adjoint representation into indecomposable modules is given in \cite{ostrik} from which we deduce that $\End_H(H_\text{ad})$ is a non-semisimple algebra in this case. 
	\end{example}

	\begin{example}[The connection to the Torelli group]\label{extorelli}
		The skein algebras in Proposition~\ref{propskeinone} are not always non-semisimple, even if the category $\cat{A}$ is non-semisimple:
		If the mapping class group representations of $\cat{A}$ annihilate all Torelli groups (this would happen in Example~\ref{exHmod} for a commutative algebra $H$), then \cite[Proposition~6.7]{mwdehn}  tells us $\mathbb{A}\cong I^{\oplus n}$ for some $n\ge 1$. 
		Then the skein algebra for a surface of genus $g$ and one boundary component is the algebra of $ng \times ng$-matrices. 
	\end{example}

		\begin{example}[Skein algebras without graphical calculus --- the Feigin-Fuchs boson]\label{exffbs}
		The Feigin-Fuchs boson~\cite[Definition~3.1]{alsw} is the lattice vertex operator algebra which produces a modular functor
		whose circle category~$\catf{VM}(\Psi)$ is not modular, but still cofactorizable and with non-degenerate braiding \cite[Example 8.8]{brochierwoike}
		(the latter two properties are actually equivalent here because the monoidal product has the property of being rigid, but the category would not be ribbon with the rigid duality).
		The underlying ribbon Grothendieck-Verdier category is, as an abelian category, the category of finite-dimensional vector spaces graded over a finite abelian group $G$ (obtained as the quotient of the dual lattice  by the lattice; these details are not vital for this example).
		The dualizing object is $k_{2\xi}$, the ground field supported at degree $2\xi$ for some $\xi \neq 0$ in $G$. 
		The space of conformal block at  a closed surface $\Sigma_g$ of genus $g$ has dimension $|G|^g$ if $2(g-1)\xi=0$, and zero otherwise. 
		Since $\catf{VM}(\Psi)$ is not a ribbon category, there is presently no graphical calculus available that one could use for skein-theoretic methods, see however the recent progress in \cite{mdcs}. In this example, we want to demonstrate that the description of the spaces of conformal blocks in the Heisenberg picture along the lines of Theorem~\ref{thmskeinalg} surprisingly still goes through:
		Thanks to Proposition~\ref{propcodeg}, $\catf{VM}(\Psi)$ is co-non-degenerate, but careful: The generator of $\int_{\mathbb{S}^2} \catf{VM}(\Psi)$ is the dualizing object $k_{2{\xi}}$ placed on a disk in the sphere, whereas quantum structure sheaf $\cat{O}_{\mathbb{S}^2}^{\catf{VM}(\Psi)}$
		is zero.
		With Theorem~\ref{thmskeinactioniso}, we find now
		\begin{align}
			\SkAlg_{\catf{VM}(\Psi)}(\Sigma_g)\cong \left\{ \begin{array}{ll} \catf{Mat}(|G|^g,k) \ , &  \text{if}\ 2(g-1)\xi=0 \ ,  \\  0 \ , & \text{else,} \end{array} \right.
		\end{align}
	and this matrix algebra is actually the endomorphism algebra of the space of conformal blocks. So the space of conformal blocks can, even $\Map(\Sigma_g)$-equivariantly, be characterized as the unique simple module over $\SkAlg_{\catf{VM}(\Psi)}(\Sigma_g)$ (where however we would exclude the trivial cases in which $2(g-1)\xi\neq 0$). 
	\end{example}

	\subsection{A comment on vertex operator algebras\label{subsecvoa}}
	The Heisenberg picture for spaces of conformal blocks in the sense of	Theorem~\ref{thmskeinalg}
	carries over to the framework  of vertex operator algebras:
	Let $V$ be a strongly rational, self-dual vertex operator algebra, and let $\mathbb{V}_g$ the space of conformal blocks for genus $g$ (no non-trivial insertions) in the sense of \cite{frenkelbenzvi,DGT1,DGT2}.
	This is a projectively flat vector bundle  on the moduli space of genus $g$ curves. Its connection is (a generalization of) the Knizhnik-Zamolodchikov connection.
	
	By the main result of \cite{damioliniwoike} $\mathbb{V}$, for varying curves and insertions, forms a modular functor. To the circle, it associates the category $\cat{C}_V$ of admissible $V$-modules that, via the evaluation on genus zero curves, inherits the structure of a modular fusion category.
	Set now 
	\begin{align}
		\mathbb{S}_g :=      \End_{\int_{\Sigma_g} \cat{C}_V}(\cat{O}_{\Sigma_g}^{\cat{C}_V}) \ . 
		\end{align}
	With the mapping class group representation induced by the homotopy fixed point structure of the quantum structure sheaf, $\mathbb{S}_g$ can be seen as a flat algebra bundle on the moduli space of genus $g$ curves. This is just an application of the usual Riemann-Hilbert correspondence~\cite[Section~6]{baki} for spaces of conformal blocks.
	 
	Then Theorem~\ref{thmskeinalg} gives us the following result:
	
	\begin{corollary}\label{corvoa}
		There is an isomorphism \begin{align}\label{eqnisoconnection} \mathbb{S}_g \ra{\cong} \mathbb{V}_g^* \otimes \mathbb{V}_g\end{align} of flat vector bundles.
		\end{corollary}
	
	In fact, the projectively flat connection on $\mathbb{V}_g$
	is the unique one such that the isomorphism of vector bundles~\eqref{eqnisoconnection} is compatible with the flat connection. 
	Similar statements can be made in the presence of boundary components, but
	 we will not spell these out here.
	
	Corollary~\ref{corvoa} should be interpreted in the sense that one can 
	actually build a flat algebra bundle $\mathbb{S}_g$ (to be thought of as algebras of observables) 
	that acts on $\mathbb{V}_g$ in way that the projectively 
	flat connection on $\mathbb{V}_g$ is pinned down by the requirement that the action be compatible with the connection. 
	
	This confirms the expectation formulated in \cite[Section~4.3]{asmoduli} that the description of spaces of conformal blocks
	in terms of algebras of observables extends from modular fusion categories obtained as modules over semisimple Hopf algebras to   spaces of conformal blocks constructed geometrically as projectively flat vector bundles over a moduli space of curves.
	Of course, the precise formulation that we prove takes the liberty to extrapolate
	the expectations in \cite[Section~4.3]{asmoduli} using the framework in \cite{frenkelbenzvi} which was not available at the time. 
	
	Beyond the rational case, the modular functors built from a modular category 
	allow in principle for a generalization of Corollary~\ref{corvoa} (after all, the novelty of Theorem~\ref{thmskeinalg} is exactly the non-semisimple case). The problem is that, beyond the rational case, the relation to \cite{frenkelbenzvi,DGT1,DGT2} is unknown. The main issue is that it is generally not known whether $\mathbb{V}$ is a modular functor.

	\subsection{A model for the surface operad and the vanishing of the anomaly}
	We will end the article with an important application that will be relevant in follow-up work:
	Adapted to a modular category, we will produce a simple model for the surface operad (not an extension thereof).

	To this end, consider for a modular category $\cat{A}$  the extension
	$\Surf_{\! \cat{A}\dg \boxtimes \cat{A}}\to \Surf$ over which the modular functor for $\cat{A}\dg \boxtimes \cat{A}$ lives, see Section~\ref{secmf}
	(see also Section~\ref{secmonoidalproductofalgebras} where the product of cyclic and modular algebras was defined).
	We show that this extension  splits in a preferred way (in fact, we prove a slightly more general statement):

	\begin{lemma}\label{propanomalycancellation}
		For a co-non-degenerate, cofactorizable cyclic framed $E_2$-algebra $\cat{A}$ in $\Rexf$ that is self-injective and has an exact monoidal product, the extension $\Surf_{\! \cat{A}\dg \boxtimes \cat{A}}\to \Surf$ splits in a preferred way, and hence
		the modular functor for $\cat{A}\dg \boxtimes \cat{A}$ 
		is naturally a modular $\Surf$-algebra. In other words, it is anomaly-free. 
	\end{lemma}
	
\begin{remark}	If $\cat{A}$ is modular, this follows also from \cite[Corollary~8.3]{sn}, where the statement is phrased in terms of the Drinfeld center of $\cat{A}$, but this does not give us the splitting explicitly in terms of $\cat{A}\dg \boxtimes\cat{A}$. 
\end{remark}
	
	As a preparation for the proof of Proposition~\ref{propanomalycancellation}, let us build the model for $\Surf$ over which the modular functor lives:
	Define 
	$\Surf_{\! \cat{A}\dg \boxtimes \cat{A}}^\circ \subset \Surf_{\! \cat{A}\dg \boxtimes \cat{A}}$ as the  suboperad of $\Surf_{\! \cat{A}\dg \boxtimes \cat{A}}$ fibered over $\Surf$ whose homotopy fiber over a connected  surface $\Sigma$ with $n$ boundary components
	is given by all $\PhiA(H)\dg \boxtimes \PhiA(H)$ for handlebodies $H$ with $\partial H=\Sigma$. By abuse of notation, we identify $\PhiA(H)\dg \boxtimes \PhiA(H)$ here with the map
	\begin{align}
		\int_\Sigma (\cat{A}\dg \boxtimes \cat{A}) \simeq \int_\Sigma \cat{A}\dg  \boxtimes \int_\Sigma \cat{A} \ra{\PhiA(H)\dg \boxtimes \PhiA(H)} \cat{A}^{\boxtimes n} \boxtimes \cat{A}^{\boxtimes n} \simeq (\cat{A}\boxtimes \cat{A})^{\boxtimes n} \ , 
	\end{align}
	where in the last step we use the symmetric braiding for $\boxtimes$ (because of the definitions in Section~\ref{secmonoidalproductofalgebras}).
	The morphisms are
	all maps \begin{align}\label{eqnisolambda} \PhiA(H)\dg \boxtimes \PhiA(H) \ra{\tau\dg \boxtimes \tau} \PhiA(H')\dg \boxtimes \PhiA(H')\end{align} for all module isomorphisms $\tau : \PhiA(H)\to\PhiA(H')$.
	We call an isomorphism of the form~\eqref{eqnisolambda} a \emph{diagonal skein isomorphism}.

	\begin{proof}[\slshape Proof of Proposition~\ref{propanomalycancellation}]
		By connectedness of $\cat{A}$, the fibers of $\Surf_{\! \cat{A}\dg \boxtimes \cat{A}}^\circ$ are connected.
		Any module automorphism of $\PhiA(H)$ is by $\Phi$-invertibility (Theorem~\ref{thminv}) and \cite[Proposition 7.11, 7.12 \& Remark~7.14]{brochierwoike} of the form $\lambda \cdot \id$ with $\lambda \in k^\times$. By construction $(\lambda \cdot \id)\dg = \lambda^{-1} \cdot \id$, which implies that the isomorphisms \eqref{eqnisolambda} do not depend on the choice of $\tau$. Therefore, the homotopy fibers of $\Surf_{\! \cat{A}\dg \boxtimes \cat{A}}^\circ\to \Surf$ are connected and simply connected, which makes $	\Surf_{\! \cat{A}\dg \boxtimes \cat{A}}^\circ\to \Surf$ an equivalence. In other words, $\Surf_{\! \cat{A}\dg \boxtimes \cat{A}}^\circ$ is a model for the surface operad, and $\Surf \simeq \Surf_{\! \cat{A}\dg \boxtimes \cat{A}}^\circ \to \Surf_{\! \cat{A}\dg \boxtimes \cat{A}}$ is the desired splitting.
		
		The modular functor $\mathfrak{F}_{\cat{A}\dg \boxtimes \cat{A}}$ for $\cat{A}\dg \boxtimes \cat{A}$ is a modular $\Surf_{\! \cat{A}\dg \boxtimes \cat{A}}$-algebra. The restriction of $\mathfrak{F}_{\cat{A}\dg \boxtimes \cat{A}}$ along $\Surf_{\! \cat{A}\dg \boxtimes \cat{A}}^\circ\subset \Surf_{\! \cat{A}\dg \boxtimes \cat{A}} $ is the anomaly-free modular functor that we are looking for.  
	\end{proof}

	The definition of $\Surf_{\! \cat{A}\dg \boxtimes \cat{A}}^\circ$
	and the description of $\SurfA$ in \cite[Section~5.2]{brochierwoike} (that relies, among other things, on the Lego Teichmüller game~\cite{bakifm}) immediately imply the following statement:
	
	\begin{proposition}
		With assumptions as in~\ref{propanomalycancellation}, the modular functor for
		$\cat{A}\dg \boxtimes \cat{A}$ lives naturally over a model for $\Surf$ that, at the level of morphisms in groupoids of operations, is generated by handlebody maps in genus zero and one of the following equivalent choices for further generating morphisms:
		\begin{pnum}
			\item Diagonal skein module isomorphisms~\eqref{eqnisolambda} for the torus with one boundary component.
			\item The mapping classes of the torus with one boundary component.
			\end{pnum}
	\end{proposition}
	
	While modular functors are uniquely determined by genus zero, see \cite[Section~6.3]{brochierwoike} generalizing \cite{andersenueno}, for modular $\Surf$-algebras, it is a priori possible that there is still data in genus one. Of course, this is not a contradiction because the 2-groupoids of modular functors and modular $\Surf$-algebras are \emph{not} the same
	(neither for the objects nor for the morphisms). 
	An abstract argument why, for a modular $\Surf$-algebra with values in a symmetric monoidal \emph{bi}category, no data needs to be provided beyond genus one is given in~\cite[Corollary~5.21]{steinebrunner}.

\small	
\newcommand{\etalchar}[1]{$^{#1}$}

	\vspace*{0.3cm} \noindent  \textsc{Université Bourgogne Europe, CNRS, IMB UMR 5584, F-21000 Dijon, France}
	

\begin{thebibliography}{CGHPM23}
	
	\bibitem[AF15]{AF}
	D.~Ayala and J.~Francis.
	\newblock {Factorization homology of topological manifolds}.
	\newblock {\em J. Top.}, 8(4):1045--1084, 2015.
	
	\bibitem[AGS96]{agsmoduli}
	A.~Alekseev, H.~Grosse, and V.~Schomerus.
	\newblock {Combinatorial quantization of the {H}amiltonian {C}hern-{S}imons
		theory {II}}.
	\newblock {\em Comm. Math. Phys.}, 174(3):561--604, 1996.
	
	\bibitem[AKZ17]{akz}
	Y.~Ai, L.~Kong, and H.~Zheng.
	\newblock {Topological orders and factorization homology}.
	\newblock {\em Adv. Theor. Math. Phys.}, 21(8):1845--1894, 2017.
	
	\bibitem[Ale94]{alekseevmoduli}
	A.~Alekseev.
	\newblock {Integrability on the {H}amiltonian {C}hern-{S}imons theory}.
	\newblock {\em St. Petersburg Math. J.}, 6(2):241--253, 1994.
	
	\bibitem[ALSW25]{alsw}
	R.~Allen, S.~Lentner, C.~Schweigert, and S.~Wood.
	\newblock {Duality structures for module categories of vertex operator algebras
		and the {F}eigin {F}uchs boson}.
	\newblock {\em Selecta Math. New Ser.}, 31(36), 2025.
	
	\bibitem[ARS95]{ars}
	M.~Auslander, I.~Reiten, and S.~O. Smal\o.
	\newblock {\em Representation Theory of Artin Algebras}, volume~36 of {\em
		Cambridge Studies Adv. Math.}
	\newblock Cambridge University Press, 1995.
	
	\bibitem[AS96]{asmoduli}
	A.~Alekseev and V.~Schomerus.
	\newblock {Representation theory of {C}hern-{S}imons observables}.
	\newblock {\em Duke Math. J.}, 85(2):447--510, 1996.
	
	\bibitem[AU12]{andersenueno}
	J.~E. Andersen and K.~Ueno.
	\newblock {Modular functors are determined by their genus zero data}.
	\newblock {\em Quantum Topol.}, 3(4):255--291, 2012.
	
	\bibitem[Bar79]{barr}
	M.~Barr.
	\newblock {\em $\star$-autonomous categories}, volume 572 of {\em Lecture Notes
		in Math.}
	\newblock Springer, 1979.
	
	\bibitem[BD04]{bdca}
	A.~Beilinson and V.~Drinfeld.
	\newblock {\em Chiral algebras}, volume~51 of {\em American Mathematical
		Society Colloquium Publications}.
	\newblock American Mathematical Society, Providence, RI, 2004.
	
	\bibitem[BD13]{bd}
	M.~Boyarchenko and V.~Drinfeld.
	\newblock A duality formalism in the spirit of {G}rothendieck and {V}erdier.
	\newblock {\em Quantum Top.}, 4(4):447--489, 2013.
	
	\bibitem[BDSPV15]{BDSPV15}
	B.~Bartlett, C.~L. Douglas, C.~Schommer-Pries, and J.~Vicary.
	\newblock Modular categories as representations of the 3-dimensional bordism
	category.
	\newblock arXiv:1509.06811 [math.AT], 2015.
	
	\bibitem[BH24]{brownhaioun}
	J.~Brown and B.~Ha\"ioun.
	\newblock Skein categories in non-semisimple settings.
	\newblock arXiv:2406.08956 [math.QA], 2024.
	
	\bibitem[BJSS21]{bjss}
	A.~Brochier, D.~Jordan, P.~Safronov, and N.~Snyder.
	\newblock {Invertible braided tensor categories}.
	\newblock {\em Alg. Geom. Top.}, 21(4):2107--2140, 2021.
	
	\bibitem[BK00]{bakifm}
	B.~Bakalov and A.~Kirillov.
	\newblock {On the {L}ego-{T}eichmüller game}.
	\newblock {\em Transf. Groups}, 6:207--244, 2000.
	
	\bibitem[BK01]{baki}
	B.~Bakalov and A.~Kirillov, Jr.
	\newblock {\em Lectures on tensor categories and modular functors}, volume~21
	of {\em University Lecture Series}.
	\newblock American Mathematical Society, Providence, RI, 2001.
	
	\bibitem[BR95]{br1}
	E.~Buffenoir and P.~Roche.
	\newblock {Two dimensional lattice gauge theory based on a quantum group}.
	\newblock {\em Comm. Math. Phys.}, 170(3):669--698, 1995.
	
	\bibitem[BR96]{br2}
	E.~Buffenoir and P.~Roche.
	\newblock {Link invariants and combinatorial quantization of Hamiltonian Chern
		Simons theory}.
	\newblock {\em Comm. Math. Phys.}, 181(2):331--365, 1996.
	
	\bibitem[BV68]{bv68}
	J.~M. Boardman and R.~M. Vogt.
	\newblock {Homotopy-everything $H$-spaces}.
	\newblock {\em Bull. Amer. Math. Soc.}, 74(6):1117--1122, 1968.
	
	\bibitem[BV73]{bv73}
	J.~M. Boardman and R.~M. Vogt.
	\newblock {\em Homotopy invariant algebraic structures on topological spaces},
	volume 347 of {\em Lecture Notes in Math.}
	\newblock Springer, 1973.
	
	\bibitem[BW22]{brochierwoike}
	A.~Brochier and L.~Woike.
	\newblock A classification of modular functors via factorization homology.
	\newblock arXiv:2212.11259 [math.QA], 2022.
	
	\bibitem[BZBJ18a]{bzbj}
	D.~Ben-Zvi, A.~Brochier, and D.~Jordan.
	\newblock {Integrating quantum groups over surfaces}.
	\newblock {\em J. Top.}, 11(4):874--917, 2018.
	
	\bibitem[BZBJ18b]{bzbj2}
	D.~Ben-Zvi, A.~Brochier, and D.~Jordan.
	\newblock {Quantum character varieties and braided module categories}.
	\newblock {\em Selecta Math.}, 24:4711--4748, 2018.
	
	\bibitem[CGHPM23]{skeintft}
	F.~Costantino, N.~Geer, B.~Ha\"ioun, and B.~Patureau-Mirand.
	\newblock {Skein (3+1)-{TQFTs} from non-semisimple ribbon categories}.
	\newblock arXiv:2306.03225 [math.GT], 2023.
	
	\bibitem[CGPM23]{asm}
	F.~Costantino, N.~Geer, and B.~Patureau-Mirand.
	\newblock Admissible skein modules.
	\newblock arXiv:2302.04493 [math.GT], 2023.
	
	\bibitem[Coo23]{cooke}
	J.~Cooke.
	\newblock Excision of skein categories and factorisation homology.
	\newblock {\em Adv. Math.}, 414:108848, 2023.
	
	\bibitem[Cos04]{costello}
	K.~Costello.
	\newblock The {A}-infinity operad and the moduli space of curves.
	\newblock arXiv:math/0402015 [math.AG], 2004.
	
	\bibitem[DGT21]{DGT1}
	Chiara Damiolini, Angela Gibney, and Nicola Tarasca.
	\newblock Conformal blocks from vertex algebras and their connections on
	{$\overline{\mathcal M}_{g, n}$}.
	\newblock {\em Geom. Topol.}, 25(5):2235--2286, 2021.
	
	\bibitem[DGT24]{DGT2}
	Chiara Damiolini, Angela Gibney, and Nicola Tarasca.
	\newblock On factorization and vector bundles of conformal blocks from vertex
	algebras.
	\newblock {\em Ann. Sci. \'Ec. Norm. Sup\'er. (4)}, 57(1):241--292, 2024.
	
	\bibitem[DL07]{daylack}
	B.~Day and S.~Lack.
	\newblock Limits of small functors.
	\newblock {\em J. Pure Appl. Alg.}, 210(3):651--683, 2007.
	
	\bibitem[DRGG{\etalchar{+}}22]{gai2}
	M.~De~Renzi, A.~M. Gainutdinov, N.~Geer, B.~Patureau-Mirand, and I.~Runkel.
	\newblock {Mapping class group representations from non-semisimple {TQFT}s}.
	\newblock {\em Comm. Contemp. Math.}, 2150091(25):01, 2022.
	
	\bibitem[Dri90]{drinfeld}
	V.~Drinfeld.
	\newblock {On almost commutative Hopf algebras}.
	\newblock {\em Leningrad Math. J.}, 1(2):321--342, 1990.
	
	\bibitem[DS25]{mdcs}
	M.~Demirdilek and C.~Schweigert.
	\newblock Surface diagrams for {F}robenius algebras and {F}robenius-{S}chur
	indicators in grothendieck-verdier categories.
	\newblock arXiv:2503.13325 [math.CT], 2025.
	
	\bibitem[DSPS19]{dss}
	C.~L. Douglas, C.~Schommer-Pries, and N.~Snyder.
	\newblock {The balanced tensor product of module categories}.
	\newblock {\em Kyoto J. Math.}, 59(1):167--179, 2019.
	
	\bibitem[DW25]{damioliniwoike}
	C.~Damiolini and L.~Woike.
	\newblock Modular functors from conformal blocks of rational vertex operator
	algebras.
	\newblock arXiv:2507.05845 [math.QA], 2025.
	
	\bibitem[EGNO15]{egno}
	P.~Etingof, S.~Gelaki, D.~Nikshych, and V.~Ostrik.
	\newblock {\em Tensor categories}, volume 205 of {\em Math. Surveys Monogr.}
	\newblock Am. Math. Soc., 2015.
	
	\bibitem[ENO04]{eno-d}
	P.~Etingof, D.~Nikshych, and V.~Ostrik.
	\newblock {An analogue of Radford's $S^4$ formula for finite tensor
		categories}.
	\newblock {\em Int. Math. Res. Not.}, 2004(54):2915--2933, 2004.
	
	\bibitem[EO04]{etingofostrik}
	P.~Etingof and V.~Ostrik.
	\newblock {Finite tensor categories}.
	\newblock {\em Mosc. Math. J.}, 4(3):627--654, 2004.
	
	\bibitem[EP24]{etingofpenneys}
	P.~Etingof and D.~Penneys.
	\newblock Rigidity of non-negligible objects of moderate growth in braided
	categories.
	\newblock arXiv:2412.17681 [math.QA], 2024.
	
	\bibitem[Fai20]{faitg}
	M.~Faitg.
	\newblock {Projective representations of mapping class groups in combinatorial
		quantization}.
	\newblock {\em Comm. Math. Phys.}, 377:161--198, 2020.
	
	\bibitem[Fai25]{faitgderived}
	M.~Faitg.
	\newblock Derived representations of quantum character varieties.
	\newblock arXiv:2502.04267 [math.QA], 2025.
	
	\bibitem[FBZ04]{frenkelbenzvi}
	E.~Frenkel and D.~Ben-Zvi.
	\newblock {\em Vertex algebras and algebraic curves}, volume~88 of {\em
		Mathematical Surveys and Monographs}.
	\newblock American Mathematical Society, Providence, RI, second edition, 2004.
	
	\bibitem[FH21]{freedhopkins}
	D.~Freed and M.~Hopkins.
	\newblock {Reflection positivity and invertible topological phases}.
	\newblock {\em Geom. Top.}, 25:1165--1330, 2021.
	
	\bibitem[FHL93]{fhl}
	I.~B. Frenkel, Y.-Z. Huang, and L.~Lepowsky.
	\newblock On axiomatic approaches to vertex operator algebras and modules.
	\newblock {\em Mem. Amer. Math. Soc.}, 104(494):viii+64, 1993.
	
	\bibitem[FKBL19]{fkbl}
	C.~Frohman, J.~Kania-Bartoszynska, and T.~Lê.
	\newblock {Unicity for representations of the Kauffman bracket skein algebra}.
	\newblock {\em Invent. Math.}, 215:609--650, 2019.
	
	\bibitem[FM12]{farbmargalit}
	B.~Farb and D.~Margalit.
	\newblock {\em A Primer on Mapping Class Groups}, volume~49 of {\em Princeton
		Math. Series}.
	\newblock Princeton University Press, 2012.
	
	\bibitem[Fre17]{FresseI}
	B.~Fresse.
	\newblock {\em Homotopy of operads and Grothendieck-Teichm\"uller groups. Part
		1: The Algebraic Theory and its Topological Background}, volume 217 of {\em
		Math. Surveys and Monogr.}
	\newblock Am. Math. Soc., 2017.
	
	\bibitem[FS17]{jfcs}
	J.~Fuchs and C.~Schweigert.
	\newblock {Consistent systems of correlators in non-semisimple conformal field
		theory}.
	\newblock {\em Adv. Math.}, 307:598--639, 2017.
	
	\bibitem[FSS20]{fss}
	J.~Fuchs, G.~Schaumann, and C.~Schweigert.
	\newblock {Eilenberg-Watts calculus for finite categories and a bimodule
		Radford $S^4$ theorem}.
	\newblock {\em Trans. Amer. Math. Soc.}, 373(1):1--40, 2020.
	
	\bibitem[FSWY25]{algcften}
	J.~Fuchs, C.~Schweigert, S.~Wood, and Y.~Yang.
	\newblock Algebraic structures in two-dimensional conformal field theory.
	\newblock In R.~Szabo and M.~Bojowald, editors, {\em Encyclopedia of
		Mathematical Physics (Second Edition)}, pages 604--617. Elsevier, 2025.
	
	\bibitem[Gia11]{giansiracusa}
	J.~Giansiracusa.
	\newblock {The framed little 2-discs operad and diffeomorphisms of
		handlebodies}.
	\newblock {\em J. Top.}, 4(4):919--941, 2011.
	
	\bibitem[GJS23]{skeinfin}
	S.~Gunningham, D.~Jordan, and P.~Safronov.
	\newblock {The finiteness conjecture for skein modules}.
	\newblock {\em Invent. Math.}, 232:301--363, 2023.
	
	\bibitem[GK95]{gk}
	E.~Getzler and M.~Kapranov.
	\newblock Cyclic operads and cyclic homology.
	\newblock In R.~Bott and S.-T. Yau, editors, {\em Conference proceedings and
		lecture notes in geometry and topology}, pages 167--201. Int. Press, 1995.
	
	\bibitem[GK98]{gkmod}
	E.~Getzler and M.~Kapranov.
	\newblock {Modular operads}.
	\newblock {\em Compositio Math.}, 110:65--125, 1998.
	
	\bibitem[GKPM11]{mtrace1}
	N.~Geer, J.~Kujawa, and B.~Patureau-Mirand.
	\newblock {Generalized trace and modified dimension functions on ribbon
		categories}.
	\newblock {\em Selecta Math. New Ser.}, 17(2):453--504, 2011.
	
	\bibitem[GKPM13]{mtrace2}
	N.~Geer, J.~Kujawa, and B.~Patureau-Mirand.
	\newblock {Ambidextrous objects and trace functions for nonsemisimple
		categories}.
	\newblock {\em Proc. Am. Math. Soc.}, 141(9):2963--2978, 2013.
	
	\bibitem[GKPM22]{mtrace}
	N.~Geer, J.~Kujawa, and B.~Patureau-Mirand.
	\newblock {M-traces in (non-unimodular) pivotal categories}.
	\newblock {\em Algebr. Rep.. Theory}, 25:759--776, 2022.
	
	\bibitem[GPMT09]{geerpmturaev}
	N.~Geer, B.~Patureau-Mirand, and V.~Turaev.
	\newblock {Modified quantum dimensions and re-normalized link invariants}.
	\newblock {\em Compositio Math.}, 145(1):196--212, 2009.
	
	\bibitem[GPMV13]{mtrace3}
	N.~Geer, B.~Patureau-Mirand, and A.~Virelizier.
	\newblock {Traces on ideals in pivotal categories}.
	\newblock {\em Quantum Topol.}, 4(1):91--124, 2013.
	
	\bibitem[GRW09]{grw}
	M.~Gaberdiel, I.~Runkel, and S.~Wood.
	\newblock {Fusion rules and boundary conditions in the $c = 0$ triplet model}.
	\newblock {\em J. Phys. A}, 42(32):325--403, 2009.
	
	\bibitem[HLZ11]{hlzi}
	Yi-Zhi Huang, James Lepowsky, and Lin Zhang.
	\newblock Logarithmic tensor category theory, {VIII}: Braided tensor category
	structure on categories of generalized modules for a conformal vertex
	algebra.
	\newblock arXiv:1110.1931 [math.QA], 2011.
	
	\bibitem[HM06]{Halvorson2006-HALAQF}
	H.~Halvorson and M.~Müger.
	\newblock Algebraic quantum field theory.
	\newblock In J.~Butterfield and J.~Earman, editors, {\em Handbook of the
		philosophy of physics}. Kluwer Academic Publishers, 2006.
	
	\bibitem[HP92]{hp92}
	J.~Hoste and J.~Przytycki.
	\newblock A survey of skein modules of 3-manifolds.
	\newblock In A.~Kawauchi, editor, {\em Proceedings of the International
		Conference on Knot Theory and Related Topics}, pages 363--379. De Gruyter,
	1992.
	
	\bibitem[Iva12]{ivanov}
	S.~O. Ivanov.
	\newblock Nakayama functors and {E}ilenberg-{W}atts theorems.
	\newblock {\em J. Math. Sci.}, 183:675--680, 2012.
	
	\bibitem[JF21]{jfheisenberg}
	T.~Johnson-Freyd.
	\newblock Heisenberg-picture quantum field theory.
	\newblock In A.~Alekseev, E.~Frenkel, M.~Rosso, B.~Webster, and M.~Yakimov,
	editors, {\em Progress in Math. vol 340}, pages 371--409. Birkhäuser, 2021.
	
	\bibitem[Kas15]{kassel}
	C.~Kassel.
	\newblock {\em Quantum Groups}, volume 155 of {\em Graduate Texts in Math.}
	\newblock Springer, 2015.
	
	\bibitem[KL01]{kl}
	T.~Kerler and V.~V. Lyubashenko.
	\newblock {\em Non-Semisimple Topological Quantum Field Theories for
		3-Manifolds with Corners}, volume 1765 of {\em Lecture Notes in Math.}
	\newblock Springer, 2001.
	
	\bibitem[KT22]{kirillovtham}
	A.~Kirillov and Y.~H. Tham.
	\newblock {Factorization homology and 4D {TQFT}}.
	\newblock {\em Quantum Top-}, 13(1):1--54, 2022.
	
	\bibitem[Lur]{higheralgebra}
	J.~Lurie.
	\newblock Higher algebra.
	\newblock Available at https://www.math.ias.edu/\~{}lurie/papers/HA.pdf.
	
	\bibitem[Lyu95a]{lyubacmp}
	V.~V. Lyubashenko.
	\newblock {Invariants of 3-manifolds and projective representations of mapping
		class groups via quantum groups at roots of unity}.
	\newblock {\em Comm. Math. Phys.}, 172:467--516, 1995.
	
	\bibitem[Lyu95b]{lyu}
	V.~V. Lyubashenko.
	\newblock {Modular transformations for tensor categories}.
	\newblock {\em J. Pure Appl. Alg.}, 98:279--327, 1995.
	
	\bibitem[Lyu96]{lyulex}
	V.~V. Lyubashenko.
	\newblock {Ribbon abelian categories as modular categories}.
	\newblock {\em J. Knot Theory and its Ramif.}, 05(03):311--403, 1996.
	
	\bibitem[Lyu99]{lyubook}
	V.~V. Lyubashenko.
	\newblock {\em Squared Hopf algebras}, volume 677.
	\newblock Amer. Math. Soc., 1999.
	
	\bibitem[May72]{mayoperad}
	P.~May.
	\newblock {\em The Geometry of Iterated Loop Spaces}, volume 217 of {\em
		Lecture Notes in Math.}
	\newblock Springer, 1972.
	
	\bibitem[McR21]{mcrae}
	R.~McRae.
	\newblock On rationality for {$C_2$}-cofinite vertex operator algebras.
	\newblock arXiv:2108.01898 [math.QA], 2021.
	
	\bibitem[MR95]{masbaumroberts}
	G.~Masbaum and J.~Roberts.
	\newblock {On central extensions of mapping class groups}.
	\newblock {\em Math. Ann.}, 302:131--150, 1995.
	
	\bibitem[MS89]{ms89}
	G.~Moore and N.~Seiberg.
	\newblock {Classical and Quantum Conformal Field Theory}.
	\newblock {\em Comm. Math. Phys.}, 123(2):177--254, 1989.
	
	\bibitem[MSWY23]{sn}
	L.~Müller, C.~Schweigert, L.~Woike, and Y.~Yang.
	\newblock The {L}yubashenko modular functor for {D}rinfeld centers via
	non-semisimple string-nets.
	\newblock arXiv:2312.140109 [math.QA], 2023.
	
	\bibitem[MW23a]{cyclic}
	L.~Müller and L.~Woike.
	\newblock {Cyclic framed little disks algebras, {G}rothendieck-{V}erdier
		duality and handlebody group representations}.
	\newblock {\em Quart. J. Math.}, 74(1):163--245, 2023.
	
	\bibitem[MW23b]{mwdehn}
	L.~Müller and L.~Woike.
	\newblock The {D}ehn twist action for quantum representations of mapping class
	groups.
	\newblock Accepted for publication in \emph{J. Top.} arXiv:2311.16020
	[math.QA], 2023.
	
	\bibitem[MW24a]{mwskein}
	L.~Müller and L.~Woike.
	\newblock Admissible skein modules and ansular functors: A comparison.
	\newblock arXiv:2409.17047 [math.QA], 2024.
	
	\bibitem[MW24b]{mwansular}
	L.~Müller and L.~Woike.
	\newblock {Classification of consistent systems of handlebody group
		representations}.
	\newblock {\em Int. Math. Res. Not.}, 2024(6):4767--4803, 2024.
	
	\bibitem[MW24c]{mwcenter}
	L.~Müller and L.~Woike.
	\newblock The distinguished invertible object as ribbon dualizing object in the
	{D}rinfeld center.
	\newblock {\em Selecta Math. New Ser.}, 30(98), 2024.
	
	\bibitem[Ost97]{ostrik}
	V.~Ostrik.
	\newblock {Decomposition of the adjoint representation of the small quantum
		$sl_2$}.
	\newblock {\em Comm. Math. Phys.}, 186:253--264, 1997.
	
	\bibitem[Prz99]{Przytycki}
	J.~Przytycki.
	\newblock {Fundamentals of Kauffman bracket skein modules}.
	\newblock {\em Kobe J. Math.}, 16:45--66, 1999.
	
	\bibitem[Rob94]{roberts}
	J.~Roberts.
	\newblock {Skeins and mapping class groups}.
	\newblock {\em Math. Proc. Camb. Phil. Soc.}, 115:53--77, 1994.
	
	\bibitem[RST24]{rst}
	I.~Runkel, C.~Schweigert, and Y.~H. Tham.
	\newblock Excision for spaces of admissible skeins.
	\newblock arXiv:2407.09302 [math.QA], 2024.
	
	\bibitem[RT90]{rt1}
	N.~Reshetikhin and V.~G. Turaev.
	\newblock {Ribbon graphs and their invariants derived from quantum groups}.
	\newblock {\em Comm. Math. Phys.}, 127:1--26, 1990.
	
	\bibitem[RT91]{rt2}
	N.~Reshetikhin and V.~Turaev.
	\newblock {Invariants of 3-manifolds via link polynomials and quantum groups}.
	\newblock {\em Invent. Math.}, 103:547--598, 1991.
	
	\bibitem[Seg88]{Segal}
	G.~Segal.
	\newblock {Two-dimensional conformal field theories and modular functors}.
	\newblock In {\em {IX International Conference on Mathematical Physics
			(IAMP)}}, 1988.
	
	\bibitem[Shi17]{shimizuunimodular}
	K.~Shimizu.
	\newblock {On unimodular finite tensor categories}.
	\newblock {\em Int. Math. Res. Not.}, 2017(1):277--322, 2017.
	
	\bibitem[Shi19]{shimizumodular}
	K.~Shimizu.
	\newblock {Non-degeneracy conditions for braided finite tensor categories}.
	\newblock {\em Adv. Math.}, 355:106778, 2019.
	
	\bibitem[Shi23]{shimizupiv}
	K.~Shimizu.
	\newblock Pivotal structures of the drinfeld center of a finite tensor
	category.
	\newblock {\em J. Pure App. Alg.}, 227(7):107321, 2023.
	
	\bibitem[SP09]{schommerpries}
	C.~J. Schommer-Pries.
	\newblock {\em The classification of two-dimensional extended topological field
		theories}.
	\newblock PhD thesis, Berkeley, 2009.
	
	\bibitem[SS23]{shibatashimizu}
	T.~Shibata and K.~Shimizu.
	\newblock {Modified Traces and the Nakayama Functor}.
	\newblock {\em Alg. Rep. Theory}, 26:513--551, 2023.
	
	\bibitem[Ste25]{steinebrunner}
	J.~Steinebrunner.
	\newblock 2-dimensional {TFT}s via modular $\infty$-operads.
	\newblock arXiv:2506.22104 [math.CT], 2025.
	
	\bibitem[SW03]{salvatorewahl}
	P.~Salvatore and N.~Wahl.
	\newblock {Framed discs operads and {B}atalin-{V}ilkovisky algebras}.
	\newblock {\em Quart. J. Math.}, 54:213--231, 2003.
	
	\bibitem[SW21]{dva}
	C.~Schweigert and L.~Woike.
	\newblock {The Hochschild Complex of a Finite Tensor Category}.
	\newblock {\em Alg. Geom. Top.}, 21(7):3689--3734, 2021.
	
	\bibitem[SW23]{tracesw}
	C.~Schweigert and L.~Woike.
	\newblock The trace field theory of a finite tensor category.
	\newblock {\em Alg. Rep. Theory}, 26:1931--1949, 2023.
	
	\bibitem[Til98]{tillmann}
	U.~Tillmann.
	\newblock {$\mathcal{S}$-Structures for $k$-Linear Categories and the
		Definition of a Modular Functor}.
	\newblock {\em J. London Math. Soc.}, 58(1):208--228, 1998.
	
	\bibitem[Tur90]{turaevck}
	V.~Turaev.
	\newblock {Conway and Kauffman modules of a solid torus}.
	\newblock {\em J. Math. Sci.}, 52:2799--2805, 1990.
	\newblock Translated from Zapiski Nauchnykh Seminarov Leningradskogo Otdeleniya
	Matematicheskogo Instituta im. V. A. Steklova AN SSSR, Vol. 167, pp. 79–89,
	1988.
	
	\bibitem[Tur91]{turaevskein}
	V.~Turaev.
	\newblock {Skein quantization of Poisson algebras of loops on surfaces}.
	\newblock {\em Ann. Sci. Sc. Norm. Sup}, 24(6):635--704, 1991.
	
	\bibitem[Tur94]{turaev}
	V.~G. Turaev.
	\newblock {\em Quantum Invariants of Knots and 3-Manifolds}, volume~18 of {\em
		Studies in Math.}
	\newblock De Gruyter, 1994.
	
	\bibitem[Wah01]{WahlThesis}
	N.~Wahl.
	\newblock {\em Ribbon braids and related operads}.
	\newblock PhD thesis, Oxford, 2001.
	
	\bibitem[Wal]{Walker}
	K.~Walker.
	\newblock {TQFT}s.
	\newblock Notes available at \url{http://canyon23.net/math/tc.pdf}.
	
	\bibitem[Woi24]{microcosm}
	L.~Woike.
	\newblock The cyclic and modular microcosm principle.
	\newblock Accepted for publication in \emph{Canad. J. Math.} arXiv:2408.02644
	[math.QA], 2024.
	
\end{thebibliography}
\end{document}